\documentclass{amsart}

\usepackage{graphicx}
\usepackage{amssymb}

\usepackage{longtable}
\usepackage{array}
\usepackage{multirow}

\usepackage[pdftex, colorlinks=true, linkcolor=blue, urlcolor=blue]{hyperref}

\usepackage{xcolor}
\usepackage{hyperref}

\usepackage{amsmath,amsthm, amssymb,amsbsy,amsfonts,latexsym,amsopn,amstext, amsxtra,euscript,amscd, graphicx, amsmath,amscd, times,  hyperref}

\usepackage[all,poly]{xy}

\usepackage[dvips]{rotating} 
\usepackage{enumerate}

\newtheorem{thm}{Theorem}
\newtheorem{prop}{Proposition}
\newtheorem{lem}{Lemma}
\newtheorem{defn} {Definition}
\newtheorem{exa}{Example}
\newtheorem{exe}{Exercise}
\newtheorem{rem} {Remark}
\newtheorem{cor}{Corollary}
\newtheorem{prob}{Problem}

\theoremstyle{definition}
\newtheorem{defi}{Definition}
\newtheorem{question}{Question}

\def\ss{$\mathfrak s$}

\newcommand\Gal{\mathop{\rm Gal}\nolimits}
\newcommand\Sym{\mathop{\rm Sym}\nolimits}
\newcommand\Spec{\mathop{\rm Spec}\nolimits}

\renewcommand\O{\mathcal{O}}
\newcommand\ord{\mathop{\rm ord}\nolimits}

\newcommand\lcm{\mathop{\rm lcm}\nolimits}
\newcommand{\fix}{\operatorname{Fix}}
\newcommand{\Div}{\operatorname{Div}}
\newcommand{\Pic}{\operatorname{Pic}}

\renewcommand{\div}{\operatorname{div}}

\newcommand{\Aut}{\operatorname{Aut}}

\DeclareMathOperator\Hom{\mbox{Hom }}
\DeclareMathOperator\End{\mbox{End }}

\DeclareMathOperator\Img{\mbox{Img }}

\DeclareMathOperator\mult{mult}

\newcommand\N{\mathbb N}
\newcommand\Z{\mathbb Z}
\newcommand\Q{\mathbb Q}
\newcommand\R{\mathbb R}
\newcommand\C{\mathbb C}

\newcommand\w{\mathfrak w}

\newcommand\m{\mathfrak m}
\newcommand\an{\mathfrak a}
\newcommand\bn{\mathfrak b}
\newcommand\en{\mathfrak e}
\newcommand\hn{\mathfrak h}

\newcommand\p{\mathfrak p}
\newcommand\f{{\mathfrak a}}

\newcommand\g{\mathfrak c}

\newcommand\HS{\mathfrak H}

\newcommand\embd{\hookrightarrow}

\newcommand\G{\bar{G}}
\newcommand\bG{\bar G}

\newcommand\e{\varepsilon}
\newcommand\normal{\triangleleft }

\newcommand\om{\omega}

\newcommand\s{\sigma}

\newcommand\J{\mathcal J}
\def\P{\mathbb P}
\def\H{\mathcal H}
\newcommand\M{\mathcal M}

\newcommand\F{\mathcal F}
\def\O{\mathfrak O}

\newcommand\T{\theta}

\newcommand\D{\Delta}
\newcommand\cR{\mathcal R}

\newcommand\<{\langle}
\def\>{\rangle}

\def\a{\alpha}
\def\b{\beta}
\def\g{\gamma}
\def\d{{\delta }}
 \def\e{\epsilon}

\def\s{\sigma}

\def\v{\mathfrak v} 
\def\u{\mathfrak u}  

\def\sn{\mbox{sn }}
\def\cn{\mbox{cn }}
\def\dn{\mbox{dn }}

\def\T{\mathfrak T}
 
\def\B{\mathcal B}
\def\E{\mathcal E}

\def\bC{\mathbf{C}}

\def\Z{\mathbb{Z}}
\def\N{\mathbb N}
\def\Z{\mathbb Z}
\def\Q{\mathbb Q}
\def\R{\mathbb R}
\def\C{\mathbb C}
\def\P{\mathbb P}
\def\F{\mathbb F}
\def\B{\mathcal B}
\def\L{\mathcal L}
\def\B{\mathfrak B}
\def\Y{\textbf{Y}}
\def\B{\textbf{B}}
\def\I{\textbf{I}}
\def\U{{\mathcal U}}
\def\W{{\mathcal W}}
\def\D{\Delta}

\def\O{\mathcal O}
\def\B{\mathcal B}
\def\E{\mathcal E}
\def\B{\mathfrak B}

\def\G{G}

\def\M{\mathcal M}

\def\cH{{\mathcal H}}
\def\fa{{\frak a}}

\def\X{\mathcal X}

\def\A{\mathbb A}
\def\T{\mathcal T}
\def\O{\mathcal O}
\def\m{\mathfrak m}

\def\a{\alpha}          
\def\b{\beta}
\def\g{\gamma}
\def\s{\sigma}

\def\d{\delta}
\def\t{\tau}
\def\l{\lambda}

\def\th{\theta}

\def\w{\mathbf w}
\def\0{\mathbf 0}
 
\def\u{\mathbf u}
\def\v{\mathbf v}
\def\e{\varepsilon} 

\def\p{\mathfrak p}
\def\m{\mathfrak m}

\def\chara{ \mbox{char }}

\def\det{\mbox{det }}       

\def\normal{\triangleleft }     

\def\<{\langle}
\def\>{\rangle}

\def\iso{{\, \cong\, }}

\def\min{\mbox{ \textit{min} }}

\def\fa{{\mathfrak a}}       

\def\l{\lambda}              

\def\lcm{\mbox{lcm }}
\def\det{\mbox{det }}

\newcommand{\ch}[2]
{\begin{bmatrix}
 #1 \\
 #2\\
\end{bmatrix}}

\newcommand{\chr}[4]
{\begin{bmatrix}
 #1 & #2\\
 #3 & #4
\end{bmatrix}}

\newcommand{\chs}[6]
{\begin{bmatrix}
 #1 & #2 & #3\\
 #4 & #5 & #6
\end{bmatrix}}

\def\o{\overline}
\def\deg{\mbox{ deg }}
\def\min{\mbox{ \textit{min} }}

\def\<{\langle}
\def\>{\rangle}

\def\F{\mathbb F}

\def\Img{\mbox{Img }}

\def\lcm{\mbox{lcm }}

\def\o{\overline}
\def\deg{\mbox{ deg }}

\def\H{\mathcal H}
\def\L{\mathcal L}

\def\Y{\mathfrak Y}
\def\v{\mathfrak v}
\def\u{\mathfrak u}
\def\w{\mathfrak w}

\def\a{\alpha}
\def\t{\tau}

\def\d{{\delta }}

\def\Jac{\mbox{Jac }}
\def\mod{\mbox{ mod }}
\def\deg{\mbox{deg }}

\def\embd{\hookrightarrow}

\def\p{\mathfrak p}
\def\O{\mathcal O}

\def\div{\mbox{div}}

\def\Z{\mathbb Z}

\def\Q{\mathbb Q}
\def\C{\mathbb C}

\def\cA{\mathcal A}
\def\L{\mathcal L}
\def\cR{\mathcal R}
\def\H{\mathcal H}
\def\M{\mathcal M}
\def\N{\mathcal N}
\def\w{\widetilde}
\def\l{\lambda}
\def\s{\sigma}
\def\a{\alpha}
\def\b{\beta}
\def\p{\mathfrak p}
\def\e{\varepsilon}
\def\iso{\equiv}

\def\bG{\overline G}
\def\g{\gamma}

\def\u{\mathfrak u}

\def\iso{{\, \cong\, }}

\def\<{\langle}
\def\>{\rangle}

\def\Y{\mathcal Y}

\def\normal{\triangleleft}
\def\D{\Delta}

\def\s{s}
\def\t{t}
\def\r{r}

\def\g{\gamma}
\def\bAut{\overline {\mathrm{Aut}}}

\def\e{\varepsilon}

\def\bC{\mathbf{C}}

\def\a{\alpha}
\def\b{\beta}
\def\g{\gamma}
\def\d{\delta}

\usepackage{fancyhdr} 
\usepackage{longtable}

\usepackage{hyperref}

\hypersetup{
  colorlinks   = true, 
  urlcolor     = blue, 
  linkcolor    = blue, 
  citecolor   = red 
}

\usepackage[sorted, msc-links]{amsrefs}

\usepackage{xy}
\input xy \xyoption{all}

\def\F{\mathbb F}
\def\iso{\cong}
\def\e{\zeta}
\def\G{\overline G}

\def\M{\mathcal M}

\def\PP{\mathcal P} 
\DeclareMathOperator\diff{Diff}
\DeclareMathOperator\degrm{\mathrm{deg}}

\def\ch{\mbox{\rm{char }}} 
\def\K{{\mathcal K}}          

\def\IA{\mathfrak A}
\def\IB{\mathfrak  B}
\def\IC{\mathfrak C}
\def\ID{\mathfrak  D}

\usepackage[capitalise]{cleveref}
\crefname{thm}{Thm.}{}
\crefname{prop}{Prop.}{}
\crefname{lem}{Lem.}{}
\crefname{cor}{Cor.}{}
\crefname{prob}{Problem}{}
\crefname{figure}{Fig.}{}
\crefname{equation}{Eq.}{}

\def\ff{\phi}
\def\<{\langle}
\def\>{\rangle}
\def\iso{\cong}
\def\w{\omega}
\def\nf{\Psi}
\def\df{\Upsilon}
\def\o{\otimes}

\def\X{\mathcal{C}}

\DeclareMathOperator\dv{div }
\DeclareMathOperator\psl{\mathrm{PSL}}
\DeclareMathOperator\pgl{\mathrm{PGL}}

\def\div{\mathfrak d}
\def\ddiv{\bar \div}

\def\wP{\mathbb\W}
\DeclareMathOperator\wh{\mathfrak{h}}   
\DeclareMathOperator\awh{\mathfrak{\tilde h} }   
\DeclareMathOperator\wt{wt}
\DeclareMathOperator\Proj{Proj}

\DeclareMathOperator\wgcd{wgcd \, }

\def\val{\mathbf{val}}
\def\I{\mathcal I}

\def\dis{\mathfrak D}
\def\N{\mathbb N}

\def\bs{\bar \sigma}

\issueinfo{13}{01}{}{2019}
\copyrightinfo{2019}{\href{http://albanian-j-math.com}{Albanian Journal of Mathematics}}
\def\issn{{\sc \textbf{ISSN: }} 1930-1235; }
\def\issueyear{\textbf{2019}}
\PII{\issn  (\issueyear)}

\pagespan{107}{200}

\title{From hyperelliptic to  superelliptic curves}

\author{A. Malmendier}
\address{Department of Mathematics and Statistics, \\    Utah State University,  Logan, UT 84322.}
\email{andreas.malmendier@usu.edu}

\author{T. Shaska}  
\address{
Department of Mathematics and Statistics\\  Oakland University, Rochester, MI 48309.}
\email{shaska@oakland.edu}

\begin{document}

\maketitle

\begin{center}
\emph{Dedicated to the memory of  Kay Magaard}
\end{center}

\vspace{.4cm}

\vskip3truept\hrule   

\hrule

\begin{abstract}
The theory of elliptic and hyperelliptic curves has been of crucial importance in the development of algebraic geometry. Almost all fundamental ideas were first obtained and generalized from computations and constructions carried out for elliptic or hyperelliptic curves. 

In this long survey, we show that this theory can be extended naturally to all superelliptic curves. We focus on automorphism groups, stratification of the moduli space $\M_g$, binary forms, invariants of curves, weighted projective spaces, minimal models for superelliptic curves, field of moduli versus field of definition, theta functions, Jacobian varieties, addition law in the Jacobian,  isogenies among Jacobians, etc. Many recent developments on the theory of superelliptic curves are provided as well as many open problems. 

\end{abstract}

\vspace{.3cm}

\hrule

\vspace{.4cm}

\begin{small}


\end{small}

\vspace{.4cm}

\hrule


\tableofcontents

\section{Introduction}\label{sect-1}
The theory of elliptic and hyperelliptic curves has been of crucial importance in the development of algebraic geometry. Almost all fundamental ideas were first obtained and generalized from computations and constructions carried out for elliptic or hyperelliptic curves. Examples are elliptic or hyperelliptic integrals, theta functions, Thomae's formula, the concept of Jacobians, etc. Some of the classical literature on the subject \cites{Mu0, Mu1, Mu2, Mu3} as well as the seminal work of Jacobi focus almost entirely on hyperelliptic curves. 

So what is so special about a hyperelliptic curve?  To begin with, a generic curve in the hyperelliptic locus admits a cyclic Galois cover to the projective line. This cover, which is called the \textbf{hyperelliptic projection} is of  degree $n=2$   and its branch points determine the curve in question (up to isomorphism). Hence, studying hyperelliptic curves over algebraically closed fields amounts to studying degree two coverings of the projective line.  

A natural generalization of the above is to study degree $n\geq 2$ cyclic Galois covers.  This means that for a curve $\X$ with automorphism group $\Aut (\X)$ there is a cyclic subgroup $H=\< \tau \>$ normal in $\Aut(\X)$ such that the quotient $\X/H$ is isomorphic to $\P^1$.  Such curves $\X$ are called \textbf{superelliptic curves}. The automorphism $\tau$ is called the \textbf{superelliptic automorphism} of $\X$. 

The goal of this paper  is to focus on the natural generalization of the theory of hyperelliptic curves to superelliptic curves, to highlight the theories that can be extended 
and all the open problems that come with this generalization.   It is a long survey on results of the last two decades of both authors, their collaborators, and other researchers. 

There are  similarities among superelliptic and hyperelliptic curves, but also differences. The obvious similarities are  that such curves have affine equations (over an algebraically closed field of characteristic relatively prime to $n$) of the form  $y^n=f(x)$, the list of full automorphism groups of such curves can be determined,  in most cases their equations can be determined over their field of moduli, and most importantly the full machinery of classical invariant theory of binary forms can be used to determine their isomorphism classes. It is such theory that makes the study of the moduli space of such curves much more concrete than for general curves. More importantly the invariant theory connects the theory of superelliptic curves to the weighted projective spaces.  

In \cref{sect-2} we give some basic generalities of algebraic curves and their function fields.  Most of the material is basic and it can be found in most of the classic books on the subject; see \cite{stichtenoth}, \cite{Mu0}. Throughout most of this paper we will assume that our curves are smooth, irreducible, defined over an algebraically closed field $k$ of characteristic $p \geq 0$. Certain restrictions on the field of definition $k$ or the characteristic $p$ will be assumed on certain sections. 

In \cref{sect-3} we focus on Weierstrass points.  Weierstrass points are an important tool in studying the automorphisms groups of curves.  For hyperelliptic curves with equation $y^2=f(x)$, the projection of Weierstrass points are exactly the roots of $f(x)$.  In \cref{sect-5} we will show that such roots are also Weierstrass points of superelliptic curves with equation $y^n = f(x)$.   

In \cref{sect-4} we focus on full automorphism groups of curves. The theory of automorphisms is especially important for superelliptic curves since the very motivation of superelliptic curves comes from the existence the superelliptic automorphism.  The automorphism groups of all superelliptic curves over any characteristic are fully classified.  We give complete list of these groups based on results from \cite{Sa}.

In \cref{sect-5} we introduce superelliptic curves, which are a generalization of hyperelliptic curves. Such curves have a degree $n\geq 2$, cyclic Galois covering $\pi : \X_g to \P^1$. We denote the branched points of this cover by the roots of some polynomial $f(x)$ and show that the curve has equation $y^n=f(x)$. 
We determine the list of possible full automorphism group of a superelliptic curve $\X_g$ of genus $g\geq 2$. Furthermore, we study the Weierstrass points of superelliptic curves and show that they are projected to the roots of $f(x)$ as in the hyperelliptic case.  

In \cref{sect-6} we study the loci of superelliptic curves in the moduli space.  We briefly introduce the moduli space of curves $\M_{g_0, r}$  and its Deligne-Mumford compactification $\overline \M_{g_0, r}$.  Then we focus on points of the $\overline \M_{g_0, r}$ which correspond to curves with automorphisms. We discuss the inclusions between such loci and give the complete stratification of the moduli space for genii $g=3, 4$. 

In \cref{sect-7} is considered the following problem: for a group $G$ which occurs as an automorphism group of a genus $g\geq 2$ algebraic curve $\X$, determine an equation of $\X$. We discuss in detail how this is accomplished for superelliptic curves. 

In \cref{sect-8} are given the preliminaries of classical invariant theory of binary forms and in  \cref{sect-9} it is shown how such invariants describe a point in the weighted moduli space  $\wP_{\w}^n (k)$.  It is shown that this is a much more convenient approach to study superelliptic curves. Weighted greatest common divisor and weighted height are introduced in \cref{sect-9} in order to study the arithmetic properties of $\wP_{\w}^n (k)$; see  \cite{b-g-sh} for further details. 

In \cref{sect-10} we study minimal models of superelliptic curves when a moduli point is given. This is well known, due to work of Tate, for elliptic curves and Liu for genus two.  We describe briefly Tate's algorithm.  For superelliptic curves we we say that a curve has minimal model when it has a minimal moduli point as in \cite{super-min}. We give necessary and sufficient condition on the set of invariants of the curve  that the curve has a minimal model.  Moreover an algorithm is provided how to find such minimal model. 
In \cref{sect-11} is discussed when the field of moduli is a minimal field of definition for superelliptic curves. 

Theta functions of superelliptic curves are discussed in \cref{sect-12}.  We give a quick review of the theory of theta functions including the Thomae's formula for hyperelliptic curves.  It is a natural question to generalize such results for cyclic or superelliptic curves. To further investigate such interesting topic one should continue with \cite{book}. 

In \cref{sect-13} we study Jacobian varieties and briefly describe Mumford's representation of divisors and Cantor's algorithm for addition of points on a hyperelliptic Jacobian; see \cite{frey-shaska} for how this fact is used on hyperelliptic curve cryptography.  Whether this algorithm can be generalized to all superelliptic Jacobians is the main focus of \cref{sect-13}. 

In \cref{sect-14} we study the Jacobian varieties with complex multiplication.  Most of the efforts here have been on determining which curves with many automorphisms have complex multiplication. Hyperelliptic Jacobians with many automorphisms which have complex multiplication have been determined (see \cite{muller-pink}). We list all superelliptic curves with many automorphisms. From such list the ones with complex multiplication are determined in \cite{obus-shaska}.

While this paper is for the most part a survey, it also includes many new results and recent developments.  It lays out a general approach of using cyclic coverings in the study of algebraic curves.  One can  attempt to further generalize the theory to more general coverings.   The  beginnings  of this program start with \cite{kyoto}. Most of the data for the list of groups, inclusion among the loci were obtained by K. Magaard.   We dedicate this paper to his memory.  

\medskip

\noindent \textbf{Acknowledgments:}  Authors want to thank Mike Fried for helpful suggestions and conversations during the process that this paper was written. 
 
\part{Curves and hyperelliptic curves}

\section{Algebraic curves and their function fields}\label{sect-2}

We assume that the reader is familiar with the basic definitions of field extensions.  This section is intended to establish the notation used throughout the rest of the paper, rather than as a comprehensive introduction to algebraic curves. Throughout $k$ is a perfect field. For more details, the reader is encouraged to consult \cite{stichtenoth} or \cite{frey-shaska} among other places.  Let us establish some notation and basic facts about algebraic curves and their function fields.

\subsection{Algebraic curves}
The following definitions are easily extended to  any algebraic variety, but we will focus on the curve case. Let $k$ be a perfect field and $\X$ an algebraic curve defined over $k$.   Then there is a homogeneous ideal  $I_\X \subset k[X_0, X_1 , \ldots ,X_n]$ defining $\X$,  and the curve $\X$ is irreducible if and only if $I_\X$ is a \textbf{prime} ideal in  $k[X_0, X_1 ,\ldots, X_n]$.  The (homogenous) coordinate ring of $\X$ is $\Gamma_h (\X) := k[X_0, X_1 , \ldots , X_n]/I_\X$,  which is an integral domain. The function field of $\X$ is the quotient field of $\Gamma_h (\X)$ and denoted by $k(\X)$. Since $\X$ is an algebraic variety of dimension one, the field $k(\X)$ is an algebraic function field of one variable.

Let $P=(a_0, a_1 ,\ldots , a_n)\in \X$. The ring
\[ \O_P (\X) = \{ f\in k(\X) \, | \, f \text{ is defined at } P \} \subset k(\X) \]
is a local ring with maximal ideal
\[ M_P (\X) = \{ f \in \O_P (\X) \, | \, f(P) = 0 \}   . \]
The point $P \in \X$ is a \textbf{non-singular point} if the local ring $\O_P (\X)$ is a discrete valuation ring. There is a 1-1 correspondence between points $P \in \X$ and the places of $k(\X)/k$, given by $P \mapsto M_P (\X)$.
This correspondence makes it possible to translate definitions from algebraic function fields to algebraic curves and vice-versa.

\subsection{Algebraic extensions of function fields}\label{subsec-algextfuncfield}
An algebraic function field $F/k$ of one variable over $k$ is a finite algebraic extension of $k(x)$ for some $x\in F$ which is transcendental over $k$. A \textbf{place} $\p$ of the function field $F/k$ is the maximal ideal for some valuation ring $\O$ of $F/k$.  We will denote by $\PP_F$ the set of all places of $F/k$.  Equivalently $\Sigma_\X(k)$ will denote the set of $k$-points of $\X$.

An algebraic function field $F^\prime / k^\prime$ is called an algebraic extension of $F/k$ if $F^\prime$ is an algebraic extension of $F$ and $k \subset k^\prime$.

A place $\p^\prime \in \PP_{F^\prime}$ is said to \textbf{lie over} $\p \in \PP_F$ if $\p \subset \p^\prime$. We write $\p^\prime | \p$.  In this case there exists an integer $e \geq 1$ such that $v_{\p^\prime} (x) = e \cdot v_\p (x)$, for all $x\in F$.  This integer is denoted by $e (\p^\prime  | \p) :=e$ and is called the \textbf{ramification index} of $\p^\prime$ over $\p$.
We say that $\p^\prime | \p$ is \textbf{ramified} when $e (\p^\prime | \p) > 1$ and otherwise \textbf{unramified}.

For any place $\p \in \PP_F$ denote by $F_\p:= \O/\p$.  The integer $f( \p^\prime | \p) := [F^\prime_{\p^\prime} : F_\p]$ is called the \textbf{relative degree} of $\p^\prime | \p$.

\begin{thm}
\label{th-fundeq}
Let $F^\prime/k^\prime$ be a finite extension of $F/k$ and $\p$ a place of $F/k$.  Let $\p_1, \dots , \p_m$ be all the places in $F^\prime/k^\prime$ lying over $\p$ and $e_i:= e(\p_i | \p)$ and $f_i :=  f(\p_i | \p)$ the relative degree of $\p_i | \p$. Then
\[ \sum_{i=1}^m e_i f_i = [F^\prime : F].\]
\end{thm}


For a place $\p \in \PP_F$ let $\O_\p^\prime$ be the integral closure of $\O_\p$ in $F^\prime$.  The complementary module over $\O_\p$ is given by $t \cdot \O_p^\prime$. Then for $\p^\prime | \p$ we define the \textbf{different exponent} of $\p^\prime$ over $\p$ as
\[ d (\p^\prime | \p) := - v_{\p^\prime} (t). \]
%
The different exponent $d (\p^\prime | \p) $ is well-defined and $d (\p^\prime | \p) \geq 0$.  Moreover, we have $d (\p^\prime | \p) =0$ for almost all $\p \in \PP_F$.  The \textbf{different divisor} is defined as
\[ 
\diff (F^\prime / F) := \sum_{\p \in \PP_F} \sum_{\p^\prime | \p} d(\p^\prime | \p) \cdot \p^\prime.   
\]
The following well-known formula for ramified coverings between Riemann surfaces of genus $g'$ and $g$, respectively, can now be generalized to function fields as follows.

\begin{thm}
\label{th-HGF}
Let $F/k$ be an algebraic function field of genus $g$ and $F^\prime/F$ a finite separable extension. Let $k^\prime$ denote the constant field of $F^\prime$ and $g^\prime$ the genus of $F^\prime/k^\prime$. Then,
\begin{equation}\label{e1}
2 (g^\prime-1)   =   \frac {[F^\prime : F]}    {[k^\prime: k]}   (2g-2)   +  \degrm     \diff (F^\prime /F)
\end{equation}
\end{thm}

For a proof see \cite{stichtenoth}*{Thm. 3.4.13}.  A special case of the above is the following:

\begin{cor}
Let $F/k$ be a function field of genus $g$ and $x\in F\setminus k$ such that $F/k(x)$ is separable. Then,
\[ 2g-2 = -2 [F:k(x)] + \degrm \diff (F/k(x)) \]
\end{cor}

The ramification index and the different exponent are closely related, as made precise by the Dedekind theorem.

\begin{thm}[Dedekind Different Theorem] 
For all $\p^\prime|\p$ we have:

i) $d ( \p^\prime | \, \p) \geq e ( \p^\prime | \p) -1$.

ii) $d ( \p^\prime | \, \p) = e ( \p^\prime | \p) -1$ if and only if $e ( \p^\prime | \p) $ is not divisible by the $\ch k$.
\end{thm}

An extension $\p^\prime | \p$ is said to be \textbf{tamely} ramified if $e (\p^\prime | \p) > 1$ and $\ch k$ does not divide $e (\p^\prime | \p)$.  If $e (\p^\prime | \p) > 1$ and $\ch k$ does  divide $e (\p^\prime | \p)$ we say that $\p^\prime | \p$ is \textbf{wildly} ramified.

The extension $F^\prime / F$ is called \textbf{ramified} if there is at least one place $\p \in \PP_F$ which is ramified in $F^\prime / F$. The extension  $F^\prime / F$ is called \textbf{tame} if there is no place $\p \in \PP_F$ which is wildly ramified in $F^\prime/F$.

\begin{lem}
Let $F^\prime/F$ be a finite separable extension of algebraic function fields. Then

\begin{itemize}
\item[a)] $\p^\prime | \p$ is ramified if and only if $\p^\prime \leq \diff (F^\prime/F)$.  Moreover, if $\p^\prime/\p$ is ramified then:

\subitem i) $d ( \p^\prime | \p) = e ( \p^\prime | \p) -1$ if and only if $\p^\prime |\p$ is tamely ramified

\subitem ii) $d ( \p^\prime | \p) >  e ( \p^\prime | \p) -1$ if and only if $\p^\prime |\p$ is wildly ramified

\item[b)] Almost all places $\p \in \PP_F$ are unramified in $F^\prime /F$.
\end{itemize}
\end{lem}

From now on we will use the term  "curve" and its function field interchangeably, depending on the context.  It is more convenient to talk about function fields than curves in most cases.

\subsection{Divisors and the Riemann-Roch theorem}

For a given curve $\X$ defined over $k$, we call a divisor $D$ the formal finite sum 
\[  D=\sum_{\p\in \Sigma_\X(k)}   z_\p \, P.     \]
The set of all divisors of $\X$ is denoted by $\Div_\X (k)$. Moreover, the divisor $(f)$ of a function $f\in k(\X)$, defined as the finite linear combination of the set of all zeroes and poles of $f$, is called a \textbf{principal divisor}. Since $(f g) = (f) + (g)$, the set of principal divisors is a subgroup of the group of divisors. Two divisors that differ by a principal divisor are called \textbf{linearly equivalent}. The symbol $\degrm (D)$ denotes the \textbf{degree} of the divisor $D$, i.e., the sum of the coefficients occurring in $D$. It can be shown that the divisor of a global meromorphic function always has degree $0$, so the degree of the divisor depends only on the linear equivalence class. The \textbf{Picard group} $\Pic_\X(k)$ is the group of divisors modulo linear equivalence. 

\subsubsection{Riemann-Roch Spaces}
Define a partial ordering of elements in  $\Div_\X (k)$ as follows; $D$   is \textbf{effective} ($D \geq 0$)  if $ z_\p \geq 0$ for every  $\p$,   and $D_1\geq D_2$ if $D_1-D_2\geq 0$.
The \textbf{Riemann-Roch space} associated to $D$ is
\[\L(D)=\{f\in k(\X) \mbox{ with } (f)\geq -D\}\cup\{0\}.\]
Thus, the elements $x\in \L(D)$ are defined by the property that $w_\p(x)\geq -z_\p$ for all $\p\in \Sigma_\X(k)$. 
$\L(D)$ is a vector space over $k$ and can be interpreted as the space of functions $f\in k(\X)$ whose poles are bounded by $D$, and is often denoted by $\mathcal{O}_\X[D]$. It has positive dimension  if and only if there is a function $f\in k(\X)$ with $D+(f)\geq 0$, or equivalently,  $D\sim D_1$ with $D_1\geq 0$. 

Here are some facts:
$\L(0)=k$, and if $\degrm(D)<0$ then $\L(D)=\{0\}$.  If $\degrm(D)=0$ then either $D$ is a principal divisor or $\L(D)=\{0\}$.
%
%
\begin{prop}
Let $D=D_1-D_2$ with $D_i\geq 0$ for $i=1,2$. Then
\[ \dim(\L(D))\leq \degrm (D_1)+1. \]
\end{prop}
We also remark that for $D\sim D'$ we have $\L(D)\sim \L(D')$.   In particular $\L(D)$ is a finite-dimensional $k$-vector space. We follow traditional conventions and denote the dimension of $\L (D)$ by 
\begin{equation}
\ell(D):=\dim_k(\L(D)).
\end{equation}
Computing  $\ell(D)$ is a fundamental problem which is solved by the   Riemann-Roch Theorem. A first estimate is a generalization of the proposition above.

\begin{lem}
For all divisors $D$ we have the inequality
\[\ell(D)\leq \degrm(D)+1.\]
\end{lem}

For a proof one can assume that $\ell(D) > 0$ and so $D\sim D'>0$. 

\begin{thm}[Riemann's inequality] 
\label{riemann}
For given curve $\X$ there is a minimal number $g_\X\in \mathbb N   \cup \{0\}$ such that for   all $D\in \Div_\X$ we have
\[\ell(D)\geq \degrm(D)+1-g_\X.\]
\end{thm}

For a proof see    \cite{stichtenoth}*{Proposition 1.4.14}.   Therefore,
\[ 
g_\X=\max \{ \degrm{D}-\ell(D)+1;\,\, D\in \Div_\X(k) \} 
\]
exists and is a non-negative integer    independent of $D$. The integer  $g_\X$ is called the \textbf{genus} of $\X$.   The genus does not change under constant field extensions because we have assumed that $k$ is perfect.   This is not correct  in general if the constant field of $\X$ has inseparable algebraic extensions.   There is a corollary of the theorem.

\begin{cor}
There is a number $n_\X$ such that for all $D$ with $\degrm(D) > n_\X$ we get equality  $\ell(D) =\degrm (D)+1-g_\X$. 
\end{cor}
\cref{riemann} together with its corollary is the  "Riemann part" of the Riemann-Roch theorem for curves. 
To determine  $n_\X$ one needs more information about the inequality for small degrees and the concept of a canonical divisor. 

\subsubsection{Canonical Divisors}
Let $k(\X)$ be the function field of a curve $\X$ defined over $k$. To every $f\in k(\X)$ we attach a symbol $df$, the \textbf{differential}  of $f$. The $k(\X)$-vector space $\Omega(k(\X))$ is the vector space generated by symbols $df$ modulo the following relations:

For $f, g\in k(\X)$ and $\lambda\in k$ we have:

\begin{itemize}
\item[i)]  $d(\lambda  f+g)=\lambda df +dg$

\item[ii)]  $d(f\cdot g)=f dg+g df$.
\end{itemize}

The relation between derivations and differentials is given by the

\begin{defi}[Chain rule]
Let $x$ be as above and $f\in k(\X)$. Then $df = (\partial f/\partial x)  dx$.
\end{defi}

The $k(\X)$-vector space of differentials $\Omega(k(\X))$ has dimension $1$ and is generated by $dx$ for any $x\in k(\X)$ for which $k(\X)/k(x)$ is finite and separable. The space $\Omega(k(\X))$ is also called the vector space of \textbf{global meromorphic} one-forms on $\X$.

We use a well known fact from the theory of function fields $F$ in one variable:
Let $\p$ be a place of $F$, i.e. an equivalence class of discrete rank one valuations of $F$ trivial on $k$.  Then, there exist a function $t_\p\in F$ with $w_\p(t_\P)=1$ and $F/k(t_\p)$ separable.

We apply this fact to $F=k(\X)$. For all $\p\in \Sigma_\X(k)$ we choose a function $t_\p$ as above.   For a differential $0\neq \omega\in \Omega(k(\X))$  we obtain
$\omega=f_\p \cdot dt_\p$.
%
The divisor $(\omega)$ of a global meromorphic one-form  is given by
\[ (\omega):= \sum_{\p\in \Sigma_\p} w_\p(f_\p)\cdot \p  \,, \]
and  is a called a \textbf{canonical divisor}.
The coefficient function of $\omega$ is transformed by the chain rule, but two coefficient functions, before and after applying the chain rule, always define the same divisor locally. Therefore, we can define the divisor of $\omega$ by using the coefficient function in \textbf{any} local expression for $\omega$. Moreover, for any function $f \in k(\X)$ we have $(f\omega)=(f)+(\omega)$, and for any two non-zero differentials $\omega_1$ and $\omega_2$, there is always a function $f \in k(\X)$ such that $\omega_1=f\omega_2$, so that the two canonical divisors $(\omega_1)$ and $(\omega_2)$ are linearly equivalent. Therefore, the linear equivalence class of canonical divisors is well-defined; this is called called the \textbf{canonical class} of $\X$, and denoted by $\K_\X\in \Pic_\X(k)$.

We are now ready to state the Riemann-Roch Theorem.
\begin{thm}
Let $K$ be a canonical divisor of $\X$. For all $D\in \Div_\X(k)$ we have
\[ \ell(D)=\degrm(D)+1-g_\X+\ell(K-D).\]
\end{thm}
%

A differential $\omega$ is \textbf{holomorphic} if $(\omega)$ is an effective divisor.    The set of holomorphic differentials is a $k$-vector space denoted by $\Omega^1_\X$. If $K=(\omega)$ is a canonical divisor, and $f\in  \L(K)$ is a function with poles bounded by $K$, then $f \omega$ is holomorphic. This gives an isomorphism between $\L(K)=\mathcal{O}_\X[K]$ and $\Omega^1_\X$. If we take $D=0$ respectively $D=K$ in the theorem of Riemann-Roch we get the following:
\begin{cor}
$\Omega^1_\X$ is a $g_\X$-dimensional $k$-vector space  and $\deg(K)=2g_\X-2$.
\end{cor}
There are two further important consequences of the Riemann-Roch theorem.
\begin{cor}  The following are true:
\begin{enumerate}
\item If $\degrm(D) > 2 g_\X-2$ then $\ell(D)=\degrm(D)+1-g_\X.$

\item{In every divisor class of degree $g$ there is a positive divisor.}
\end{enumerate}
\end{cor}

\proof
Take $D$ with $\degrm(D) \geq 2g_\X -1$. Then $\degrm(W-D)\leq -1$ and therefore $\ell(W-D) =0$.  Take $D$ with $\degrm(D)=g_\X$. Then $\ell(D) =1+\ell(W-D)\geq 1$ and so there is a positive divisor in the class of $D$.
\qed



\section{Weierstrass points}\label{sect-3}
%
The material of this section can be found in every book on the subject.   We mostly refer to \cites{cornalba1, Farkas-Kra, shor-shaska-1, shor-shaska-2}. 

\subsection{Weierstrass points via linear systems}
Let $D$ be a divisor on $\X_g$.  The \textbf{complete linear system} of $D$, denoted $|D|$, is the set of all effective divisors $E\geq0$ that are linearly equivalent to $D$; that is, 
\[ |D|=\{E\in  \Div_\X(k) : E=D+(f)\text{ for some } f\in \L(D)\}.\]  
Note that any function $f\in k(\X)$ satisfying this definition will necessarily be in $\L(D)$ because $E\geq0$.  A complete linear system has a natural projective space structure which we denote $\P(\L(D))$.
Now, consider the projectivization $\P(\L(D))$ and the function 
\[S : \P(\L(D))\to|D| \,,\] 
which takes the span of a function $f\in\L(D)$ and maps it to $D+(f)$.

A \textbf{(general) linear system} is a subset $Q$ of a complete linear system $|D|$ which corresponds to a linear subspace of $\P(\L(D))$.  The \textbf{dimension} of a general linear system is its dimension as a projective vector space. Let $Q\subseteq |D|$ be a nonempty linear system on $\X_g$ with corresponding vector subspace $V\subseteq \L(D)$, and let $P\in\X_g$.  For any integer $n$, consider the vector space $V(-nP):=V \cap \L(D-nP),$ which consists of those functions in $\L(D)$ with order of vanishing at least $n$ at $P$.  This leads to a chain of nested subspaces \[V(-(n-1)P)\supseteq V(-nP)\] for all $n\in\Z$.  Since $\L(D-nP)=\{0\}$ for $n\geq\deg(D)$, this chain eventually terminates and becomes $\{0\}$.  As in  \cref{prop:dimension-increase-by-1}, which appears later, the dimension drops by at most 1 in each step.  We define gap numbers as follows.

\begin{defn}
An integer $n\geq 1$ is a \textbf{gap number} for $Q$ at $P$ if 
\[ V(-nP)=V(-(n-1)P)-1.\]
The set of gap numbers for $Q$ at $P$ is denoted $G_P(Q)$.
\end{defn}

Let $Q(-nP)$ denote the linear system corresponding to the vector space $V(-nP)$.  Then $Q(-nP)$ consists of divisors $D\in Q$ with $D\geq nP$.  An integer $n\geq 1$ is a gap number for $Q$ at $P$ if and only if 
\[\dim Q(-nP)=\dim Q(-(n-1)P)-1.\]  
A linear system $Q$ is denoted by $g_d^r$ if $\dim Q=r$ and $\deg Q=d$. For such a system, the sequence of gap numbers is a subset consisting of $r+1$ elements of $\{1, 2, \dots, d+1\}$.  If this sequence is anything other than $\{1, 2, \dots, r+1\}$, we call $P$ an \textbf{inflection point for the linear system $Q$}.  The terms linear system and linear series are completely interchangeable. 

Suppose the sequence of gap numbers is $\{n_1,n_2,\dots, n_{r+1}\}$, written in increasing order.  For each $n_i$, one can choose an element $f_i\in Q(-(n_i-1)P)\setminus Q(-n_iP)$.  Then, the vanishing order at $P$ is 
\[ \ord_P(f_i)=n_i-1-\ord_P(D),\]
 and because of the different orders of vanishing at $P$, these functions are linearly independent, so $\{f_1,f_2,\dots, f_{r+1}\}$ is a basis for $V$.  Such a basis is called an \textbf{inflectionary basis} for $V$ with respect to $P$.  

Conversely, given a basis for $V$, a change of coordinates can produce an inflectionary basis and hence construct the sequence of gap numbers.  Fix a local coordinate $z$ centered at $P$, and suppose $\{h_1,h_2,\dots, h_{r+1}\}$ is any basis for $V$.  Set $g_i=z^{\ord_P(D)}h_i$ for each $i$.  Then, the functions $g_i$ are holomorphic at $P$ and thus have Taylor expansions 
\[g_i(z) = g_i(0) + g_i'(0)z + \frac{g_i^{(2)}(0)}{2!}z^2+\cdots+\frac{g_i^{(r)}(0)}{r!}z^r+\cdots.\]  
We want to find linear combinations 
\[G_j(z)=\sum_{i=1}^{r+1} c_{i,j}g_i(z)\] 
of these functions to produce orders of vanishing from $0$ to $r$ at $P$.  This is possible precisely when the matrix
\[
\begin{bmatrix}
g_1(0) & g_1'(0) & g_1^{(2)}(0) & \cdots & g_1^{(r)}(0) \\ 
g_2(0) & g_2'(0) & g_2^{(2)}(0) & \cdots & g_2^{(r)}(0) \\ 
\vdots & \vdots & \vdots & \ddots & \vdots \\
g_{r+1}(0) & g_{r+1}'(0) & g_{r+1}^{(2)}(0) & \cdots & g_{r+1}^{(r)}(0)
\end{bmatrix}
\]
is invertible.  When that occurs, the same constants $c_{i,j}$ can be used to let $f_j=\sum_i c_{i,j}h_i$ and produce an inflectionary basis $\{f_j\}$ of $V$ such that $\ord_P(f_j)=j-1-\ord_P(D)$.  Thus, $G_P(Q)=\{1,2,\dots,r+1\}$ and so $P$ is an inflection point for $Q$.
\begin{defn}
The Wronskian of a set of functions $\{g_1,g_2,\dots,g_r\}$ of a variable $z$ is the function
\[
W(g_1,g_2,\dots,g_r) = 
\begin{vmatrix}
g_1(z) & g_1'(z) & g_1^{(2)}(z) & \cdots & g_1^{(r)}(z) \\ 
g_2(z) & g_2'(z) & g_2^{(2)}(z) & \cdots & g_2^{(r)}(z) \\ 
\vdots & \vdots & \vdots & \ddots & \vdots \\
g_{r+1}(z) & g_{r+1}'(z) & g_{r+1}^{(2)}(z) & \cdots & g_{r+1}^{(r)}(z)
\end{vmatrix}.
\]
\end{defn}

As with its use in differential equations, the Wronskian is identically zero if and only if the functions $g_1,\dots,g_r$ are linearly dependent.
We summarize with the following.

\begin{lem}
Let $\X_g$ be a curve with a divisor $D$ and $Q$ a linear system corresponding to a subspace $V\subseteq\L(D)$.  Let $\{f_1,\dots,f_{r+1}\}$ be a basis for $V$, and for each $i$, let $g_i=z^{\ord_P(D)}f_i$.  Let $P$ be a point with local coordinate $z$.
Then $P$ is an inflection point for $Q$ if and only if $W(g_1,\dots,g_{r+1})=0$ at $P$.
\end{lem}

\begin{cor}
For a fixed linear system $Q$, there are finitely many inflection points.
\end{cor}

\proof See \cite{Miranda}*{Lemma 4.4, Corollary 4.5}.
\qed


\begin{defn}
A \textbf{meromorphic $n$-fold differential} in the coordinate $z$ on an open set $V\subseteq\mathbb{C}$ is an expression $\mu$ of the form $\mu=f(z)(dz)^n$ where $f$ is a meromorphic function on $V$.
\end{defn}
Suppose $\omega_1,\dots,\omega_m$ are meromorphic $1$-fold differentials in $z$ where $\omega_i=f_i(z)dz$ for each $i$.  Then their product is defined locally as the meromorphic $m$-form $f_1\cdots f_m (dz)^m$.  With this, we consider the Wronskian.

\begin{lem}
Let $\X_g$ be an algebraic curve with meromorphic functions $g_1,\dots,g_m$.  Then $W(g_1,\dots,g_m)(dz)^{m(m-1)/2}$ defines a meromorphic $m(m-1)/2$-fold differential on $\X_g$.
\end{lem}

\proof Since each $g_i$ is meromorphic, the Wronskian is as well, and so this is clearly a meromorphic $m(m-1)/2$-fold differential locally.  What remains to be shown is that the local functions transform to each other under changes of coordinates;  see  \cite{Miranda}*{Lemma 4.9} for  details.\qed

From here on, let $W(g_1,\dots,g_m)$ denote this global meromorphic $m(m-1)/2$-fold differential.  We now investigate the poles of the Wronskian.
As with meromorphic functions and meromorphic $1$-forms, the order of vanishing of a meromorphic $n$-fold differential $f(z)(dz)^n$ is given by 
\[\ord_P(f(z)(dz)^n) = \ord_P(f(z)).\]  
Divisors are defined in a similar way; namely, 
\[(\mu) = \sum_{P}\ord_P(\mu) \, P.\]
With these definitions, we can consider spaces of meromorphic $n$-fold differentials whose poles are bounded by $D$.  So we  let 
\[\mathcal{L}^{(n)}(D)=\{\mu \text{ a meromorphic $n$-fold differential} : (\mu)\geq-D\},\]  
and for $n=0$ we recover the Riemann-Roch spaces encountered before, i.e., $\mathcal{L}^{(0)}(D)=\mathcal{L}(D)$.
Equivalently, for a local coordinate $z$, if $(dz)=K$, then 
\[L^{(n)}(D)=\{f(z)(dz)^n : f\in\mathcal{L}(D+nK)\}.\]

\begin{lem}\label{lem:bounded-pole-orders}
Let $D$ be a divisor on an algebraic curve $\X_g$.  Let $f_1,\dots,f_m$ be meromorphic functions in $\L(D)$.  Then the meromorphic $n$-fold differential $W(f_1,\dots,f_m)$ has poles bounded by $mD$.  That is, \[W(f_1,\dots,f_m)\in\L^{m(m-1)/2}(mD).\]
\end{lem}

\proof Fix a point $P$ with local coordinate $z$.  For each $i$, let $g_i=z^{\ord_P(D)}f_i$ so that the $g_i$'s are holomorphic at $P$.  Then the Wronskian $W(g_1,\dots,g_m)$ is holomorphic at $P$ as well.  Since the Wronskian is multilinear, \[W(z^{\ord_P(D)}f_1,\dots,z^{\ord_P(D)}f_m)=z^{m\cdot\ord_P(D)}W(f_1,\dots,f_m).\]  Since this is holomorphic at $P$, we have $\ord_P(W(f_1,\dots,f_m))\geq -mD$ as desired.
\qed

Suppose $\{f_1,\dots,f_{r+1}\}$ and $\{h_1,\dots,h_{r+1}\}$ are two bases for a subspace $V\subseteq\L(D)$ with corresponding linear system $Q\subseteq|D|$.  Consider the Wronskian of each basis.  Since we have a change of basis, given by a matrix that transforms from the basis given by the $f_i$'s to the one given by $h_j$'s, the Wronskian is scaled by the determinant of such a matrix which is a scalar and thus doesn't affect the zeroes or poles.  Therefore, the Wronskian is well-defined (up to a scalar multiple) by the linear system $Q$ rather than the choice of a basis.  We denote this Wronskian by $W(Q)$ and see that 
\[W(Q)\in\L^{(r(r+1)/2)}((r+1)D)\] 
by ~\cref{lem:bounded-pole-orders}.

\begin{prop}
For an algebraic curve $\X_g$ of genus $g$ with linear system $Q$ of dimension $r$, 
\[ \deg(W(Q))=r(r+1)(g-1).\]
\end{prop}

\proof The proof follows from the fact that $W(Q)$ is a meromorphic $r(r+1)/2$-fold differential of the form $f(z)(dz)^{r(r+1)/2}$ for some local coordinate $z$.  Since $f(z)$ is meromorphic, the degree of $(f(z))$ is zero.  And on a curve of genus $g$, the degree of $(dz)$ is $2g-2$.  Thus, the degree of $(f(z)(dz)^{r(r+1)/2})$ is 
$\dfrac{r(r+1)}{2}(2g-2)=r(r+1)g-1$.
\qed

We define the \textbf{inflectionary weight} of a point $P$ with respect to a linear system $Q$ to be \[w_P(Q)=\sum_{i=1}^{r+1}(n_i-i),\]  where $\{n_1,\dots,n_{r+1}\}$ is the sequence of gap numbers for $Q$ at $P$ written in ascending order.  It follows that $P$ is an inflection point for $Q$ precisely when $w_P(Q)>0.$  It turns out that the inflectionary weight of $P$ is exactly the order of vanishing of the Wronskian at $P$.

\begin{lem}
If $G_P(Q)=\{n_1,\dots,n_{r+1}\}$ and $\{f_1,\dots,f_{r+1}\}$ is a basis for $V$, then 
\[ w_P(Q)=\ord_P(W(z^{\ord_P(D)}f_1,\dots,z^{\ord_P(D)}f_{r+1})).\]
\end{lem}

\proof See \cite{Miranda}*{Lemma 4.14}.
\qed

\begin{thm}\label{thm:total-inf-weight}
For $\X_g$ an algebraic curve of genus $g$ with $Q$ a $g_d^r$ on $\X_g$, the total inflectionary weight on $\X_g$ is \[\sum_{P\in\X_g}w_P(Q)=(r+1)(d+rg-r).\]
\end{thm}

The canonical series is the complete linear system $|K|$ with $[K]=\K_\X$. By Riemann-Roch, $\dim |K|=g-1$ and $\deg K=2g-2$. Moreover, it is the only series on a curve of genus $g$ that has order $d=2g-1$ and dimension $r=g-1$.  Inflection points for this system are called \textbf{Weierstrass points}, and the \textbf{Weierstrass weight} of such a point is its inflectionary weight with respect to $K$.

\begin{cor}
The total Weierstrass weight on a curve of genus $g$ is 
\[ g^3-g=(g+1)g(g-1).\]
\end{cor}

\proof \cref{thm:total-inf-weight} with $d=2g-2$ and $r=g-1$.
\qed

For any $q\geq 1$, we use the linear system $qK$ to define $q$-Weierstrass points, which have $q$-Weierstrass weights.  For $q=1$, the results are above.  For $q=2$, $d=\deg qK=q(2g-2)$ and $r=\dim |qK|=(2q-1)(g-1)$.

\begin{cor}\label{cor:total-q-weight}
The total $q$-Weierstrass weight, for $q\geq2$, on a curve of genus $g$ is \[g(g-1)^2(2q-1)^2.\]
\end{cor}

\begin{rem}
There are $q$-Weierstrass points for any curve of genus $g>1$ and any $q\geq 1$.
\end{rem}

\subsection{Weierstrass points via gap numbers}

Let $P$ be a point on $\X_g$ and consider the vector spaces $\L(nP)$ for $n=0,1,\dots,2g-1$.  These vector spaces contains functions with poles only at $P$ up to a specific order.  This leads to a chain of inclusions \[ \L(0)\subseteq \L(P) \subseteq \L(2P) \subseteq \dots \subseteq \L((2g-1)P) \,,\] with a corresponding non-decreasing sequence of dimensions \[ \ell(0) \leq \ell(P) \leq \ell(2P) \leq \dots \leq \ell((2g-1)P).\]
The following proposition shows that the dimension goes up by at most 1 in each step.

\begin{prop}\label{prop:dimension-increase-by-1} For any $n>0$, we have 
\[ \ell((n-1)P)\leq \ell(nP)\leq \ell((n-1)P)+1.\]
\end{prop}

\proof
It suffices to show $\ell(nP)\leq \ell((n-1)P)+1$.  Suppose $f_1, f_2\in\ell(nP)\setminus\ell((n-1)P).$  Since $f_1$ and $f_2$ have the same pole order at $P$, using the series expansions of $f_1$ and $f_2$ with a local coordinate, one can find a linear combination of $f_1$ and $f_2$ to eliminate their leading terms.  That is, there are constants $c_1, c_2\in k$ such that $c_1f_1+c_2f_2$ has a strictly smaller pole order at $P$, so $c_1f_1+c_2f_2\in\L((n-1)P)$.  Then $f_2$ is in the vector space generated by a basis of $\L((n-1)P)$ along with $f_1$.  Since this is true for any two functions $f_1,f_2$, we conclude $\ell(nP)\leq \ell((n-1)P)+1$, as desired.

\qed

For any integer $n>0$, we call $n$ a \textbf{Weierstrass gap number of $P$} if 
$ \ell(nP)=\ell((n-1)P)$, 
 that is, if there is no function $f\in k(\X_g)^\times$ such that $(f)_\infty=nP$.  
%
\begin{thm}
For any point $P$, there are exactly $g$ gap numbers $\alpha_i(P)$ with \[1=\alpha_1(P)<\alpha_2(P)<\cdots<\alpha_g(P)\leq2g-1.\]
\end{thm}

This theorem is a special case of the Noether ``gap'' theorem, which we state and prove below. The set of gap numbers, denoted by $G_P$,  forms the \textbf{Weierstrass gap sequence} for $P$.  

\begin{defn} 
If the gap sequence at $P$ is anything other than $\{1,2,\dots,g\}$, then $P$ is called a \textbf{Weierstrass point}.
\end{defn}

Equivalently, $P$ is a Weierstrass point if $\ell(gP)>1$; that is, if there is a function $f$ with $(f)_\infty=mP$ for some $m$ with $1<m\leq g$.
The notion of gaps can be generalized, which we briefly describe.  Let $P_1,P_2,\dots,$ be a sequence of (not necessarily distinct) points on $\X_g$.  Let $D_0=0$ and, for $n\geq 1$, let $D_n=D_{n-1}+P_n$.  One constructs a similar sequence of vector spaces 
\[
\L(D_0)\subseteq\L(D_1)\subseteq\L(D_2)\subseteq\dots\subseteq\L(D_n)\subseteq\cdots \,,\] 
with a corresponding non-decreasing sequence of dimensions 
\[
\ell(D_0)<\ell(D_1)<\ell(D_2)<\dots<\ell(D_n)<\cdots.
\]  
If $\ell(D_n)=\ell(D_{n-1})$, then $n$ is a \textbf{Noether gap number} of the sequence $P_1, P_2, \dots .$

\begin{thm}
For any sequence $P_1,P_2, \dots$, there are exactly $g$ Noether gap numbers $n_i$ with \[1=n_1<n_2<\dots<n_g\leq 2g-1.\]
\end{thm}

\proof In analog with  \cref{prop:dimension-increase-by-1}, one can show the dimension goes up by at most 1 in each step; that is, 
\[ 
\ell(D_{n-1})\leq \ell(D_n)\leq \ell(D_{n-1})+1,
\] 
for all $n>0$.  First, note that the Riemann-Roch theorem is an equality for $n>2g-1$, so the dimension goes up by 1 in each step, so there are no gap numbers greater than $2g-1$.

Now, consider the chain $\L(D_0)\subseteq\dots\subseteq\L(D_{2g-1})$.  By Riemann-Roch, $\ell(D_0)=1$ and $\ell(D_{2g-1})=g$, so in this chain of vector spaces, the dimension must increase by 1 exactly $g-1$ times in $2g-1$ steps.  Thus, for $n\in\{1, 2, \dots, 2g-1\}$,  there are $g$ values of $n$ such that $\ell(D_n)=\ell(D_{n-1})$.  These $g$ values are the Noether gap numbers.
\qed

\begin{rem}
The Weierstrass ``gap'' theorem is a special case of the Noether ``gap'' theorem, taking $P_i=P$ for all $i$.
\end{rem}
This result is a direct application of the Riemann-Roch theorem, and the proof can be found in \cite{Farkas-Kra}*{III.5.4}.

Since a Weierstrass gap sequence contains $g$ natural numbers between $1$ and $2g-1$, and since its complement in $\N$ is a semi-group, we can begin to list the possible gap sequences for points on curves of small genus.
%
\begin{itemize}
\item For $g=1$, the only possible gap sequence is $\{1\}$.  Note that this means a curve of genus $g=1$ has no Weierstrass points.
\item For $g=2$, the possible sequences are $\{1,2\}$ and $\{1,3\}$.
\item For $g=3$, the possible sequences are $\{1,2,3\}, \{1,2,4\}, \{1,2,5\}, \{1,3,5\}$.
\end{itemize}

\subsection{Weierstrass points via holomorphic differentials}
Continuing with a point $P$ on a curve $\X_g$, recall that $n$ is a gap number precisely when $\ell(nP)=\ell((n-1)P)$.  By Riemann-Roch, this occurs exactly when 
\[\ell(K-(n-1)P)-\ell(K-nP)=1\] 
for a canonical divisor $K$, which is the divisor associated to some differential $dx$.  Thus there is $f\in k(\X_g)^\times$ such that 
\[ (f)+K-(n-1)P\geq0\]
 and $(f)+K-nP\not\geq0$, which implies that $\ord_P(f\cdot dx)=n-1$.  Since 
 \[ (f)+K\geq(n-1)P\geq0, \quad  \text{ for } n\geq 1,\]
$n$ is a gap number of $P$ exactly when there is a holomorphic differential $f\cdot dx$ such that $\ord_P(f\cdot dx)=n-1$.

For $H^0(\X_g,\Omega^1)$ the space of holomorphic differentials on $\X_g$, by Riemann-Roch, the dimension of $H^0(\X_g,\Omega^1)$ is $g$.  Let $\{\psi_i\}$, for $i=1,\dots, g$, be a basis, chosen in such a way that \[\ord_P(\psi_1)<\ord_P(\psi_2)<\cdots<\ord_P(\psi_g).\]  Let $n_i=\ord_P(\psi_i)+1$.
%
The \textbf{1-gap sequence at $P$} is $\{n_1,n_2,\dots,n_g\}$.

We then have the following equivalent definition of a Weierstrass point.
%
If the 1-gap sequence at $P$ is anything other than $\{1,2,\dots,g\}$, then $P$ is a Weierstrass point.

It follows that $P$ is a Weierstrass point exactly when there is a holomorphic differential $f\cdot dx$ with $\ord_P(f\cdot dx)\geq g$.

\begin{defn}
The \textbf{Weierstrass weight} of a point $P$ is 
\[w(P)=\sum_{i=1}^g (n_i-i).\]
\end{defn}
In particular, $P$ is a Weierstrass point if and only if $w(P)>0$.

\subsection{Bounds for weights of Weierstrass points}
Suppose $\X_g$ is a curve of genus $g\geq 1$, $P\in\X_g$, and consider the 1-gap sequence of $P$ $\{n_1,n_2,\dots,n_g\}$.  We will refer to the non-gap sequence of $P$ as the complement of this set within the set $\{1,2,\dots,2g\}$.  That is, the non-gap sequence is the sequence $\{\alpha_1,\dots,\alpha_g\}$ where %
\[
1<\alpha_1<\dots<\alpha_g=2g.
\]

\begin{prop}\label{prop:sums-of-non-gaps}
For each integer $j$ with $0<j<g$, $\alpha_j+\alpha_{g-j}\geq 2g$.
\end{prop}

\proof
Suppose there is some $j$ with $\alpha_j+\alpha_{g-j}<2g.$  The non-gaps are contained in a semigroup under addition, so for every $k\leq j$, since $\alpha_k+\alpha_{g-j}<2g$ as well, $\alpha_k+\alpha_{g-j}$ is also a non-gap which lies between $\alpha_{g-j}$ and $\alpha_g=2g$.  There are $j$ such non-gaps, though there can only be $j-1$ non-gaps between $\alpha_{g-j}$ and $\alpha_g$.  Thus, we have a contradiction.
\qed

\begin{prop}\label{prop:max-w-weight}
For $P\in\X_g$, 
\[ w(P)\leq \frac {g(g-1)} 2,\] 
with equality if and only if $P$ is a branch point on a hyperelliptic curve $\X_g$.
\end{prop}



\proof
The Weierstrass weight of $P$ is
\begin{align*}
w(P) & = \sum_{i=1}^g n_i- \sum_{i=1}^g i = \sum_{i=1}^{2g} i - \sum_{i=1}^g \alpha_i - \sum_{i=1}^g i = \sum_{i=g+1}^{2g-1} i - \sum_{i=1}^{g-1} \alpha_i.
\end{align*}
The first sum is $3g(g-1)/2$ and the second sum, via  \cref{prop:sums-of-non-gaps} is at least $(g-1)g$.  Hence, $w(P)\leq g(g-1)/2$.
To prove the second part, we note that the weight is maximized when the sum of the non-gaps is minimized.  That occurs when $\alpha_1=2$, which implies the non-gap sequence is $\{2,4,\dots,2g\}$, and so the 1-gap sequence is $\{1,3,5,\dots,2g-1\}$, which is the 1-gap sequence of a branch point on a hyperelliptic curve.
\qed


\begin{cor}\label{cor:number-w-pts}
For a curve of genus $g\geq2$, there are between $2g+2$ and $g^3-g$ Weierstrass points.  The lower bound of $2g+2$ occurs only in the hyperelliptic case.
\end{cor}

\proof The total weight of the Weierstrass points is $g^3-g$.  In  \cref{prop:max-w-weight}, we see that the maximum weight of a point is $g(g-1)/2$, which occurs in the hyperelliptic case.  Thus, there must be at least $\dfrac{g^3-g}{g(g-1)/2} = 2g+2$ Weierstrass points.  On the other hand, the minimum weight of a point is 1, so there are at most $g^3-g$ Weierstrass points.
\qed

\subsection{Higher-order Weierstrass points via holomorphic $q$-differentials}
In the above, we described Weierstrass points by considering the vector spaces $\mathcal{L}(K-nP)$ for $n\geq0$.  Now, we let $q\in\N$ and proceed analogously with the vector spaces $\mathcal{L}(qK-nP)$ to describe $q$-Weierstrass points.
If 
\[\ell(qK-(n-1)P)-\ell(qK-nP)=1,\] 
then there is some $q$-fold differential $dx^q$ and some $f\in k(\X_g)^\times$ such that $f\cdot dx^q$ is a holomorphic $q$-fold differential with $\ord_P(f\cdot dx^q)=n-1$.
Let $H^0(\X_g,(\Omega^1)^q)$ denote the space of holomorphic $q$-fold differentials on $\X_g$, and let $d_q$ denote the dimension of this space.  By the Riemann-Roch, it follows that
\[
d_q=\begin{cases}
g & \text{if $q=1$,} \\ 
(g-1)(2q-1) & \text{if $q>1$.}
\end{cases}
\]  
Let $\{\psi_i\}$, for $i=1,\dots,d_q$, be a basis of $H^0(\X_g,(\Omega^1)^q)$, chosen in such a way that 
\[\ord_P(\psi_1)<\ord_P(\psi_2)<\cdots<\ord_P(\psi_{d_q}).\]  
Let $n_i=\ord_P(\psi_i)+1$.
The \textbf{$q$-gap sequence at $P$} is $\{n_1,n_2,\dots,n_{d_q}\}$.
If the $q$-gap sequence is anything other than $\{1,2,\dots,d_q\}$, then $P$ is a $q$-Weierstrass point.

Thus, $P$ is a $q$-Weierstrass point exactly when there is a holomorphic $q$-fold differential $f\cdot dx^q$ such that $\ord_P(f\cdot dx^q)\geq d_q$.
When $q=1$, we have a Weierstrass point.  For $q>1$, a $q$-Weierstrass point is also called a \textbf{higher-order Weierstrass point}.  
The \textbf{$q$-Weierstrass weight} of a point $P$ is \[w^{(q)}(P)=\sum_{i=1}^{d_q}(n_i-i).\]
In particular, $P$ is a $q$-Weierstrass point if and only if $w^{(q)}(P)>0$.
For each $q\geq 1$, there are a finite number of $q$-Weierstrass points, which follows from  ~\cref{cor:total-q-weight}.



\section{Automorphisms}\label{sect-4}

Let $\X$ be an irreducible and non-singular algebraic curve defined over a field $k$. We denote its function field by $F:=k(\X)$.  The automorphism group of $\X$ is the group $G:=\Aut(F/k)$ (i.e., all field automorphisms of $F$ fixing $k$).  It has been the focus of research activity for over two hundred years and focused on one the following problems.

\begin{prob}
For a given $g\geq 2$ and  an algebraically closed field $k$,    
determine:
\begin{enumerate}
\item   a bound for $\Aut (\X_g)$ 

\item the list all groups which occur as full automorphism groups of curves $\X_g$ of genus $g$ defined over $k$. 

\item for every group $G$ from the list above, write down an equation for $\X_g$ such that $G \iso \Aut (\X_g)$. 
\end{enumerate}
\end{prob} 

For further details on automorphisms we will refer to \cite{b-sh-w}.  Throughout this section $C_n$ denotes the cyclic group of order $n$ and $D_n$ the dihedral group of order $2n$.

\subsection{The action of $k$-automorphisms on places}
$G$ acts on the places of $F/k$.  Since there is a one-to-one correspondence between places of $F/k$ and points of $\X$, this action naturally extends to the points of $\X$. For $\alpha \in G$ and $P \in \X$, we denote its image under $\alpha$ by $P^\alpha$.
In a natural way we extend this $G$-action to $\Div_k (\X)$. Let $D \in \Div_k (\X)$, say $D= \sum n_P \cdot P$. Then, the image of the divisor $D$ under the action of $\alpha$ is given by
\[ D^\alpha = \sum n_p \cdot P^\alpha.\]
\begin{lem}\label{lem-1} 
$G$ acts on the set $\W$ of Weierstrass points.
\end{lem}

\proof
The set $\W$ of  Weierstrass points do not depend on the choice of the local coordinate and so it is invariant under any $\sigma \in \Aut(\X_g)$.
\qed

Hence, in order to determine the automorphism group we can just study the action of the group on the set of Weierstrass point of the curve.
Then we have the following.

\begin{prop} 
Let $\a \in \Aut(\X)$ be a non-identity element. Then $\a$ has at most $2g+2$ fixed places.
\end{prop}

\proof
Let $\a$ be a non-trivial element of  $\Aut(F/k)$.  Since $\a$ is not the identity, there is some place $\p \in \PP_F$ not fixed by $\a$. Here, $\PP_F$ is the set of all places of $F/k$. Take $g+1$ distinct places $\p_1, \dots , \p_{g+1} $ in $\PP_F$ such that $D = \p_1 + \cdots +\p_{g+1}$ and $D^\a$ share no place. By \cite{hkt}*{Thm. 6.82} there is $z\in F\setminus k$ such that $\dv (z)_\infty = D$.  Then consider $w= z - \a (z)$.  Since $z$ and $\alpha (z)$ have different poles then $w\neq 0$.  Hence, $w$ has exactly $2g+2$ poles. Then $w$ has exactly $2g+2$ zeroes. But every fixed place of $\a$ is a zero of $w$.  Hence $\a$ has at most $2g+2$ fixed places.
\qed

Let $\W$ be the set of Weierstrass points.  From \cref{cor:number-w-pts} we know that $\W$ is finite. Since for every $\a  \in \Aut(\X)$, from \cref{lem-1} we have $\a(\W)=\W$.  Then we have the following.

\begin{thm}  
Let $\X$ be a genus $g\geq 2$ irreducible, non-hyperelliptic curve defined over $k$ such that $\ch k = p$ and $\a \in \Aut(\X)$. If $p=0$ or $p > 2g-2$ then $\a$ has finite order.
\end{thm}

Hence,  we have:

\begin{lem}   
If $p=0$ and $g\geq 2$ then every automorphism is finite.
\end{lem}

In the case of $p=0$,  Hurwitz \cite{Hu} showed $|\a| \leq 10 (g-1)$. In 1895, Wiman improved this bound to be $ |\a| \le 2(2 g+1)$ and showed this is best possible. If $|\a|$ is a prime then $|\a| \leq 2g +1$. Homma \cite{Ho}   shows that this bound is achieved for a prime $q \neq p$  if and only if the curve is birationally
equivalent to
\[ y^{m-s} (y-1)^s = x^q, \quad for \quad 1 \leq s < m \leq g +1. \]
If $p > 0$, we have the following; see \cite{hkt}*{Thm.~11.34}.

\begin{thm}  
Let $\X$ be a genus $g\geq 2$, irreducible curve defined over $k$, with $\ch k =p >0$ and $\alpha \in \Aut(\X)$ which fixes a place $\p \in \PP_F$.  Then the order of $\alpha $ is bounded by
\[ | \alpha | \leq 2 p (g+1) (2g+1)^2. \]
\end{thm}

\subsection{Finiteness of $\Aut(\X)$}

The main difference for $g=0, 1$ and $g\geq 2$ is that for $g\geq 2$ the automorphism group is a finite group.  This result was proved first by Schmid (1938).

\begin{thm}[\cite{Schmid}]
Let $\X$ be an irreducible curve of genus $g\geq 2$, defined over a field $k$, $\ch k = p \geq 0$. Then $\Aut(\X)$ is finite.
\end{thm}

\subsubsection{Characteristic $p=0$}

\def\s{\sigma}

For any $\s \in \Aut(\X_g)$, we denote by $|\s|$ its order and $\fix(\s)$ the set of fixed points of $\s$ on $\X_g$. Then we have:

\begin{prop}\label{prop:number-fixed-pts} 
Let $\s \in \Aut (\X_g)$ be a non-identity element. Then $\s$ has at most $2g+2$ fixed points.
\end{prop}

\proof  Let $\s$ be a non-trivial automorphism of $\X_g$ and let $\s^*$ denote the corresponding automorphism of $k(\X_g)$.  Since $\s$ is not the identity, there is some $P\in\X_g$ not fixed by $\s$.  By Riemann-Roch, $\ell((g+1)P)\geq2$, so there is a meromorphic $f\in k(\X_g)$ with $(f)_\infty = rP$ for some $r$ with $1\leq r\leq g+1$. Consider the function $h=f-\s^*(f)$.  The poles of $h$ are limited to the poles of $f$ and $\s^*(f)$, so $h$ has at most $2r$ poles.  Since $h$ is meromorphic, $h$ similarly has at most $2r$ zeroes, which correspond exactly to fixed points of $\s$.  Since $r\leq g+1$, we conclude $\s$ has at most $2g+2$ fixed points. 

\qed

\begin{prop}
Any genus $g\geq2$ nonhyperelliptic Riemann surface $\X_g$ has a finite automorphism group $\Aut(\X_g)$.
\end{prop}


\proof
Let $\s\in\Aut(\X_g)$ with corresponding automorphism $\s^*$ of $k(\X_g)$.  The Wronskian does not depend on choice of local coordinate and thus is invariant under $\s^*$.  Therefore, if $P$ is a $q$-Weierstrass point of a certain $q$-Weierstrass weight, then $\s(P)$ is a $q$-Weierstrass point with the same weight.
Thus, any automorphism permutes the set of Weierstrass points.  

Let $S_\W$ denote the permutation group of the set of Weierstrass points.  Since there are finitely many Weierstrass points (as in  ~\cref{cor:number-w-pts}), $S_\W$ is a finite group.  We have a homomorphism 
\[ \phi:\Aut(\X_g)\to S_\W.\]  
It will suffice to show that $\phi$ is injective.  We prove this separately in the cases that $\X_g$ is hyperelliptic or nonhyperelliptic.

Suppose $\X_g$ is non-hyperelliptic and suppose $\s\in\ker(\phi)$.  Then $\s$ fixes all of the Weierstrass points.  From  ~\cref{cor:number-w-pts}, since $\X_g$ is non-hyperelliptic, there are more than $2g+2$ Weierstrass points.  By  \cref{prop:number-fixed-pts}, $\s$ fixes more than $2g+2$ Weierstrass points and so must be the identity automorphism on $\X_g$. Thus, $\phi$ is an injection into a finite group, so $\Aut(\X_g)$ is finite.

Suppose $\X_g$ is hyperelliptic, and let $\omega\in\Aut(\X_g)$ denote the hyperelliptic involution.  Suppose $\s\in\ker(\phi)$ with $\s\neq\omega$.  $\s$ fixes the $2g+2$ branch points of $\X_g$.  Consider the map 
\[ \pi:\X_g\to\X_g/\langle\omega\rangle\cong\P^1.\]
  $\s$ descends to an automorphism of $\P^1$ which fixes  $2g+2 \geq 6$ points. Thus, $\s$   is the identity on $\P^1$.  Thus, $\s\in \< \omega \>$, so $\s$ is the identity in $\Aut(\X_g)$, which means $\ker(\phi)$ is finite, so $\Aut(\X_g)$ is finite.

\qed


Next is the  famous Hurwitz's theorem.

\begin{thm}[Hurwitz]
Any genus $g\geq 2$ Riemann surface $\X_g$ has at most $84 (g-1)$ automorphisms. 
\end{thm}

The following two results consider the number of fixed points of an automorphism $\s \in \Aut (\X_g)$. 

\begin{lem}
Let $\s\in\Aut(\X_g)$ be a non-trivial automorphism.  Then
\[ | \fix (\s)  \,  |  \, \leq 2 \, \frac {|\s| + g -1 } {|\s| -1}.\]  If $\X_g/\s\cong\P^1$ and $|\s|$ is prime, then this is an equality.
\end{lem}

\begin{cor} If $\X_g$ is not hyperelliptic, then for any non-trivial $\s \in \Aut (\X_g)$ the number of fixed points of $\s$ is $| \fix (\s) | \leq 2g-1$.
\end{cor}

Curves that attain this bound are called \textbf{Hurwitz curves}.  Klein's quartic is the only such Hurwitz curve of genus $g\le 3$. Fricke showed that the next Hurwitz group occurs for $g=7$ and has order 504. Its group is $\operatorname{SL}(2,8)$, and an equation for it was computed by Macbeath \cite{Mac} in 1965. Further Hurwitz curves occur for $g=14$ and $g=17$ (and for no other values of $g\le 19$).

For a fixed $g\geq 2$ denote by $N(g)$ the maximum of the $|\Aut(\X_g)|$. Accola \cite{Ac1} and Maclachlan \cite{Mc1} independently show that $N(g) \geq 8 (g+1)$ and this bound is sharp for infinitely many $g$'s. If $g$ is divisible by 3 then $N(g) \geq 8 (g+3)$.

The following terminology is standard: we say $G\le\Aut(\X_g)$ is a {\bf large automorphism group} in
genus $g$ if  $|G|\ \ >\ \ 4 (g-1)$. 
In this case the quotient of $\X_g$ by $G$ is a curve of genus $0$, and the number of points of this quotient ramified
in $\X_g$ is 3 or 4 (see \cite{kyoto}  or \cite{Farkas-Kra}, pages 258-260). 


\subsubsection{Characteristic $p>0$}

In the case of positive characteristic the bound is higher due to possible wild ramifications.  The following was proved by Stichtenoth   by extending previous results of P. Roquette and others.

\begin{thm}[\cite{Sti-73}]
Let $\X$ be an irreducible curve of genus $g\geq 2$, defined over a field $k$, $\ch k = p>0$. Then
\[ |\Aut(\X)| < 16 \cdot g^4, \]
unless $\X$ is the curve with equation
\[ y^{p^n} + y= x^{p^{n+1}}, \]
in which case it has genus $g= \frac 1 2 p^n (p^n-1) $ and $ |\Aut(\X)| = p^{3n} (p^{3n} +1) (p^{2n}-1)$.
\end{thm}

Hence, we have a bound for curves of genus $g\geq2$ even in characteristic $p>0$.  It turns out that all curves with large groups of automorphisms are special curves.  So getting ``better" bounds for the complementary set of curves has always been interesting.  There is an extensive amount of literature on this topic due to the interest of such bounds for coding theory. 

The following theorem, which  is due to Henn, provides a better bound if the following four families of curves are left out.   
This result may be sharpened to show that the order of $\Aut(\X)$  is less than $3\cdot (2g)^{5/2}$ except when $k (\X)$  belongs to one of five types of function fields, as Henn points out in a footnote.  Note that there is a flaw in Henn's article which was corrected in \cite{GK}.  A full detailed account of automorphisms of curves has lately appeared in the wonderful book \cite{hkt}.
   
\begin{thm}[\cite{Henn}] \label{main-thm}  
Let $\X$ be an irreducible curve of  genus $g\geq 2$.  If $|G| \geq 8g^3$, then $\X$ is isomorphic to one of the following:

i) The hyperelliptic curve
$ y^2+y + x^{2^k+1} =0$,
defined over a field of characteristic $p=2$.  In this case the genus is $g=2^{k-1}$ and $|G|=2^{2k+1} (2^k+1)$.

ii)  The hyperelliptic curve
$  y^2=x^q- x$, 
defined over a field of characteristic $p>2$ such that $q$ is a power of $p$.   In this case $g=\frac 1 2 (q-1)$ and the reduced group $\bar G$ is isomorphic to $\psl_2(q)$ or $\pgl_2(q)$.

iii) The Hermitian curve
$ y^q + y = x^{q+1}$, 
defined over a field of characteristic $p\geq 2$ such that $q$ is a power of $p$. In this case $g= \frac 1 2 (q^2-q)$ and $G$ is isomorphic to $PSU(3, q)$ or $PGU(3, q)$.

iv) The curve    $y^q+y = x^{q_0} (x^q+x)$, 
for $p=2$, $q_0=2^r$, and $q=2q_0^2$.  In this case, $g= q_0 (q-1)$ and $G \iso Sz (q)$.
\end{thm}

Determining the equation of the curve with given automorphism group is generally a difficult problem which we will discuss in more details 
in the coming sections. Before we go into detail about special families of curves we want to leave the reader with the following problem.

\begin{prob}
Given an irreducible algebraic curve $\X$ with affine equation  $F(x, y) =0$, defined over a field $k$, find an algorithm which determines the automorphism group of $\X$  over $\bar k$.
\end{prob}

\subsection{Hyperelliptic curves}
Let $k$  be an algebraically  closed field of characteristic  zero and $\X_g$  be a genus  $g$ hyperelliptic curve given  by the equation $y^2=f(x)$. Denote  the function field of $\X_g$ by $K:=k(\X_g)=k(x,y)/\<y^2-f(x)\>$. Then, $k(x)$ is the  unique degree 2 genus zero subfield of  $K$. $K$ is  a quadratic extension field of $k(x)$ ramified exactly at $d=2g+2$ places $\a_1, \dots , \a_d$  of $k(x)$. The corresponding places of $K$ are   the Weierstrass points of $K$. Let $\B:=\{ \a_1, \dots , \a_d \}$ and $G:=\Aut(K/k)$. Since $k(x)$  is the only  genus 0 subfield of degree  2  of $K$, then  $G$  fixes $k(x)$.  Thus, $G_0:= \Gal (K/k(x))=\< \tau \>$, with $\tau^2=1$, is central in $G$. We call  \textbf{the reduced automorphism group} of $K$ the  group $\G:=G/G_0$. 

The reduced automorphism group $\G$ is isomorphic to one of the following: 
\[  C_n, \, D_n, \,  A_4, \, S_4, \,  A_5
\]
and branching indices of the corresponding cover $\P^1_x \to \P^1/ \G$ given by 
 \[ (n,n),\  (2, 2, n),\ (2, 3, 3),\ (2, 4, 4),  (2, 3, 5), 
 \]
respectively.  We fix a coordinate $z$ in $\P^1/ \G$. Thus, $\G$ is the monodromy group of a cover $\ff: \P^1_x \to \P^1_z$. Denote by $q_1, \dots , q_r $ the corresponding branch points of $\ff$. Let $S$ be the set of branch points of $\f: \X_g \to \P^1_z$. Clearly $q_1, \dots , q_r \in S$. Let $W$ denote the images in $\P^1$ of Weierstrass points of $\X_g$ and $V:=\cup_{i=1}^r \ff^{-1} (q_i) $. For each $q_1, \dots , q_r$ we have a corresponding permutation $\s_1, \dots , \s_r \in S_n$. The tuple $\bar \s:=(\s_1, \dots , \s_r)$ is the signature of $\G$. Thus, %
$ \G= \< \s_1, \dots , \s_r \>, \quad  \text{and }   \quad \s_1 \cdots \s_r =1$. 
Since each of the above groups is embedded in $\pgl_2 (k)$ then we can have these generating systems $\s_1, \dots , \s_r$ as matrices in $\pgl_2 (k)$. Below we display all the cases:

\begin{equation}\label{eq1a}
\begin{split}
i) \quad C_n  &  \iso \left\< \begin{bmatrix}   \zeta_n & 0 \\ 0 & 1 \\ \end{bmatrix}, \begin{bmatrix}   \zeta_n^{n-1} & 0 \\ 0 & 1 \\ \end{bmatrix} \right\> \\
 ii) \quad D_n & \iso \left\< \begin{bmatrix} 0 & 1\\1 & 0 \\\end{bmatrix}, \begin{bmatrix} 0 & 1\\1 & 0 \\\end{bmatrix}, \begin{bmatrix}   \zeta_n & 0 \\ 0 & 1 \\ \end{bmatrix}  \right\> \\
iii) \quad A_4 & \iso \left\< \begin{bmatrix}  -1 & 0 \\ 0 & 1 \\ \end{bmatrix}, \begin{bmatrix} 1 & i  \\ 1 & -i  \\ \end{bmatrix}  \right\>  \\
iv) \quad S_4 & \iso \left\< \begin{bmatrix}   -1 & 0 \\ 0 & 1 \\ \end{bmatrix}, \begin{bmatrix}   0 & -1 \\ 1 & 0 \\ \end{bmatrix}, \begin{bmatrix} -1 & -1 \\ 1 &  1 \\ \end{bmatrix} \right\>  \\
v) \quad A_5 & \iso \left\< \begin{bmatrix}   \w & 1 \\ 1 & -\w \\ \end{bmatrix}, \begin{bmatrix}   \w & \e^4 \\ 1 & -\e^4\w \\ \end{bmatrix} \right\> \\
\end{split}
\end{equation}
where $\omega=\frac{-1+\sqrt{5}}{2}$,  $\zeta_n$ is a primitive $n^{th}$ root of unity, $\e$ is a primitive $5^{th}$ root of unity, and $i$ is a primitive $4^{th}$ root of unity.

The group $\G$ given above acts on $k(x)$ via the natural way. The fixed field is a genus 0 field, say $k(z)$. Thus, $z$ is a degree $| \G |$ rational function in $x$, say $z=\phi(x)$. 
%
\begin{lem}\label{lem1a}
Let $H$ be a finite subgroup of $\pgl_2(k)$. Let us identify each element of $H$ with the corresponding Moebius transformation and let $s_i$ be the $i$-th elementary symmetric polynomial in the elements of $H$,  $i=1,\ldots,|H|$. Then any non-constant $s_i$ generates $k(z)$.
\end{lem}
\begin{proof}
It is easy to check that the $s_i$ are the coefficients of the minimum polynomial of $x$ over $k(z)$. It is well-known that any non-constant coefficient of this polynomial generates the field.
\end{proof}

The fixed field for each of the groups $\G$ in cases i) - v) is generated  by the function
\begin{equation}
\begin{split}
i)   \quad  &  z= x^n  \\
ii) \quad  &  z= x^n + \frac 1 {x^n}\\
iii) \quad  &  z=\frac {x^{12} -33x^8 -33 x^4+1} {x^2 (x^4-1)^2} \\
iv) \quad  &  z= \frac {(x^8 + 14 x^4 +1 )^3} {108 \left( x (x^4-1) \right)^4} \\
v) \quad &  z= \frac {\left(-x^{20}+228 x^{15}-494x^{10}-228x^5-1 \right)^3} {1728 \left(x(x^{10} + 11x^5 -1)\right)^5} \\
\end{split}
\end{equation}
 
Notice that the branch points of a rational function $\phi (x)= \frac {f(x)} {g(x)}$ are exactly the zeroes of the discriminant of the polynomial $r(x):=f(x)-t\cdot g(x)$ with respect to $x$. Then the branch points of each of the above functions are 
\begin{itemize}

\item[i)] $\{0, \infty\}$,

\item[ii)] $\{ -2, 2, \infty\}$,

\item[iii)] $\{\infty, -6i\sqrt{3}, 6i\sqrt{3}\}$,

\item[iv)] $\{0, 1, \infty\}$,

\item[v)]  $\{ 0, 1728, \infty\}$.
\end{itemize}

The  group $G$  is a degree 2 central extension of $\G$.  The following  is proved in \cite{g_sh}.
\begin{lem}
Let $p\geq 2$, $\a \in G$ and $\bar \a$ its image in $\G$ with order $| \, {\bar \a} \, |= p $. Then,

i) $| \, \a \, | = p$ if and only if it fixes no Weierstrass points.

ii) $| \, \a \, | = 2p$ if and only if it fixes some Weierstrass point.
\end{lem}

Let $W$ denote the images in $\P^1_x$ of Weierstrass places of $\X_g$ and $V:=\cup_{i=1}^3 \phi^{-1} (q_i)$.
Let $z= \frac {\nf (x) } { \df (x) }$, where $\nf, \df \in k[x]$. For each branch point $q_i$, $i=1,2,3$ we
have the degree $|\G|$ equation $z\cdot \df(x) - q_i \cdot \df (x) =\nf (x),$ where the multiplicity of the
roots correspond to the ramification index for each $q_i$ (i.e., the index of the normalizer in $\G$ of
$\s_i$). We denote the ramification of $\phi: \P^1_x \to \P^1_z$, by $\varphi_m^r, \chi_n^s,  \psi_p^t$,
where the subscript denotes the degree of the polynomial.

Let $\l \in S \setminus \{ q_1, q_2, q_3\}$. The points in the fiber of a non-branch point $\l$ are the roots
of the equation: $ \nf (x) -\l \cdot \df (x) =0.$ To determine  the equation of the curve we simply need to
determine the Weierstrass points of the curve. For each fixed $\phi$ there are the following eight cases:
\begin{equation}\label{cases}
\begin{split}
1) & \quad  \, V \cap W = \emptyset,\\
2) & \quad  \, V \cap W =\phi^{-1} (q_1),\\
3) & \quad  \, V \cap W =\phi^{-1} (q_2),\\
4) & \quad  \, V \cap W =\phi^{-1} (q_3),\\
5) & \quad  \, V \cap W =\phi^{-1} (q_1) \cup \phi^{-1} (q_2),\\
6) & \quad  \,V \cap W =\phi^{-1} (q_2) \cup \phi^{-1} (q_3),\\
7) & \quad  \, V \cap W =\phi^{-1} (q_1) \cup \phi^{-1} (q_3),\\
8) & \quad  \, V \cap W =\phi^{-1} (q_1) \cup \phi^{-1} (q_2) \cup \phi^{-1} (q_3).
\end{split}
\end{equation}
\noindent It turns out that the above cases also determine the   full automorphism groups. We define the
following groups  as follows:
\begin{equation}\label{groups1}
\begin{split}
V_n := & \<  \, \, x, y \, | \, x^4, y^n,  (xy)^2, (x^{-1}y)^2\, \>, \quad
H_n :=  \<\,  x, y \, \, | \, \, x^4, y^2x^2, (xy)^n \, \>,   \\
G_n := & \< \, x, y \, \, | \, \, x^2 y^n, y^{2n},  x^{-1} y x y \, \>,\quad
U_n :=  \< \, x, y \, | \, x^2, y^n,  xyxy^{n+1} \>,  \\
\end{split}
\end{equation}
These groups are also called \textbf{twisted dihedral, double dihedral, generalized quaternion}, and
\textbf{semidihedral}. We warn the reader that these terms are not standard in the literature. They are all
four degree 2 central extensions of the dihedral group $D_n$ and therefore have order $4n$. Notice that $V_2$
is isomorphic with the dihedral group of order 8 and $H_2 \iso U_2 \iso C_2 \o C_4$. Furthermore, we have
the following result, the proof is elementary, and we skip the details.

\begin{rem}  i) If $n\equiv 1 \mod 2$ then $H_{4n} \iso G_{4n}$

 ii) If $n=2^{s+1}$ then $G_n = Q_{2^{s+1}}$ for any $s\in \Z$.
\end{rem}

\noindent The following groups
\begin{equation*}
\begin{split}
W_2:=  & \< \, x, y \, | \, x^4, y^3, y x^2  y^{-1} x^2, (x y)^4 \>, \quad
W_3:=   \< \, x, y \, | \, x^2, y^3, x^2 (x y)^4, (x y)^8 \> \\
\end{split}
\end{equation*}
are degree 2 central extensions of $S_4$. We have the following result:
\begin{thm}
The full automorphism group of a hyperelliptic curve is isomorphic to one of the following
$C_2 \times C_n$,  $C_n$,
$C_2 \times D_n$, $V_n$, $D_n$, $H_n$, $G_n$, $U_n$,
$C_2 \times A_4$, $SL_2(3)$, $C_2 \o S_4$,  $GL_2 (3)$, $W_2$, $W_3$
$C_2 \times A_5$, $SL_2 (5)$.
%
\end{thm}
In \cref{sect-7} we will show how to determine a parametric equation of the  curve for each case. 
Can this be done for non-hyperelliptic curves? A natural generalization of hyperelliptic curves are the superelliptic curves which we will discuss next.

\part{Superelliptic curves}

\section{Superelliptic curves}\label{sect-5}
To generalize the theory of hyperelliptic case, we consider curves which have an automorphism similar to the  \emph{hyperelliptic involution}. 

\subsection{Superelliptic Riemann surfaces}
A curve $\X$ is called \textbf{cyclic $n$-gonal}, where $n \geq 2$ is an integer, if there exists $\tau \in \Aut(\X)$ of order $n$ so that the quotient   ${\mathcal O}=\X/\langle \tau\rangle$ has genus zero;  $\tau$ is called a \emph{$n$-gonal automorphism} and $H=\langle \tau \rangle \cong C_{n}$ a \emph{$n$-gonal group} of $\X$.  
Let us consider, in this case, a regular branched covering $\pi: \X \to \P^1 (k)$ whose deck covering group is $H$. If $H$ is a normal subgroup of $\Aut(\X)$, then the computation of the group $G=\Aut(\X)$ can be done by studying the short sequence
\[ 1 \to H \to G \to \overline{G},\]
where $\overline{G}:= G/H$ is called the \textbf{reduced automorphism group} of $\X$. 

The case $n=p$ a prime integer has been the most studied one. For instance, in \cite{Gabino} it was observed that any two $p$-gonal groups of $\X$ are conjugated in $\Aut(\X)$ and, by 
Castelnuovo-Severi's inequality \cites{Accola1, CS}, for $g>(p-1)^{2}$ the $p$-gonal group is unique.
This uniqueness property also holds for any integer $n$ if $\X/H$ is fully ramified, see \cite{K}. The uniqueness also holds if $2 \leq g < (p-1)(p-5)/10$ (for instance, for $p \geq 11$ and $g=(p-1)/2$) \cite{Hid:pgrupo}.

Let us assume that $\pi:\X \to \P^1$ is tame and  the finite branch values of $\pi:\X \to \P^1$ are given by the collection of pairwise different points $a_{1}, \ldots, a_{r} \in \P^1$. Then the cyclic $n$-gonal curve $\X$ can be represented by an affine irreducible  algebraic curve, which might have singularities, of the following form (called a \emph{cyclic $n$-gonal curve})
\begin{equation}\label{ngonal}
\quad y^{n}=\prod_{j=1}^{r}(x-a_{j})^{l_{j}},
\end{equation}
where (i) $l_{1},\ldots,l_{r} \in \{1,\ldots,n-1\}$, (ii) $\gcd(n,l_{1},\ldots,l_{r})=1$; in this model, $\tau$ and $\pi$ are given by $\tau(x,y)=(x,\omega_{n}y)$, where $\omega_{n}=e^{2 \pi i/n}$, and $\pi(x,y)=x$. The point $\infty$ is a branch value of $\pi$ if and only if $l_{1}+\cdots+l_{r}$ is not congruent to zero module $n$. Let us denote by $N$ the normalizer of $H$ in $\Aut(\X)$.

A particular class of cyclic $n$-gonal curves, called \emph{superelliptic curves of level $n$}, has been introduced in \cites{super1}. These correspond, in the above algebraic description \cref{ngonal}, to the case when all the exponents $l_{j}$ are equal to $1$. In this case, $\tau$ happens to be central in $N$. 
In the generic situation, it happens that $N=\Aut(\X)$, that is, $\tau$ is central in $\Aut(\X)$;
$\tau$ is called a \emph{superelliptic automorphism of level $n$} and $H=\langle \tau\rangle$ a \emph{superelliptic group of level $n$}. 
In this case, all cone points of $\X/H$ have order $n$ and a classification of those was provided in \cite{Sa}. 
%

For the general cyclic $n$-gonal curve \cref{ngonal} 
it happens that, for the generic case, $\tau$ is central in $N$. In this situation we call $\tau$ a \textbf{generalized superelliptic automorphism of level $n$}, $H$ a \textbf{generalized superelliptic group of level $n$}, $\X$ a \textbf{generalized superelliptic surface of level $n$} and the corresponding cyclic $n$-gonal curve \cref{ngonal} a \textbf{generalized superelliptic curve of level $n$}; see \cite{HQS} for details.

Motivated by the above discussion we have the following definition.

\begin{defi} A  a genus $g \geq 2$ smooth, irreducible,  algebraic curve    $\X$ defined over an algebraically closed field $k$ is called  a \textbf{superelliptic curve of level $n$}  if there exist an element $\tau \in \Aut (\X)$ of order $n$ such that $\tau$ is central and  the quotient $\X / \< \tau \> $ has genus zero.
\end{defi}

Next we will see that with the above definition, superelliptic curves mimic exactly the theory of hyperelliptic curves.

Let $\X$ be a genus $g\geq 2$  defined over $k$ such that there exists an order $n>1$ automorphism $\sigma \in \Aut (\X)$ with the following properties:   i) $H:=\<\s\>$ is normal in $\Aut (\X)$, and ii) $\X/\<\s\>$ has genus zero.  Such curves are called superelliptic curves and their Jacobians, superelliptic Jacobians.  They have affine equation 
\begin{equation}\label{super}
\X  : \; y^n = f(x) = \prod_{i=1}^d (x-\a_i)
\end{equation}
We denote by $\sigma$ the superelliptic automorphism of $\X$.  So $\sigma : \X \to \X$ such that 
\[ \sigma (x, y) \to (x, \xi_n y),
\]
where $\xi_n$ is a primitive $n$-th root of unity.  Notice that $\sigma$ fixes 0 and  the point at infinity in $\P_y^1$.   

The natural projection 
\[ \pi : \X \to \P^1_x=\X/\<\sigma\>\]
is called the \textbf{superelliptic projection}.  It  has  $\deg \pi =n$ and 
$ \pi (x, y) =  x$. 
This cover is branched at exactly at the roots $\a_1, \dots , \a_d$ of $f(x)$.

If the discriminant $\Delta (f, x) \neq 0$ and $d>n$   then from the Riemann-Hurwitz formula we have
\[ g = \frac 1 2 \left(  n(d-1) - d - \gcd (n, d)    \frac {} {}  \right)     + 1 \]  
There is a lot of confusion in the literature over the term \textit{superelliptic} or \textit{cyclic} curves.  To us a \textit{superelliptic curve} it is a curve which satisfies \cref{super} with discriminant $\Delta (f, x) \neq 0$.  

If $\gcd (n, d)=1$ then $\deg f$ is either  $\frac {2g} {n-1}+ 2$ or $ \frac {2g} {n-1}+ 1$, depending on whether or not the place at infinity is a branch point of the superelliptic projection map.

\subsection{Automorphism groups}
Let $k$ be an algebraically closed field of characteristic $p\geq 0$ and $\X_g$ be a genus $g$ cyclic curve given by the equation $y^n=f(x)$ for some $f \in k[x]$. Let $K:=k(x,y)$ be the function field of $\X_g$. Then $k(x)$ is degree n genus zero subfield of $K$. Let $G=\Aut(K/k)$. Since 
\[ C_n: = \Gal (K/k(x))=\langle \tau \rangle,\]
 with $\tau^n=1$ such that $\langle \tau \rangle \lhd G$, then group $\G:=G/C_n$ and $\G \leq \pgl_2(k)$. Hence $\G$ is isomorphic to one of the following:   
\[ C_m, D_m, A_4, S_4, A_5,\]
 \emph{semidirect product  of elementary  Abelian  group  with  cyclic  group}, $ \psl_2(q)$ and $\pgl_2(q)$,   see \cite{VM}.

\[
\xymatrix{
K =k(x, y)\ar@{-}[d]^{\, \, \, C_n} \ar@/_1.8pc/[dd]_{\, \, G} \\
k(x, y^n) \ar@{-}[d]^{\, \, \, \G}  \\ k (z)  \\ }
\qquad \qquad \qquad \xymatrix{ \X_g \ar[d]_{\, \, \phi_0 }^{\, \, \,
C_n} \ar@/_1.8pc/[dd]_{\, \, \Phi} \\
\P^1 \ar[d]_{\, \, \, \phi}^{\, \, \, \G} \\ \P^1  \\ }
\]

The group $\G$ acts on $k(x)$ via the natural way. The fixed field is a genus 0 field, say $k(z)$. Thus $z$ is a degree $|\G|$ rational function in $x$, say $z=\phi (x)$. We illustrate with the above diagram.

Let $\phi_0 : \X_g \to \P^1$ be the cover which corresponds to the degree n extension $K/k(x)$. Then $\Phi :=\phi \circ \phi_0$ has monodromy group $G:=\Aut(\X_g)$. From the basic covering theory, the group $G$ is embedded in the group $S_l$ where $l=\deg$ $\P$. There is an $r$-tuple $\overline{\sigma}:=(\sigma_1, \dots , \sigma_r)$, where $\sigma_i \in S_l$ such that $\sigma_1, \dots ,\sigma_r$ generate $G$ and $\sigma_1\dots \sigma_r=1$. The signature of $\P$ is an $r$-tuple of conjugacy classes $\bC:=(C_1, \ldots ,C_r)$ in $S_l$ such that $C_i$ is the conjugacy class of $\sigma_i$. We use the notation $n$ to denote the conjugacy class of permutations which is cycle of length $n$. Using the signature of $\phi : \P^1 \to \P^1$ one finds out the signature of $\Phi : \X_g \to \P^1$ for any given $g$ and $G$.
Let $E$ be the fixed field of G, the Hurwitz genus formula states that
\begin{equation}\label{hurwitz-1}
2(g_K-1)=2(g_E-1)|G|+\deg(\mathfrak{D}_{K/E})
\end{equation}
with $g_K$ and $g_E$ the genera of $K$ and $E$ respectively and $\mathfrak{D}_{K/E}$ the different of $K/E$. Let $\overline{P}_1, \overline{P}_2, \dots , \overline{P}_r$ be ramified primes of $E$. If we set $d_i=\deg(\overline{P}_i)$ and let $e_i$ be the ramification index of the $\overline{P}_i$ and let $\beta_i$ be the exponent of $\overline{P}_i$ in $\mathfrak{D}_{K/E}$. Hence, \cref{hurwitz-1} may be written as
\begin{equation}\label{e2}
2(g_K-1)=2(g_E-1)|G|+|G|\sum_{i=1}^{r}\frac{\beta_i}{e_i}d_i
\end{equation}
If $\overline{P}_i$ is tamely ramified then $\beta_i=e_i-1$ or if $\overline{P}_i$ is wildly ramified then $\beta_i=e_i^*q_i+q_i-2$ with $e_i=e_i^*q_i$, $e_i^*$ relatively prime to $p$, $q_i$ a power of $p$ and $e_i^*|q_i-1$. For fixed $G$, $\bC$ the family of covers $\P:\X_g\to \P^1$ is a Hurwitz space $\cH(G,\bC)$. $\cH(G,\bC)$ is an irreducible algebraic variety of dimension $\d(G,\bC)$. Using equation \cref{e2} and signature $\bC$ one can find out the dimension for each $G$.

We denote by $K_m$ the  following semidirect product of elementary Abelian group  with  cyclic group  
$K_m:=\left\langle\left\{ \sigma_a, \t |a \in \U_m\right\}\right\rangle$, 
where $\t(x)=\xi^2x, \quad \sigma_a(x)=x+a,$ for each $a \in \U_m$,
\[\U_m :=\{a \in k| (a\prod_{j=0}^{\frac{p^t-1}{m}-1}(a^m-b_j))=0\}\]
$b_j \in \F^*_q$, $m|p^t-1$ and $\xi$ is a primitive $2m$-th root of unity.  $\U_m$ is a subgroup of the additive group of $k$.

\begin{small}
\begin{table}[hbt]
\begin{center}
\begin{tabular}{cccc}
$Case$ & $\G$ & $z$ & $Ramification$  \\
\hline \\
1 & $C_m$, $(m,p)=1$& $x^m$ & $(m,m)$\\ \\
2 & $D_{2m}$, $(m,p)=1$& $x^m+\frac{1}{x^m}$ & $(2,2,m)$\\  \\
  3 & $A_4, \, p\neq 2, 3$ & $\frac{x^{12}-33x^8-33x^4+1}{x^2(x^4-1)^2}$ & $(2,3,3)$      \\   \\
4 & $S_4, \, p\neq 2, 3$ & $\frac{(x^8+14x^4+1)^3}{108(x(x^4-1))^4}$ & $(2,3,4)$\\  \\
5 & $A_5, \,p\neq 2, 3, 5$ & $\frac{(-x^{20}+228x^{15}-494x^{10}-228x^5-1)^3}{(x(x^{10}+11x^5-1))^5}$ & $(2,3,5)$\\  \\
  & $A_5, \,p=3$ & $\frac{(x^{10}-1)^6}{(x(x^{10}+2ix^5+1))^5}$ & $(6,5)$\\  \\
6 & $U$ & $  \displaystyle{\prod_{a \in H_t}} (x+a)$ & $(p^t)$\\  \\
7 & $K_m$ & $(x   \displaystyle{\prod_{j=0}^{\frac{p^t-1}{m}-1} } (x^m-b_j))^m$ & $(mp^t,m)$ \\  \\
8 & $\psl_2(q), \,p\neq 2$ & $\frac{((x^q-x)^{q-1}+1)^{\frac{q+1}{2}}}{(x^q-x)^{\frac{q(q-1)}{2}}}$ & $(\alpha,\beta)$ \\  \\
9 & $\pgl_2(q)$ & $\frac{((x^q-x)^{q-1}+1)^{q+1}}{(x^q-x)^{q(q-1)}}$ & $(2\alpha,2\beta)$ \\
\end{tabular}
\caption{Rational functions correspond to each reduced automorphism group} 
\label{t1}
\end{center}
\end{table}

\end{small}


\begin{lem}\label{l1}
Let $k$  be an algebraically closed field of characteristic $p$, 
$\G$ be a finite subgroup of $\pgl_2(k)$ acting on the field $k(x)$. Then, $\G$ is isomorphic to one of the following groups 
\[  
C_m, D_m, A_4, S_4, A_5, U= C_p^t, K_m, \psl_2(q), \pgl_2(q),
\]
where $q=p^f$ and $(m,p)=1$. Moreover, the fixed subfield $k(x)^{\G}=k(z)$ is given by  ~\cref{t1}, where $\alpha =\frac{q(q-1)}{2}$, $\beta= \frac{q+1}{2}$, and $H_t$ is a subgroup of the additive group of $k$ with $| H_t | = p^t$ and $b_j \in k^*$.
\end{lem}
  

By considering the lifting of ramified points in each $\G$, we divide each $\G$ into sub cases and determine the signature of each sub case by looking the behavior of lifting and ramification of $\G$. Using that signature and \cref{e2} we calculate the moduli dimension $\delta$ (cf. \cref{sect-5})  for each case. 

\begin{thm}[\cite{Sa-sh}] \label{th1}
Let  $\X$ be a genus $g\geq 2$ superelliptic curve.  The signature of  $\Phi :\X \to \X^{\Aut(\X)}$ and  the moduli dimension $\delta$ are given in  \cref{long-table}, 
where $m=|\psl_2(q)|$ for cases 38-41 and $m=|\pgl_2(q)|$ for cases 42-45.
\end{thm}


\begin{footnotesize}

\renewcommand{\arraystretch}{2}

\begin{longtable}{|c|c|c|c|}
\hline \hline
\multicolumn{1}{|c|}{$\#$} & \multicolumn{1}{|c|}{$\G$} & \multicolumn{1}{|c|}{$\delta(G,\bC)$} & \multicolumn{1}{|c|}{$\bC=(C_1, \dots ,C_r)$}\\
\hline \hline
\endfirsthead
\hline \hline
\multicolumn{1}{|c|}{$\#$} & \multicolumn{1}{|c|}{$\G$} & \multicolumn{1}{|c|}{$\delta(G,\bC)$} &  \multicolumn{1}{|c|}{$\bC=(C_1, \dots ,C_r)$}\\
\hline \hline
\endhead
\hline
\multicolumn{4}{r}{\itshape continued on the next page}\\
\endfoot
\multicolumn{2}{r}{ }
\endlastfoot
$1$ &$(p,m)=1$ & $\frac{2(g+n-1)}{m(n-1)}-1$ &  $(m,m,n,\dots,n)$ \\
$2$ & $C_m$ & $\frac{2g+n-1}{m(n-1)}-1$ &  $(m,mn,n,\dots,n)$\\
$3$ &  & $\frac{2g}{m(n-1)}-1$ &   $(mn,mn,n,\dots,n)$\\
\hline
$4$ & $(p,m)=1$ & $\frac{g+n-1}{m(n-1)}$ &  $(2,2,m,n,\dots,n)$  \\
$5$ &  & $\frac{2g+m+2n-nm-2}{2m(n-1)}$ &  $(2n,2,m,n,\dots,n)$  \\
$6$ & $D_{2m}$ & $\frac{g}{m(n-1)}$ &  $(2,2,mn,n,\dots,n)$ \\
$7$ & &$\frac{g+m+n-mn-1}{m(n-1)}$ &  $(2n,2n,m,n,\dots,n)$  \\
$8$ &  & $\frac{2g+m-mn}{2m(n-1)}$ &  $(2n,2,mn,n,\dots,n)$  \\
$9$ &  & $\frac{g+m-mn}{m(n-1)}$ &  $(2n,2n,mn,n,\dots,n)$  \\
\hline \hline
$10$ &  & $\frac{n+g-1}{6(n-1)}$ &  $(2,3,3,n,\dots,n)$  \\
$11$ &$ A_4$ & $\frac{g-n+1}{6(n-1)}$ &  $(2,3n,3,n,\dots,n)$ \\
$12$ &  & $\frac{g-3n+3}{6(n-1)}$ &  $(2,3n,3n,n,\dots,n)$  \\
$13$ &  & $\frac{g-2n+2}{6(n-1)}$ &  $(2n,3,3,n,\dots,n)$  \\
$14$ &  & $\frac{g-4n+4}{6(n-1)}$ &  $(2n,3n,3,n,\dots,n)$  \\
$15$ &  & $\frac{g-6n+6}{6(n-1)}$ &  $(2n,3n,3n,n,\dots,n)$  \\
\hline \hline
$16$ &  & $\frac{g+n-1}{12(n-1)}$ &  $(2,3,4,n,\dots,n)$  \\
$17$ &  & $\frac{g-3n+3}{12(n-1)}$ &  $(2,3n,4,n,\dots,n)$  \\
$18$ &  &  $\frac{g-2n+2}{12(n-1)}$ &  $(2,3,4n,n,\dots,n)$ \\
$19$ &  & $\frac{g-6n+6}{12(n-1)}$ &  $(2,3n,4n,n,\dots,n)$  \\
$20$ &  $S_4$ & $\frac{g-5n+5}{12(n-1)}$ &  $(2n,3,4,n,\dots,n)$  \\
$21$ &  & $\frac{g-9n+9}{12(n-1)}$ &  $(2n,3n,4,n,\dots,n)$  \\
$22$ &  & $\frac{g-8n+8}{12(n-1)}$ &  $(2n,3,4n,n,\dots,n)$  \\
$23$ &  & $\frac{g-12n+12}{12(n-1)}$  & $(2n,3n,4n,n,\dots,n)$  \\
\hline \hline
$24$ &  & $\frac{g+n-1}{30(n-1)}$ &  $(2,3,5,n,\dots,n)$  \\
$25$ &  & $\frac{g-5n+5}{30(n-1)}$ &  $(2,3,5n,n,\dots,n)$ \\
$26$ &  & $\frac{g-15n+15}{30(n-1)}$  & $(2,3n,5n,n,\dots,n)$  \\
$27$ &  & $\frac{g-9n+9}{30(n-1)}$ &  $(2,3n,5,n,\dots,n)$  \\
$28$ &  $A_5$ & $\frac{g-14n+14}{30(n-1)}$ &  $(2n,3,5,n,\dots,n)$  \\
$29$ &  & $\frac{g-20n+20}{30(n-1)}$ &  $(2n,3,5n,n,\dots,n)$  \\
$30$ &  & $\frac{g-24n+24}{30(n-1)}$ &  $(2n,3n,5,n,\dots,n)$  \\
$31$ &  & $\frac{g-30n+30}{30(n-1)}$ &  $(2n,3n,5n,n,\dots,n)$  \\
\hline \hline
$32$ &  & $\frac{2g+2n-2}{p^t(n-1)}-2$ &  $(p^t,n,\dots,n)$  \\
$33$ & $U$ & $\frac{2g+np^{t}-p^t}{p^t(n-1)}-2$ &   $(np^t,n,\dots,n)$ \\
\hline \hline
$34$ &  & $\frac{2(g+n-1)}{mp^t(n-1)}-1$ &   $(mp^t,m,n,\dots,n)$  \\
$35$ &  & $\frac{2g+2n+p^t-np^t-2}{mp^t(n-1)}-1$ &  $(mp^t,nm,n,\dots,n)$  \\
$36$ & $K_m$ & $\frac{2g+np^t-p^{t}}{mp^t(n-1)}-1$ &   $(nmp^t,m,n,\dots,n)$\\
$37$ &  & $\frac{2g}{mp^t(n-1)}-1$ &   $(nmp^t,nm,n,\dots,n)$ \\
\hline \hline
$38$& & $\frac{2(g+n-1)}{m(n-1)}-1$ &   $(\alpha,\beta,n,\dots,n)$  \\
$39$& $\psl_2(q)$  & $\frac{2g+q(q-1)-n(q+1)(q-2)-2}{m(n-1)}-1$ &   $(\alpha,n\beta,n,\dots,n)$  \\
$40$& & $\frac{2g+nq(q-1)+q-q^2}{m(n-1)}-1$ &   $(n\alpha,\beta,n,\dots,n)$  \\
$41$& & $\frac{2g}{m(n-1)}-1$ &   $(n\alpha,n\beta,n,\dots,n)$  \\
\hline \hline
$42$& & $\frac{2(g+n-1)}{m(n-1)}-1$ &   $(2\alpha,2\beta,n,\dots,n)$  \\
$43$& $\pgl_2(q)$  & $\frac{2g+q(q-1)-n(q+1)(q-2)-2}{m(n-1)}-1$ &   $(2\alpha,2n\beta,n,\dots,n)$  \\
$44$& & $\frac{2g+nq(q-1)+q-q^2}{m(n-1)}-1$ &   $(2n\alpha,2\beta,n,\dots,n)$ \\
$45$& & $\frac{2g}{m(n-1)}-1$ &   $(2n\alpha,2n\beta,n,\dots,n)$ \\
\hline \hline
\caption{The signature $\bC$ and dimension $\d$ for $\chara >5 $} 
\label{t2}
\label{long-table}
\end{longtable}

\end{footnotesize}

\vspace{-1cm}

Next  we can   complete the classification of automorphism groups of superelliptic curves defined over any algebraically closed  field of characteristic $\chara k > 2$.

\begin{thm}[\cite{Sa}]  \label{th14}
Let $\X_g$ be an irreducible cyclic curve of genus $g\geq2$, defined over an algebraically closed field $k$, $\ch (k)=p\neq2$, $G=Aut(\X_g)$, $\G$ its reduced automorphism group.
\begin{enumerate}
\item If $\G \cong C_m$ then $G \cong C_{mn}$ or
\begin{center}
$\left\langle \r, \s \right|\r^n=1,\s^m=1,\s\r\s^{-1}=\r^l \rangle$
\end{center}
where (l,n)=1 and $l^m\equiv 1$ (mod n).

\item If $\G \cong D_{2m}$ then $G \cong D_{2m} \times C_n$ or
\begin{align*}
\begin{split}
G_{5}=& \left\langle \r, \s, \t \right|\r^n=1,\s^2=\r,\t^2=1,(\s\t)^m=1,\s\r\s^{-1}=\r,\t\r\t^{-1}=\r^{n-1} \rangle\\
G_{6}=& D_{2mn} \\
G_{7}=& \left\langle \r, \s, \t \right|\r^n=1,\s^2=\r,\t^2=\r^{n-1},(\s\t)^m=1,\s\r\s^{-1}=\r,\t\r\t^{-1}=\r \rangle\\
G_{8}=& \left\langle \r, \s, \t \right|\r^n=1,\s^2=\r,\t^2=1,(\s\t)^m=\r^{\frac{n}{2}},\s\r\s^{-1}=\r,\t\r\t^{-1}=\r^{n-1} \rangle \\
G_{9}=& \left\langle \r, \s, \t \right|\r^n=1,\s^2=\r,\t^2=\r^{n-1},(\s\t)^m=\r^{\frac{n}{2}},\s\r\s^{-1}=\r,\t\r\t^{-1}=\r \rangle\end{split}
\end{align*}

\item If $\G \cong A_4$ and $p\neq3$ then $G \cong A_4 \times C_n$ or
\begin{align*}
\begin{split}
G'_{10}=& \left\langle \r, \s, \t \right|\r^n=1,\s^2=1,\t^3=1,(\s\t)^3=1,\s\r\s^{-1}=\r,\t\r\t^{-1}=\r^l \rangle\\
G'_{12}=& \left\langle \r, \s, \t \right|\r^n=1,\s^2=1,\t^3=\r^{\frac{n}{3}},(\s\t)^3=\r^{\frac{n}{3}},\s\r\s^{-1}=\r,\t\r\t^{-1}=\r^l \rangle\\
\end{split}
\end{align*}
where $(l,n)=1$ and $l^3\equiv 1$ (mod n) or
\begin{center}
$\left\langle \r, \s, \t \right|\r^n=1,\s^2=\r^{\frac{n}{2}},\t^3=\r^{\frac{n}{2}},(\s\t)^5=\r^{\frac{n}{2}},\s\r\s^{-1}=\r,\t\r\t^{-1}=\r \rangle $
\end{center}
or
\begin{align*}
\begin{split}
G_{10}=& \left\langle \r, \s, \t \right|\r^n=1,\s^2=1,\t^3=1,(\s\t)^3=1,\s\r\s^{-1}=\r,\t\r\t^{-1}=\r^k \rangle\\
G_{13}=& \left\langle \r, \s, \t \right|\r^n=1,\s^2=\r^{\frac{n}{2}},\t^3=1,(\s\t)^3=1,\s\r\s^{-1}=\r,\t\r\t^{-1}=\r^k \rangle\\
\end{split}
\end{align*}
where $(k,n)=1$ and $k^3\equiv 1$ (mod n).

\item If $\G \cong S_4$ and $p\neq3$ then $G \cong S_4 \times C_n$ or
\begin{align*}
\begin{split}
G_{16}=& \left\langle \r, \s, \t \right|\r^n=1,\s^2=1,\t^3=1,(\s\t)^4=1,\s\r\s^{-1}=\r^l,\t\r\t^{-1}=\r \rangle\\
G_{18}=& \left\langle \r, \s, \t \right|\r^n=1,\s^2=1,\t^3=1,(\s\t)^4=\r^{\frac{n}{2}},\s\r\s^{-1}=\r^l,\t\r\t^{-1}=\r \rangle\\
G_{20}=& \left\langle \r, \s, \t \right|\r^n=1,\s^2=\r^{\frac{n}{2}},\t^3=1,(\s\t)^4=1,\s\r\s^{-1}=\r^l,\t\r\t^{-1}=\r \rangle \\
G_{22}=& \left\langle \r, \s, \t \right|\r^n=1,\s^2=\r^{\frac{n}{2}},\t^3=1,(\s\t)^4=\r^{\frac{n}{2}},\s\r\s^{-1}=\r^l,\t\r\t^{-1}=\r \rangle\\
\end{split}
\end{align*}
where $(l,n)=1$ and $l^2\equiv 1$ (mod n).

\item If $\G \cong A_5$ and $p\neq5$ then $G \cong A_{5}\times C_{n}$ or
\begin{center}
$\left\langle \r, \s, \t \right|\r^n=1,\s^2=\r^{\frac{n}{2}},\t^3=\r^{\frac{n}{2}},(\s\t)^5=\r^{\frac{n}{2}},\s\r\s^{-1}=\r,\t\r\t^{-1}=\r \rangle $
\end{center}

\item If $\G \cong U$ then $G \cong U \times C_n$ or
\begin{multline*}
<\r,\s_1,\s_2,\dots,\s_t|\r^n=\s_1^p=\s_2^p=\dots=\s_t^p=1,\\ \s_i\s_j=\s_j\s_i, \s_i\r\s_i^{-1}=\r^{l}, 1\leq i,j\leq t>
\end{multline*}
where $(l,n)=1$ and $l^p \equiv 1$ (mod n).

\item If $\G \cong K_m$ then $G \cong$
\begin{multline*}
<\r,\s_1,\dots,\s_t,v|\r^n=\s_1^p=\dots=\s_t^p=v^m=1, \s_i\s_j=\s_j\s_i,\\
v\r v^{-1}=\r, \s_i\r\s_i^{-1}=\r^{l}, \s_iv\s_i^{-1}=v^{k}, 1\leq i,j\leq t >
\end{multline*}
where $(l,n)=1$ and $l^p \equiv 1$ (mod n), $(k,m)=1$ and $k^p \equiv 1$ (mod m) or
\begin{align*}
\begin{split}
\left\langle \r,\s_1,\dots,\s_t|\r^{nm}=\s_1^p=\dots=\s_t^p=1, \s_i\s_j=\s_j\s_i, \s_i\r\s_i^{-1}=\r^{l}, 1\leq i,j\leq t\right\rangle
\end{split}
\end{align*}
where $(l,nm)=1$ and $l^p \equiv 1$ (mod nm).

\item If $\G \cong \psl_{2}(q)$ then $G\cong \psl_2(q)\times C_n$ or $SL_2(3)$.

\item If $\G \cong \pgl_2(q)$ then $G \cong \pgl_2(q) \times C_n$.
\end{enumerate}
\end{thm}

Applying the above theorem  we can obtain the automorphism groups of a genus 3 superelliptic curves  defined over algebraically closed field of characteristic $p \neq 2$. Below we list the GAP group ID's of those groups.

\begin{lem}\label{thm_g_3}
Let $\X_g$ be a genus 3 superelliptic curve  defined over a field of characteristic $p \neq 2$. Then the automorphism groups of $\X_g$ are as follows.
\begin{description}
\item[i)] $p=3$: $(2,1)$, $(4,2)$, $(3,1)$, $(4,1)$, $(8,2)$, $(8,3)$, $(7,1)$, $(14,2)$, $(6,2)$, $(8,1)$, $(8,5)$, $(16,11)$, $(16,10)$, $(32,9)$, $(30,2)$, $(16,7)$, $(16,8)$, $(6,2)$.
\item[ii)] $p=5$: $(2,1)$, $(4,2)$, $(3,1)$, $(4,1)$, $(8,2)$, $(8,3)$, $(7,1)$, $(21,1)$, $(14,2)$, $(6,2)$, $(12,2)$,
    $(9,1)$, $(8,1)$, $(8,5)$, $(16,11)$, $(16,10)$, $(32,9)$, $(42,3)$, $(12,4)$, $(16,7)$, $(24,5)$, $(18,3)$,
    $(16,8)$, $(48,33)$, $(48,48)$.
    \item[iii)] $p=7$: $(2,1)$, $(4,2)$, $(3,1)$, $(4,1)$, $(8,2)$, $(8,3)$, $(7,1)$, $(21,1)$, $(6,2)$, $(12,2)$,
    $(9,1)$, $(8,1)$, $(8,5)$, $(16,11)$, $(16,10)$, $(32,9)$, $(30,2)$, $(42,3)$, $(12,4)$, $(16,7)$, $(24,5)$, $(18,3)$,
    $(16,8)$, $(48,33)$, $(48,48)$.
\item[iv)] $p=0$ or $p > 7$: $(2,1)$, $(4,2)$, $(3,1)$, $(4,1)$, $(8,2)$, $(14,2)$, $(6,2)$,  
    $(9,1)$, $(8,5)$, $(16,11)$,  $(32,9)$,   $(12,4)$, $(16,13)$, $(24,5)$, 
      $(48,33)$, $(48,48)$, (96, 64).
\end{description}
\end{lem}

Recall that the list for  $p=0$ is the same as for $p > 7$.   While the above result seems rather technical it can be used very effectively to write down the complete list of automorphism  groups for all superelliptic curves for any given $g\geq 2$.  Such lists were compiled for all $2 \leq g \leq 10$ in \cite{MPRZ}.

\subsection{Weierstrass points of  superelliptic curves} 

\def\br{\mathfrak b}

Most of this section is summarizing the results in \cite{shor-shaska-1} and \cite{shor-shaska-2}.
Let  $\X_g$ be a smooth superelliptic curve given by an affine equation  $y^n=f(x)$ with $n\geq2$ and $f(x)\in k[x]$.  Since  we are assuming that $\X_g$ is smooth, then $f(x)$ is a separable polynomial of degree $\deg f = d >n$.  Hence, $\D_f \neq 0$. Consider the following.

\begin{prob}
Determine all the $q$-Weierstrass points superelliptic curves  $y^n= f(x)$. 
\end{prob}

Let $\{\alpha_1, \alpha_2, \dots, \alpha_d\}$ denote the $d$ distinct roots of $f(x)$, and for each $i$ let $\br_i=(\alpha_i,0)$ be an affine branch point of the cover $\phi: \X_g\to\P^1 (k)$.  %
For any   $c \in \P^1 (k)$ let   $P^c_1, \dots, P^c_n$ denote the $n$ points in the fiber  $\phi(c)$.  
Let $r=\gcd(n,d)$.  All points on this model of the curve are smooth except possibly the point at infinity, which is singular when $d>n+1$.  In a smooth model for the curve, the point at infinity splits into $r$ points which we denote $P^\infty_1, \dots, P^\infty_r$.
One then has the following divisors:
\begin{itemize}
\item $(x-c) = \displaystyle\sum_{j=1}^n  P^c_j - \dfrac{n}{r}\sum_{m=1}^r  P^\infty_m,$
\item $(x-\alpha_i) = \displaystyle n \br_i - \dfrac{n}{r}\sum_{m=1}^r  P^\infty_m,$
\item $(y) = \displaystyle\sum_{j=1}^d \br_j - \dfrac{d}{r}\sum_{m=1}^r  P^\infty_m,$
\item $(dx) = (n-1)\displaystyle\sum_{j=1}^d \br_j - \left(\dfrac{n}{r}+1\right)\sum_{m=1}^r  P^\infty_m.$
\end{itemize}
Since $(dx)$ is a canonical divisor and hence has degree $2g-2$, we find the genus $g$ of $\X_g$ is given by 
\[
2g-2=nd-n-d-\gcd(n,d).
\]  
In particular, if $n$ and $d$ are relatively prime, then we obtain $g=\dfrac{(n-1)(d-1)}{2}$. 
\begin{lem}
For a curve $\X_g$  given by an affine equation $y^n=f(x)$, with $f(x)$ separable of degree $d$ and $g>1$, we have $g\geq n$ with equality only when $(n,d)=(2,5), (2,6),$ or $(3,4)$.
\end{lem}

\proof
One can check that if $(n,d)=(2,5), (2,6),$ or $(3,4)$, then $g=n$.
If $n=2$ and $d\geq7$, then 
\[ 
g=  \frac {d-\gcd(d,2)}  2     \geq 3 >n.
\]
If $n=3$ and $d\geq 5$, then 
\[ 
g= \frac {2d-1-\gcd(d,3)}  2 \geq 4 > n.
\]
If $n\geq4$, then $d\geq5$, and so 
\[ 2g=(n-1)(d-1)-\gcd(n,d)+1\geq (n-1)(d-2) \geq 3(n-1).\]
Thus, $g\geq \frac{3}{2}(n-1)$, which is larger than $n$ for $n>3$.  
\qed

To construct a basis of $H^0(\X_g,(\Omega^1)^q)$, we first note that 
\[ \left(\displaystyle\frac{dx}{y^{n-1}}\right)=  \dfrac{2g-2}{r}    \displaystyle\sum_{m=1}^r P^\infty_m.\]
Fix $\alpha_i$ and $q\geq 1$; for any $a, b\in \Z$ and we also let 
\[  
h_{a, b, q}(x,y)= (x-\alpha_i)^a  \,  y^b   \left(\frac{dx} {y^{n-1}}\right)^q.
\]
  Then, the divisor of $h_{a,b,q}$ is given by
\[\left( h_{a,b,q}(x,y)\right) = an  \br_i + b\sum_{j=1}^d  \br_j+\frac{(2g-2)q - a n - b d}{r} \sum_{m=1}^r P_m^\infty.\]  
In particular, this divisor is effective precisely when $a\geq 0$, $b\geq 0$, and $an+bd\leq (2g-2) q.$  Since $y^n=f(x)$, the functions $h_{a,b,q}(x,y)$ are linearly independent if we assume $a\geq0$ and $0\leq b<n$.

Define the set 
\[S_{n,d,q} : = \{(a,b)\in\Z^2 : a\geq0,\, 0\leq b<n,\, 0 \leq an+bd \leq (2g-2)q\}.\]  
A simple counting argument gives the following:
\begin{lem}
The set $S_{n,d,q}$ contains exactly $d_q$ distinct elements.
\end{lem}
From this set $S_{n,d,q}$, we obtain a basis 
\[ 
\mathfrak{B}_{q}=\{h_{a,b,q}(x,y) : (a,b)\in S_{n,d,q}\}.
\]
  Since we already have $\dim(H^0(\X_g,(\Omega^1)^q))=d_q$, we obtain the following:

\begin{thm}\label{super-basis}
For any root $\alpha_i$ and any $q\geq 1$, the set $\mathfrak{B}_{q}$ forms a basis of $H^0(\X_g,(\Omega^1)^q)$.
\end{thm}

The above result was proved in \cite{nato-3}*{Prop.~13}.     Next we have the following result:
\begin{prop}
Any affine branch point $\br_i$ is a $q$-Weierstrass point for all $q\geq 1$.
\end{prop}

\proof
One can calculate the $q$-Weierstrass weight of any branch point $\br_i=(\alpha_i,0)$ by calculating the order of vanishing of the basis elements at $\br_i$.  In particular, one checks that 
\[\ord_{\br_i}\left(h_{a,b,q}(x,y)\right) = an+b.\]
Since $0\leq b<n$, these valuations are all distinct non-negative numbers. Thus, we obtain for the $q$-Weierstrass weight of the point $\br_i=(\alpha_i,0)$ the following
\[w^{(q)}(\br_i) = \sum_{(a,b)\in S_{n,d,q}} (an+b+1) - \sum_{m=1}^{d_q}m.\]
Thus, this formula shows that $w^{(q)}(\br_i)>0$ for any $q$.
\qed


Determining Weierstrass points gives a Weierstrass equation for hyperelliptic curves. The above results seem to suggest that the same can be done for superelliptic curves.

Next, we leave the reader with a problem of using the information on Weierstrass points to determine if the curve is superelliptic.  As far as we are aware, this is still an open problem.

\begin{prob}
Given an irreducible algebraic curve $\X$ with affine equation $F(x, y)=0$, find an algorithm which determines whether $\X$ is superelliptic. 
\end{prob}

A further discussion of this problem is intended in \cite{sevilla-shaska}.  Moreover, using the approach in \cite{issac} and \cite{Sa-thesis} this would determine the full automorphism group of superelliptic curves.

\section{Moduli space   of curves and superelliptic loci}\label{sect-6}
\subsection{Moduli space of curves}
Let $\M_{g}$ be the moduli space of smooth, projective curves of genus $g$, and $\M_{g_0,r}$ the moduli space of genus-$g_0$ curves
with $r$ distinct marked points, where we view the marked points as unordered. The term \textbf{space} here refers to a Deligne-Mumford stack (in algebraic geometry) or orbifold (in an analytic setting). We will focus on the latter notion to describe the moduli space.

To explain this in more detail, we will first define $\M_{g_0,r}$ as a set and then endow this set with the structure of a smooth, complex, $n=3g_0-3+r$-dimensional orbifold that is locally an open ball in $\mathbb{C}^n$ divided by a finite group action. \emph{As a set}, we define $\M_{g_0,r}$ to be the set of isomorphism classes of smooth, projective curves of genus $g_0$ with $r$ marked points.

Here, we must insist that $2-2g_0-r<0$, since only the group of marked-points-preserving automorphisms for a smooth algebraic curve satisfying $2-2g_0-r<0$ is finite. On the other hand, every algebraic curve with $2-2g_0-r\ge 0$ has an infinite group of marked-points-preserving automorphisms, which makes it impossible to define the moduli spaces $\M_{0,0}$, $\M_{0,1}$, $\M_{0,2}$, and $\M_{1,0}$ as orbifolds. The difficulty with viewing moduli spaces $\M_{g_0,r}$ only as sets is easily observed in the following example: as we have seen, a genus-two curve is uniquely defining by six distinct unordered points on a rational curve, i.e., its Weierstrass points. Thus, we have -- on the level of sets -- $\M_{2,0}=\M_{0,6}/S_6$ where $S_6$ is the symmetric group in six elements. However, any meaningful notion of moduli space should distinguish $\M_{2,0}$ and $\M_{0,6}/S_6$ since every genus-two curve carries an additional automorphism, i.e., the hyperelliptic involution, that the genus-zero curve with six marked points does not have.

The set $\M_{g_0,r}$ with $2-2g_0-r<0$ can be endowed with the structure of a smooth complex $3g_0-3+r$-dimensional \textbf{orbifold}, that is, $\M_{g_0,r}$ can be covered by a family of compatible charts such that the stabilizer of any point in $\M_{g_0,r}$ is the automorphism group of the corresponding algebraic curves of genus $g_0$ with $r$ marked points. In the aforementioned example, the moduli spaces $\M_{2,0}$ and $\M_{0,6}/S_6$, though equal as sets, then have different orbifold structures, and as orbifolds are isomorphic only up to a $\mathbb{Z}/2\mathbb{Z}$ action. This is based on the following theorem [citation needed]:

\begin{thm}
Given any smooth projective genus-$g_0$ curve $\X$ with $r$ marked points, and finite automorphism group $G$, there exists an open, bounded, simply connected domain $U \subset \mathbb{C}^{3g_0-3+r}$, a family $p : \X' \to U$ of smooth projective genus-$g_0$ curves with $r$ marked points, and an action of the group $G$ on $\X'$ commuting with $p$, satisfying the following conditions: (1) the central fiber $\X'_0$ over $0 \in U$ is isomorphic to $\X$ , i.e., $\X'_0 \cong \X$, (2) the action of $G$ preserves $\X'_0$ and coincides with the natural action of $G$ on $\X$, and (3) any other family of smooth projective genus-$g_0$ curves with $r$ marked points and central fiber $\X$ is the pull-back of the family $p : \X'_t \to U$ (after suitable restriction).
\end{thm}

In other words, $\M_{g_0,r}$ is a smooth, complex $3g_0-3+r$-dimensional orbifold and is covered by charts of the form $U/G$ such that the stabilizer of $[\X] \in \M_{g_0,r}$ is isomorphic to the symmetry group of the surface $\X'$. Moreover, the theorem also yields the construction a second smooth orbifold $\mathcal{N}_{g_0,r}$ that is covered by (suitable subdivisions of) the open sets $\{ \X'\}$, and an induced orbifold morphism $p:\mathcal{N}_{g_0,r} \to \M_{g_0,r}$ between them, called the \textbf{universal curve} over $\M_{g_0,r}$. The fibers of the universal curve are smooth, projective genus-$g_0$ curves with $r$ marked points, such that each curve appears exactly once among the fibers.

The moduli space $\M_{g_0,r}$ is, in general, not compact. We now compactify it by adding new points that correspond to so-called \textbf{stable curves}. A curve singularity $(\X,p)$ is called a node if locally the singularity $P \in \X$ is isomorphic to the plane curve singularity $xy = 0$.  Thus, we think of the neighborhood of a node as isomorphic to two discs with identified centers. A curve $\X$ is called nodal if the only singularities of $\X$ are nodes. There are two different ways of desingularizing curves. In our situation, a node can be \textbf{resolved} by replacing the two discs with identified centers that form its neighborhood by a cylinder.  On the other hand, we say that a node is \textbf{normalized} if the two discs with identified centers are unglued, i.e., replaced by disjoint discs. The concept of normalization is based on the algebraic construction of the normalization of the coordinate ring of $\X$. However, given any affine variety $X$, one can always construct the normalization $X^\nu$ along with a normalization morphism $\nu: X^\nu \to X$ explicitly. To do so, one constructs the normalization for each affine open chart of $X$, and shows that they glue together. In fact, in the case of a curve, the integral closure of the coordinate ring inside the function field can entirely be studied locally, since the integral closure of the coordinate ring of an algebraic curve is broken only at singular points, i.e., in our situation the nodes. In this way, the normalization of a nodal curve is the curve obtained by normalizing all its nodes $P_i \in \X$. It is smooth, but not necessarily connected. The arithmetic genus of a nodal curve is the genus of the curve obtained by resolving all its nodes. We make the following:

\begin{defi}
A \textbf{stable curve} $\X$ with $r$ marked points is a connected, complete, projective curve of arithmetic genus $g_0$ satisfying the following conditions: (1) the only singularities of $\X$ are nodes, i.e., the curve is nodal, (2) the marked points are distinct and do not coincide with any nodes, (3) the curve $\X$ has a finite number of marked-points-preserving automorphisms.
\end{defi}

To be able to check the conditions of this definition, in particular reformulate condition (3) in a way that is checked easily, one uses the dualizing sheaf of $\X$. If $\X$ is a nodal connected curve of arithmetic genus $g_0$, the dualizing sheaf $\omega_{\X}$\footnote{For a normal projective variety $\X$, the dualizing sheaf exists and it is in fact the canonical sheaf, i.e., $\omega_{\X} = \mathcal{O}_{\X}(K_{\X})$ where $K_{\X}$ is a canonical divisor.} is an invertible sheaf of degree $2g_0-2$ and $h^0(\X,\omega_{\X})=g_0$. It can be described explicitly: let $\X$ be a connected curve of arithmetic genus $g_0$ with just one node at $P \in \X$ and $\nu: \X^\nu \to \X$ the normalization with $\{r,s\} = \nu^{-1}(P)$. Then, $\omega_{\X}$ is the sheaf that associates to any open subset $V \subset \X$ the rational differentials $\eta$ on $\nu^{-1}(V)$ having at worst simple poles at $r,s$ such that $\operatorname{Res}_r(\eta) + \operatorname{Res}_s(\eta)=0$. For a connected, complete, nodal curve (with nodes $\{P_i\}$) of arithmetic genus $g_0\ge 2$  the following three conditions are equivalent:
\begin{enumerate}
 \item $\omega_{\X}( \sum P_i)$ is ample,
 \item If $\X_i^{\nu}$ is a genus-zero component of the normalization of $\X$, then $\X_i^{\nu}$ has at least three points mapped  by $\nu$ to nodes or marked points of $\X$.
 \item The group of marked points preserving automorphisms of $\X$ is finite.
\end{enumerate}
An immediate consequence is the following: if $\X_i^{\nu}$ are the connected, genus-$g_i$ components of the normalization of $\X$, and $n_i$ the number of marked points plus the number of preimages of nodes on the component $\X_i^{\nu}$, then Condition (3) in the above definition is satisfied if and only if $2-2g_i-n_i<0$ for all $i$.

The following theorem is essential  [citation needed]:
\begin{thm}
There exist \emph{compact}, smooth, complex orbifolds $\overline{\M}_{g_0,r}$ of dimension $3g_0-3+r$ and $\overline{\mathcal{N}}_{g_0,r}$ of dimension $3g_0-2+r$, and an orbifold morphism $\bar{p}:\overline{\mathcal{N}}_{g_0,r} \to \overline{\M}_{g_0,r}$ such that (1) $\M_{g_0,r} \subset \overline{\M}_{g_0,r}$ and $\mathcal{N}_{g_0,r} \subset \overline{\mathcal{N}}_{g_0,r}$ are open dense sub-orbifolds, (2) $\bar{p}$ restricts to $p$ on $\M_{g_0,r}$, $\bar{p}^{-1}(\overline{\M}_{g_0,r})=\overline{\mathcal{N}}_{g_0,r}$, and the fibers of $\bar{p}$ are stable curves of arithmetic genus $g_0$ with $r$ marked points, (3) each stable curve is isomorphic to exactly one fiber of $\bar{p}$, and (4) the stabilizer of a point $[\X] \in \overline{\M}_{g_0,r}$ is the automorphism group of the corresponding stable curve $\X$.
\end{thm}

We make the following:
\begin{defi}
The space $\overline{\M}_{g_0,r}$ is called the Deligne-Mumford compactification of the moduli space $\M_{g_0,r}$. The family $\bar{p}:\overline{\mathcal{N}}_{g_0,r} \to \overline{\M}_{g_0,r}$ is called the universal curve over $\overline{\M}_{g_0,r}$.
\end{defi}

Notice that $\overline{\M}_{g_0,r}$ is a \emph{smooth and compact} orbifold. The set $\overline{\M}_{g_0,r} \backslash \M_{g_0,r}$ is called the \textbf{boundary} of $\overline{\M}_{g_0,r}$ and parametrizes singular stable curves. The boundary is a sub-orbifold of codimension 1, whence given by a divisor.  A generic point of the boundary is a stable curve with one node. If a point of the boundary corresponds to a stable curve $\X$ with $k$ nodes, that means that there are $k$ local components of the boundary that intersect transversally, and this is the only way local components can intersect. Therefore, the boundary is a so-called \textbf{normal crossing divisor}.

\subsection{Curves with automorphisms in the moduli space}

\def\c{{\bf c}}
\def\dd{{\bf d}}
\def\C{{\bf C}}
\def\D{\Delta}

Fix the genus $g \geq 2$.  Consider the following problem.

\begin{prob}
Could one list all groups which occur as a full automorphism group of a genus $g$ smooth, irreducible algebraic  curve $\X$ defined over a field $k$ of characteristic $\ch (k) = p \geq 0$? 
\end{prob}

 In the previous section we were able to do this for all superelliptic curves for all genera and $\ch k  \neq 2$.  The case of $\ch k =2$ is more technical and we avoid it here.  However, there are plenty of curves which are not superelliptic.  The generic curve of genus three, for example, has equation isomorphic to a ternary quartic and is not a superelliptic curve.  The classification of automorphism groups is still an open problem for $\ch k = p>0$, but it can be done in $\ch k = 0$ due to results of the last two decades by Breuer, Magaard, Shaska, Shpectorov, Volklein.  We summarize these results briefly below. 

Recall that a group $G$ acts faithfully on a genus $g$ curve if and only if it has a genus $g$ generating system; see \cite{kyoto}. For $g$ up to 48, all such groups and the signatures of all their genus-$g$ generating systems have been listed by Breuer \cite{Breuer}. More precisely, for each genus $g\le48$, he produced a list containing all {\bf signature-group pairs} in genus $g$, i.e., pairs consisting of a group $G$ together with the signature of a genus $g$ generating system of $G$.

If $G$ acts on $X_g$ then so does each subgroup of $G$. This shows that Breuer's lists have to be long, and contain some redundancies. Th work in \cite{kyoto} eliminates those signature-group pairs that do not yield the full automorphism group of a curve.  It turns out that the larger $g$ is, the larger the ratio is of entries in Breuer's lists that do occur as full automorphism group in genus $g$. This can already be seen from the fact  that if a signature-group pair does not yield the full automorphism group of a curve, then its $\delta$-invariant (dimension of corresponding locus in $\M_g$) is at most 3.

For small genus $g$, a relatively large portion of those groups do not occur as full automorphism group in genus $g$. Among those that do occur, we distinguish those that occur for a particularly simple class of curves: we call a group homocyclic if it is a direct product of isomorphic cyclic groups. 
%

\subsection{Ramification type and signature of a $G$-curve}
Fix an integer $g\ge2$ and a finite group $G$. Let $C_1,\dots , C_r$ be conjugacy classes $\ne\{1\}$ of $G$. Let $\C=(C_1,\dots ,C_r)$ be an unordered tuple, where repetitions are allowed. We also allow $r$ to be zero, in which case $\C$ is empty.
Consider pairs $(X,\mu)$, where $X$ is a curve and $\mu: G\to\Aut(X)$ is an injective homomorphism. We will often suppress $\mu$ and just say $X$ is a curve with $G$-action, or a $G$-curve, for short. Two $G$-curves $X$ and $X'$ are called equivalent if there is a $G$-equivariant isomorphism $X\to X'$.

We say a $G$-curve $X$ is {\bf of ramification type} $(g,G,\C)$ if the following holds: the curve $X$ has genus $g$, the points of the quotient $X/G$ that are ramified in the cover $X\to X/G$ can be labelled as $p_1,\dots ,p_r$ such that $C_i$ is the conjugacy class in $G$ of distinguished inertia group generators over $p_i$ (for $i=1,\dots ,r$). (Distinguished inertia group generator means the generator acts in the tangent space as multiplication by $\exp(2\pi\sqrt{-1}/e)$, where $e$ is the ramification index). For short, we will just say $X$ is of type $(g,G,\C)$.

If $X$ is a $G$-curve of type $(g,G,\C)$ then the genus $g_0$ of $X/G$ is given by the Riemann-Hurwitz formula
\[
\frac{2\ (g-1)}{|G|}\ \ \ = \ \ \ 2\ (g_0-1)\ +\ \sum_{i=1}^r\ \left(1-\frac{1}{c_i}\right),
\]
where $c_i$ is the order of the elements in $C_i$. Note that $g_0$ (the {\bf orbit genus}) depends only on $g$, $|G|$ and the {\bf signature} $\c=(c_1,\ldots,c_r)$ of the $G$-curve $X$.

\subsection{Hurwitz spaces and moduli of curves}\label{section2}

Define $\H=\H(g,G,\C)$ to be the set of equivalence classes of $G$-curves of type $(g,G,\C)$. By covering
space theory (or the theory of Fuchsian groups), $\H$ is non-empty if and only if $G$ can be generated by
elements $\a_1,\b_1,\dots ,\a_{g_0},\b_{g_0},\g_1,\dots ,\g_r$ with $\g_i\in C_i$ and
$$\prod_{j}\ [\a_j,\b_j]\ \ \prod_{i}\ \g_i\ \ \ \
= \ \ \ 1 \leqno{(2)} \,.$$ Here $[\a,\b]=\ \a^{-1}\b^{-1}\a\b$.
Consider the map
\[\Phi:\ \H\ \to \ \M_{g} \,,\]
obtained by forgetting the $G$-action, and the map   $\Psi:\ \H\ \to \ \M_{g_0,r} $
mapping (the class of) a $G$-curve $X$ to the class of the quotient curve $X/G$ together with the (unordered)
set of branch points $p_1,\dots ,p_r$. If $\H\ne\emptyset$ then $\Psi$ is surjective and has finite fibers, by
covering space theory. Also $\Phi$ has finite fibers, since the automorphism group of a curve of genus $\ge2$
is finite.

The set $\H$ carries the structure of a quasi-projective variety (over $\bC$) such that the maps
$\Phi$ and $\Psi$ are finite morphisms. If $\H\ne\emptyset$ then all components of $\H$ map surjectively to
$\M_{g_0,r}$ (through a finite map), hence they all have the same dimension
\[ \d(g,G,\C):= \ \ \dim\ \M_{g_0,r} \ \ = \ \ 3g_0-3+ r.\]

\begin{lem} \label{Lemma1} Let $\M(g,G,\C)$ denote the image of $\Phi$,
i.e., the locus of genus $g$ curves admitting a $G$-action of type $(g,G,\C)$. If this locus is non-empty
then each of its components has dimension $\d(g,G,\C)$.
\end{lem}


\subsection{Restriction to a subgroup} \label{restr}

Let $H$ be a subgroup of $G$. Then each $G$-curve can be viewed as an $H$-curve by restriction of action. Let
$X$ be a $G$-curve of type $(g,G,\C)$. Then the resulting $H$-curve is of type $(g,H,\D)$, where $\D$ is
obtained as follows: Choose $\g_i\in C_i$ and let $\s_{i,1}, \s_{i,2},\dots $ be a set of representatives for
the double cosets $<\g_i>\s H$ in $G$. Let $m_{ij}$ be the smallest integer $\ge 1$ such that the element
$\s_{ij}^{-1}\g_{i}^{m_{ij}}\s_{ij}$ lies in $H$, and let $D_{ij}$ be the conjugacy class of this element in
$H$. Then $\D$ is the tuple consisting of all $D_{ij}$ with $D_{ij}\ne\{1\}$. (More precisely, the tuple $\D$
is indexed by the set of possible pairs $(i,j)$, and its $(i,j)$-entry is $D_{ij}$.) The definition of $\D$
does not depend on the choice of the $\g_i$ and $\s_{ij}$. Note that the signature of the $H$-curve depends
on the type of the $G$-curve, not only on its signature.   We have 
\[  \M(g,G,\C) \ \ \subset \ \ \M(g,H,\D) \,.\]
Hence, their dimensions satisfy
$ \d(g,G,\C)\ \ \ \le\ \ \ \d(g,H,\D)$. 
If this is a strict inequality then the complement of the closure of $ \M(g,G,\C)$ in $\M(g,H,\D)$ is open
and dense. In particular, it is not true that every $H$-curve of type $(g,H,\D)$ is the restriction of a
$G$-curve of type $(g,G,\C)$. 

\xymatrixrowsep{4ex}
\xymatrixcolsep{2.7ex}
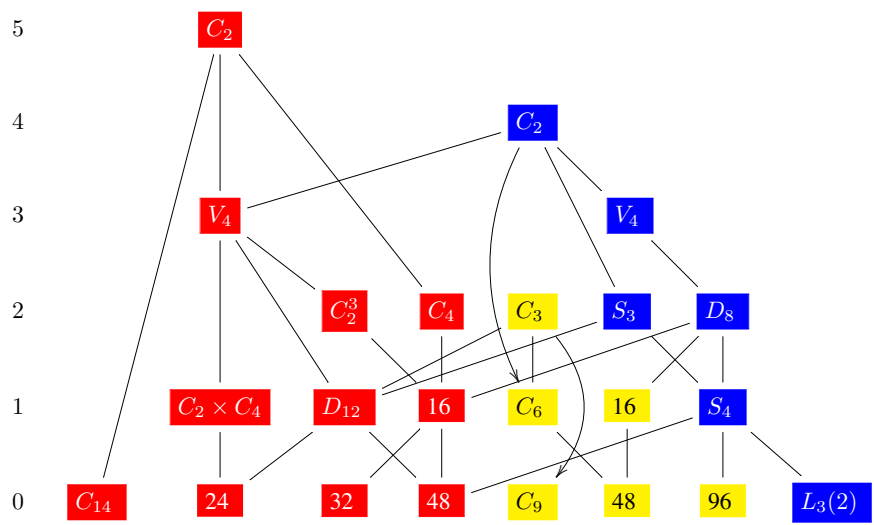
\begin{figure}[h] 
\[
\scalebox{.9}{
 \xymatrix{ 5 & &{\fcolorbox{red}{red}{\color{white}$C_2$}} \ar@{-}[dd]
\ar@{-}[dddddl] \ar@{-}[dddrr] & & & & & & \\
4 & & & & & {\fcolorbox{blue}{blue}{\color{white}$C_2$ }}
\ar@{-}[ddr] \ar@{-}[dlll] \ar@{-}[dr] \ar@/_1.5pc/[ddd] & & & \\
3 & & {\fcolorbox{red}{red}{\color{white}$V_4$}} \ar@{-}[dr] \ar@{-}[ddr] \ar@{-}[dd] & & & & \ {\fcolorbox{blue}{blue}{\color{white}$V_4$ }} \ar@{-}[dr] & & \\
2 & & &{\fcolorbox{red}{red}{\color{white}$C_2^3$}} \ar@{-}[dr] 
& {\fcolorbox{red}{red}{\color{white}$C_4$}} \ar@{-}[d] &
{\fcolorbox{yellow}{yellow}{\color{black}$C_3$ }}  \ar@{-}[d] \ar@{-}[dll] \ar@/^1.8pc/[dd] 
&  {\fcolorbox{blue}{blue}{\color{white}$S_3$ }}  \ar@{-}[dr] \ar@{-}[dlll] 
&  {\fcolorbox{blue}{blue}{\color{white}$D_8$ }}    \ar@{-}[d] \ar@{-}[dl] \ar@{-}[dlll]
& \\
1 & &{\fcolorbox{red}{red}{\color{white} $ C_2 \times C_4$}} \ar@{-}[d] &{\fcolorbox{red}{red}{\color{white}
$D_{12}$ }} \ar@{-}[dl]\ar@{-}[dr] & {\fcolorbox{red}{red}{\color{white} 16 }} \ar@{-}[dl] \ar@{-}[d] &
 {\fcolorbox{yellow}{yellow}{\color{black}$C_6$ }} \ar@{-}[dr] &
  {\fcolorbox{yellow}{yellow}{\color{black}16 }} \ar@{-}[d] &
   {\fcolorbox{blue}{blue}{\color{white}$S_4$ }} \ar@{-}[d] \ar@{-}[dr] \ar@{-}[dlll] \\
0 & {\fcolorbox{red}{red}{\color{white} $C_{14}$ }} & {\fcolorbox{red}{red}{\color{white} 24 }} &
{\fcolorbox{red}{red}{\color{white} 32 }} & {\fcolorbox{red}{red}{\color{white} 48 }} &
 {\fcolorbox{yellow}{yellow}{\color{black}$C_9$ }} &
  {\fcolorbox{yellow}{yellow}{\color{black}48 }} &
  {\fcolorbox{yellow}{yellow}{\color{black}96 }} &
    {\fcolorbox{blue}{blue}{\color{white}$L_3 (2)$ }} \\
} }
\]
\caption{Poset of Hurwitz loci for $\M_3$.}
   \label{loci-3}
\end{figure}

\subsubsection{The moduli space $\M_3$}
In \cite{kyoto} the inclusions among the loci in $\M_g$ with different automorphism groups and their dimension were determined.  We illustrate the inclusion and dimensions of the different loci in \cref{loci-3} for $g=3$.  The red cases represent hyperelliptic loci, and the yellow ones are superelliptic (non-hyperelliptic). Notice that from 23 cases only 6 are non-hyperelliptic. 


\subsubsection{The moduli space $\M_4$}
%
In \cref{tab_1}  we present all automorphism groups and their signatures $g=4$. 
Each one of the families above is an irreducible algebraic locus in $\M_4$.   Notice that there are 41 cases from which only 13 are non-superelliptic (colored in blue). 

\newpage



\begin{center}
\begin{longtable}{||l|l|l|l|l|l|l||}
\hline
\hline
$\#$ & dim &G& ID & sig & type & subcases \\
\hline \hline
1 & 0 & $S_5$ & (120,34) & 0-$(2, 4, 5)$ & 1 &  \\
2 & 0 &$C_3\times S_4$ &(72,42) & 0-$(2, 3, 12)$ & 3 &  \\
3 & 0 &         &(72,40) & 0-$(2, 4, 6)$ & 4 &  \\
4 & 0 & $V_{10}$ &(40,8) & 0-$(2, 4, 10)$ & 7 &  \\
5 & 0 & $C_6 \times S_3$ &(36,12) & 0-$(2, 6, 6)$ & 10 &  \\
6 & 0 & $U_8$    &(32,19) & 0-$(2, 4, 16)$ & 16 &  \\
7 & 0 & $SL_2(3)$ &(24,3) & 0-$(3, 4, 6)$ & 20 &  \\
8 & 0 & $C_{18}$  &(18,2) & 0-$(2, 9, 18)$ & 27 &  \\
9 & 0 & $C_{15}$ &(15,1) & 0-$(3, 5, 15)$ & 38 &  \\
10 & 0 & $C_{12}$  &(12,2) & 0-$(4, 6, 12)$ & 45 &  \\
11 & 0 & $C_{10}$  &(10,2) & 0-$(5, 10, 10)$ & 51 &  \\
12 & 1 & $S_3^2$&(36,10) & 0-$(2, 2, 2, 3)$ & 12 & 3 \\
13 & 1 & $S_4$&(24,12) & 0-$(2, 2, 2, 4)$ & 18 & 1, 2 \\
14 & 1 & $C_2\times D_5$ &(20,4) & 0-$(2, 2, 2, 5)$ & 21 & 4 \\
15 & 1 & $C_3\times S_3$ &(18,3) & 0-$(2, 2, 3, 3)$ & 30 & 2, 5 \\
16 & 1 & $D_8$           &(16,7) & 0-$(2, 2, 2, 8)$ & 35 & 6 \\
17 & 1 & $C_2\times C_6$ &(12,5) & 0-$(2, 2, 3, 6)$ & 46 & 2, 5 \\
18 & 1 & $C_2\times S_3$ &(12,4) & 0-$(2, 2, 3, 6)$ & 41 & 3 \\
19 & 1 & $A_4$       &(12,3) & 0-$(2, 3, 3, 3)$ & 43 & 2 \\
20 & 1 & $D_{10}$ &(10,1) & 0-$(2, 2, 5, 5)$ & 49 & 1 \\
21 & 1 & $Q_8$ &(8,4) & 0-$(2, 4, 4, 4)$ & 59 & 6, 7 \\
22 & 1 & $C_6$ &(6,2) & 0-$(2, 6, 6, 6)$ & 66 & 5, 10 \\
23 & 1 & $C_5$ &(5,1) & 0-$(5, 5, 5, 5)$ & 69 & 9, 11 \\
24 & 2 & $D_6$ &(12,4) & 0-$(2^{5})$ & 40 & 1, 5, 12 \\
25 & 2 & $D_4$ &(8,3) & 0-$(2^{4}, 4)$ & 57 & 3, 13 \\
26 & 2 & $D_4$ &(8,3) & 0-$(2^{4}, 4)$ & 56 & 4, 16 \\
27 & 2 & $C_6$ &(6,2) & 0-$(2^{3}, 3, 6)$ & 64 & 7, 8 \\
28 & 2 & $C_6$ &(6,2) & 0-$(2^{2}, 3^{3})$ & 65 & 15, 17 \\
29 & 2 & $S_3$ &(6,1) & 0-$(2^{2}, 3^{3})$ & 62 & 12, 18 \\
30 & 2 & $C_4$ &(4,1) & 0-$(2, 4^{4})$ & 77 & 10 \\
31 & 3 & $S_3$ &(6,1) & 0-$(2^{6})$ & 61 & 13, 15, 24 \\
32 & 3 & $V_4$ &(4,2) & 1-$(2, 2, 2)$ & 72 & 18, 19, 25 \\
33 & 3 & $C_4$ &(4,1) & 0-$(2^{4}, 4^{2})$ & 76 & 21, 26 \\
34 & 3 & $C_3$ &(3,1) & 0-$(3^{6})$ & 80 & 9, 28 \\
35 & 3 & $C_3$ &(3,1) & 0-$(3^{6})$ & 81 & 29 \\
36 & 3 & $C_3$  &(3,1) & 1-$(3, 3, 3)$ & 79 & 15, 19, 22, 27 \\
37 & 4 & $V_4$ &(4,2) & 0-$(2^{7})$ & 73 & 14, 26 \\
38 & 4 & $V_4$ &(4,2) & 0-$(2^{7})$ & 74 & 17, 24, 25 \\
39 & 5 & $C_2$ &(2,1) & 2-$(2, 2)$ & 82 & 11, 20, 29, 32, 37, 38 \\
40 & 6 & $C_2$ &(2,1) & 1-$(2^{6})$ & 83 & 22, 28, 30, 31, 38 \\
41 & 7 & $C_2$ &(2,1) & 0-$(2^{10})$ & 84 & 27, 33, 37 \\
\hline \hline 
\caption{Hurwitz loci of genus 4 curves}
\label{tab_1}
\end{longtable}

\end{center}





\begin{sideways}
\begin{minipage}{19.5cm}
\[
\scalebox{.95}{
\xymatrixrowsep{.7cm}
\xymatrixcolsep{1cm}
\xymatrix@C-8pt{
 7        &    & & {\fcolorbox{red}{red}{\color{white}41}} \ar@{-}[ddddl] \ar@{-}[ddd] \ar@{-}[dddddr]&    &    &    &    &    &    &    &              \\
 6        &    &    &    &    &    &    &    &      {\fcolorbox{blue}{blue}{\color{white}40}} \ar@{-}[dddddl] \ar@/^1.2pc/[dddd] \ar@{-}[dddrr] \ar@{-}[ddddrr]&    &    &      \\
 5        &    &    &    &    &    &
 {\fcolorbox{blue}{blue}{\color{white}39}} \ar@{-}[dddddl]   \ar@{-}[dlll] \ar@{-}[ddd] \ar@/_1.3pc/[ddddll] \ar@{-}[ddrrr] \ar@{-}[drrrrr]&    &    &    &    &     \\
 4        &    &    &
 {\fcolorbox{red}{red}{\color{white}37}} \ar@{-}[ddl]\ar@{-}[ddd] &     &    &    &    &    &    &    & {\fcolorbox{blue}{blue}{\color{white}38}}  \ar@{-}[dd]\ar@{-}[ddr] \ar@/^2.0pc/[dddll]  \\
 3        &    &
 {\fcolorbox{red}{red}{\color{white}33}} \ar@{-}[d] \ar@{-}[ddl] &    &    &    {\fcolorbox{blue}{blue}{\color{white}36}} \ar@{-}[dl] \ar@{-}[dd]\ar@/_1.5pc/[ddrr] \ar@/^1.6pc/[ddrrr]&    & {\fcolorbox{yellow}{yellow}{\color{black}35}}  \ar@{-}[dl] & {\fcolorbox{yellow}{yellow}{\color{black}34}}   \ar@{-}[drr] \ar@{-}[dddr]& {\fcolorbox{blue}{blue}{\color{white}32}}  \ar@{-}[drrr] \ar@{-}[ddllll]  \ar@{-}[ddlll] & {\fcolorbox{blue}{blue}{\color{white}31}}  \ar@{-}[dddr] \ar@{-}[dr] \ar@{-}[ddll]&      \\
2   &    &
  {\fcolorbox{red}{red}{\color{white}26}} \ar@{-}[d] &    &
     {\fcolorbox{red}{red}{\color{white}27}} \ar@/_1.5pc/[dd] \ar@{-}[ddlll]    &    &
     {\fcolorbox{yellow}{yellow}{\color{black}29}}  \ar@{-}[d]\ar@{-}[drrrrrr] &    & {\fcolorbox{yellow}{yellow}{\color{black}30}}  \ar@{-}[ddl] &    &
 {\fcolorbox{yellow}{yellow}{\color{black}28}}  \ar@{-}[dl] \ar@{-}[dll]& {\fcolorbox{blue}{blue}{\color{white}24}}  \ar@/^1.2pc/[dd] \ar@{-}[ddl] \ar@{-}[dr]& {\fcolorbox{blue}{blue}{\color{white}25}}  \ar@/_1.1pc/[dd]       \\
1        &
 {\fcolorbox{red}{red}{\color{white}21}} \ar@{-}[d] \ar@{-}[dr]&
 {\fcolorbox{red}{red}{\color{white}16}} \ar@{-}[d] &
 {\fcolorbox{red}{red}{\color{white}14}} \ar@{-}[d] &
 {\fcolorbox{blue}{blue}{\color{white}20}} \ar@{-}[drrrr] &
 {\fcolorbox{blue}{blue}{\color{white}19}} \ar@{-}[drrr] &
 {\fcolorbox{yellow}{yellow}{\color{black}18}} \ar@/_0.3pc/[drrrrrr] &
 {\fcolorbox{blue}{blue}{\color{white}22}} \ar@{-}[d]\ar@{-}[drrr] &
 {\fcolorbox{yellow}{yellow}{\color{black}15}} \ar@{-}[d]\ar@{-}[drr] &
 {\fcolorbox{yellow}{yellow}{\color{black}17}} \ar@{-}[dr]\ar@{-}[dl] &
 {\fcolorbox{blue}{blue}{\color{white}13}} \ar@{-}[d]\ar@{-}[dr] &
 {\fcolorbox{yellow}{yellow}{\color{black}23}} \ar@{-}[dllllll] \ar@{-}[dll]  &
 {\fcolorbox{yellow}{yellow}{\color{black}12}} \ar@{-}[d]&     \\
0    &
{\fcolorbox{red}{red}{\color{white}7}}  &
{\fcolorbox{red}{red}{\color{white}6}}  &
{\fcolorbox{red}{red}{\color{white}4}}  &
{\fcolorbox{red}{red}{\color{white}8}}    &
{\fcolorbox{yellow}{yellow}{\color{black}11}}  &    &
{\fcolorbox{yellow}{yellow}{\color{black}10}}      &
{\fcolorbox{yellow}{yellow}{\color{black}2}}   &
{\fcolorbox{yellow}{yellow}{\color{black}9}}    &
{\fcolorbox{yellow}{yellow}{\color{black}5}}   &
{\fcolorbox{blue}{blue}{\color{white}1}}   &
{\fcolorbox{yellow}{yellow}{\color{black}3}}      \\
}}
\]
\end{minipage}
\end{sideways}



\newpage

In \cite{kyoto} all large automorphism groups, i.e., $| G| > 4 (g-1)$ are displayed for all genera $g \leq 10$. It would be interesting to have some bounds on the ratio between non-superelliptic cases over the total number of cases.  At least for the hyperelliptic cases we can get some estimates.

For a fixed $g$ we denote by $N_g$ the number of groups that occur as automorphism groups of genus $g$
curves. We would like to determine what happens to $N_g$ as $g$ increases.

Let $n \in \Z$ such that $n=p_1^{\a_1} \cdots p_s^{\a_s}$. Denote by $\div (n)$ the number of divisors of
$n$. It is well known that $\div (n)= \prod_{i=1}^s (\a_i+1)$. Further, we denote by $\ddiv (n)$ the number
of even divisors of $n$. We have the following lemma:
\begin{lem}
Let $g$ be fixed. The number of automorphism groups that can occur as automorphism groups $\Aut (\X_g)$
of a genus-$g$ hyperelliptic curves is given by the following:

i) if $\bAut(\X_g) \iso C_n$ then $n_1=\div (g+1)+ \div (2g+1)+\div (2g)-1 $

ii) if   $\bAut(\X_g) \iso D_n$ then $n_2=3\ddiv (g+1) + 2 \ddiv (g) + \div (g) -2 $

iii) if   $\bAut(\X_g) \iso A_4$ and $g > 6$ then $n_3=1$

iv) if  $\bAut(\X_g) \iso S_4$ then $n_4=1$ or 0.

v) if  $\bAut(\X_g) \iso A_5$ then $n_5=1$ or 0.
\end{lem}

\begin{proof} The proof is elementary and we skip the details.
\end{proof}
%
%
%
%



\subsubsection{Gonality of curves}
Let $\X$ be a curve  defined over $k$ and  $\eta: \X \rightarrow  \P^1$  a degree $n$ cover.  We assume that $\X$ has a $k$-rational point $P_\infty$ and hence a prime divisor $\p_\infty$ of degree $1$. The \textbf{gonality} $\gamma_\X$ of $\X$  is defined as
\[ 
\gamma_\X=\min \left\{ \deg(\eta):\X \rightarrow  \P^1  \right\}  =  \min  \left\{  [k(\X):k(x)]    \; | \;     x\in k(\X)  \right\}   .
\]
For $x\in k(\X)^*$,  define the pole divisor $(x)_\infty$ by 
\[ (x)_\infty=\sum_{\p\in \Sigma_\X(k)}   \max (0, -w_\p(x)) \cdot \p.\]
By the property of conorms of divisors, we  obtain  $\deg{(x)_\infty} =[k(\X):k(x)]$ if $x\notin k$. Thus, we have
\[ 
\gamma_\X = \min \left\{  \deg{ (x)_\infty}  \; | \;   x\in  k(\X) \setminus k   \right\} \,.
\]
\begin{prop}\label{gon}
For $\g_\X\geq 2$ we have $\gamma_\X \leq g$.
\end{prop}

The following statement strengthens the proposition.

\begin{cor} For curves $\X$ of genus $\geq 2$ with prime divisor $\p_\infty$ of degree $1$ there exists a cover
\[ \eta: \X\rightarrow  \P^1 \]
of $\deg(\eta)=n \leq g_\X$, such that $\p_\infty$  is ramified of order $n$ and so the point $P_\infty\in \X(k)$ attached to $\p_\infty$ is the only point on $\X$ lying over the point $[0:1] \in \P^1$.
\end{cor}

In general, the inequality in the proposition is not sharp, but of size $g/2$; see \cite{frey-shaska} for details.  
Curves with smaller gonality are special for various reasons.

\section{Equations of  curves with prescribed automorphism group}\label{sect-7}
Determining an equation for a family of curves with fixed automorphism group $G$ is an open problem.  Celebrated special solutions are the cases of the Klein curve, the Friecke or Friecke-MacBeath curve. In general, the following remains a difficult problem:

\begin{prob} 
Given an automorphism group $G$, determine an equation of a curve $\X$ such that $\Aut (\X)\iso G$.
\end{prob}

We know the  solution to the above problem for genus $g\leq 3$, but it is an open problem even for $g=4$.  For example, it is unknown what the corresponding equations for all cases in \cref{tab_1} are.  The only families of curves which we know how to determine an equation are the superelliptic curves. 

The method is almost identical to that of hyperelliptic curves, but now we have more choices for the reduced automorphism group $\overline G$.  We follow closely the terminology and notation of \cite{Sa-sh}. 

\subsection{Equations of superelliptic curves}
Let   $\d$ be given in~\cref{t2} and $M$, $\Lambda$, $Q$, $B$, $\Delta$, $\Theta$ and $\Omega$ are as follows: 
\begin{equation*}
\begin{split}
M = & \prod_{i=1}^\d \left(  x^{24}+\lambda_ix^{20}+(759-4\lambda_i)x^{16}+2(3\lambda_i+1228)x^{12} \right. \\
   & + \left. (759-4\lambda_i)x^8+\lambda_ix^4+1 \right) \\ 
\Lambda= & \prod_{i=1}^\d \left(-x^{60}+(684-\lambda_i)x^{55}-(55\lambda_i+157434)x^{50}-(1205\lambda_i-12527460)x^{45}\right.\\
                &-(13090\lambda_i+77460495)x^{40}+(130689144-69585\lambda_i)x^{35}\\
                & +(33211924-134761\lambda_i)x^{30}+(69585\lambda_i-130689144)x^{25}\\
                & -(13090\lambda_i+77460495)x^{20}-(12527460-1205\lambda_i)x^{15}\\
                &  \left. -(157434+55\lambda_i)x^{10} +(\lambda_i-684)x^5-1 \right) \\ 
Q  = & x^{30}+522x^{25}-10005x^{20}-10005x^{10}-522x^5+1\\ 
B = &  \prod_{i=1}^\delta \displaystyle{\prod_{a \in H_t}}   \left( (x+a)-\lambda_i \right) \\ 
\Theta = & \prod_{i=1}^\delta G_{\lambda_i} (x), \text{ where  } G_{\lambda_i}= \left( x \cdot \prod_{j=1}^{\frac {p^t -1} m} (x^m-b_j) \right)^m  -\lambda_i \\ 
\Delta = & \prod_{i=1}^{\d}  \left(  \left(  (x^q-x)^{q-1}+1 \right)^{\frac{q+1}{2}}-\lambda_i  (x^q-x)^{\frac{q(q-1)}{2}}   \right) \\ 
\Omega = & \prod_{i=1}^{\d}  \left( ((x^q-x)^{q-1}+1)^{q+1}-\lambda_i(x^q-x)^{q(q-1)}  \right) \\ 
\end{split}
\end{equation*}

\noindent Then we have the following result:

\begin{thm}\cite{Sa-sh}\label{thm_4}
Let $\X_g$ be an algebraic curve of genus $g \geq 2$ defined over an algebraically closed field $k$, $G$ its automorphism group over $k$, and $C_n$  a cyclic normal subgroup of   $G$ such that $g (X_g^{C_n} ) =0$.  Then, the equation for $\X_g$ falls into one of the following cases as in \cref{equations}.
\end{thm}


Each case in the \cref{equations} correspond to a $\delta$-dimensional family, where $\delta$ can be found in \cref{long-table}. Moreover, our parameterizations are exact in the sense that the number of parameters in each case equals the dimension.  It would be interesting to find invariants classifying isomorphism classes of superelliptic curves, and these families of curves in particular and to find equations in the moduli space of curves to determine these loci.


\begin{small}
%
\begin{center}
\begin{longtable}{|l|l|l|} 
\hline
$\#$ &$ \bar G$& $y^n=f(x)$ \\
\hline
1 &     & $x^{m\d}+a_1x^{m(\d-1)}+\dots +a_\d x^{m}+1$\\
2 & $C_m$    & $x^{m\d}+a_1x^{m(\d-1)}+\dots +a_\d x^{m}+1$\\
3 &     & $x(x^{m\d}+a_1x^{m(\d-1)}+\dots +a_\d x^{m}+1)$\\
\hline
4 &        		& $F(x):= \prod_{i=1}^\d(x^{2m}+\lambda_ix^m+1)$\\
5 &        		& $(x^m-1)\cdot F(x)$\\
6 &        		& $x\cdot F(x)$\\
7 &$D_{2m}$		& $(x^{2m}-1)\cdot F(x)$\\
8 &        		& $x(x^m-1)\cdot F(x)$\\
9 &        		& $x(x^{2m}-1)\cdot F(x)$\\
\hline
10 &     & $G(x):= \prod_{i=1}^\delta(x^{12}-\lambda_ix^{10}-33x^8+2\lambda_ix^6-33x^4-\lambda_ix^2+1)$\\
11 &     & $(x^4+2i\sqrt{3}x^2+1)\cdot G(x)$\\
12 &$A_4$& $(x^8+14x^4+1)\cdot G(x)$\\
13 &     & $x(x^4-1)\cdot G(x)$\\
14 &     & $x(x^4-1)(x^4+2i\sqrt{3}x^2+1)\cdot G(x)$\\
15 &     & $x(x^4-1)(x^8+14x^4+1)\cdot G(x)$\\
\hline
16 &     & $M(x)$\\
17 &     & $  \left( x^8+14x^4+1 \right)  \cdot M(x)$\\
18 &     & $x(x^4-1) \cdot M(x)$\\
19 &     & $\left( x^8+14x^4+1 \right)  \cdot x(x^4-1) \cdot M(x)$\\
20 &$S_4$& $\left( x^{12}-33x^8-33x^4+1  \right)\cdot M(x)$\\
21 &     & $\left( x^{12}-33x^8-33x^4+1  \right)  \cdot \left( x^8+14x^4+1 \right)  \cdot M(x)$\\
22 &     & $\left( x^{12}-33x^8-33x^4+1  \right) \cdot x(x^4-1) \cdot M(x)$\\
23 &     & $\left( x^{12}-33x^8-33x^4+1  \right) \cdot \left( x^8+14x^4+1 \right)  \cdot x(x^4-1)  M(x)$\\
\hline
24 &     & $\Lambda(x)$\\
25 &     & $x(x^{10}+11x^5-1) \cdot \Lambda(x)$\\
26 &     & $(x^{20}-228x^{15}+494x^{10}+228x^5+1)(x(x^{10}+11x^5-1))\cdot \Lambda(x)$\\
27 &     & $(x^{20}-228x^{15}+494x^{10}+228x^5+1)\cdot \Lambda(x)$\\
28 &$A_5$& $Q (x) \cdot \Lambda(x)$\\
29 &     & $x(x^{10}+11x^5-1).\psi(x)\cdot \Lambda(x)$\\
30 &     & $(x^{20}-228x^{15}+494x^{10}+228x^5+1)\cdot \psi(x)\cdot\Lambda(x)$\\
31 &     & $(x^{20}-228x^{15}+494x^{10}+228x^5+1)(x(x^{10}+11x^5-1))\cdot\psi(x)\cdot\Lambda(x)$\\
\hline
32 &$U$& $B(x)$\\
33 &   & $B(x)$\\
\hline
34 &     & $\Theta(x)$\\
35 &$K_m$& $x\prod_{j=1}^{\frac{p^t-1}{m}}\left(x^m-b_j\right)\cdot\Theta(x)$\\
36 &     & $\Theta(x)$\\
37 &     & $x\prod_{j=1}^{\frac{p^t-1}{m}}\left(x^m-b_j\right)\cdot\Theta(x)$\\
\hline
38 &          & $\Delta(x)$\\
39 &$\psl_2(q)$& $((x^q-x)^{q-1}+1)\cdot\Delta(x)$\\
40 &          & $(x^q-x)\cdot\Delta(x)$\\
41 &          & $(x^q-x)((x^q-x)^{q-1}+1)\cdot\Delta(x)$\\
\hline
42 &          & $\Omega(x)$\\
43 &$\pgl_2(q)$& $((x^q-x)^{q-1}+1)\cdot\Omega(x)$\\
44 &          & $(x^q-x)\cdot\Omega(x)$\\
45 &          & $(x^q-x)((x^q-x)^{q-1}+1)\cdot\Omega(x)$\\
\hline
%
\caption{Superelliptic  curves according to the automorphism group}
\label{equations}
\end{longtable}
\end{center}
\end{small}


\section{Binary forms and their invariants}\label{sect-8}

A superelliptic curve $\X_g$ defined over an algebraically closed field $k$ is given by a projective  equation of the form 
\begin{equation}\label{super-eq}
\X : \quad  y^n z^{d-n}= f(x, z), 
\end{equation}
for some degree $d$ binary form $f(x, z)$.  Let us assume that 
\[ y^n z^{d-n} = f(x, z)=\prod_{i=1}^s (x-\a_i z)^{d_i}, \quad 0 < d_i < d.\]
We have that $\sum_{i=1}^s d_i =d$.  A degree $d \geq 2$ binary form $f(x, z)$  is called \textbf{semistable} if it has no root of multiplicity $> \frac d 2$.
The only places where $\pi : \X_g \to \P^1$ ramifies correspond to the points $x=\a_i$. We denote such places by $Q_1, \dots , Q_s$ and denote the set of these places by $\B:=\{ Q_1, \dots , Q_s\}$.  The ramification indices are $e (Q_i) = \frac n {(n , d_i)}$.
Hence, every set $\B$ determines a genus $g$ superelliptic curve $\X_g$. However, the correspondence between the sets $\B$ and the isomorphism classes of $\X_g$ is not a one-to-one correspondence. Obviously the set of roots of $f(x)$ does not determine uniquely the isomorphism class of $\X_g$ since every coordinate change in $x$ would change the set of these roots. Instead, the isomorphism classes are classified by the invariants of binary forms.  
There is a huge amount of literature on classical invariant theory from XIX-century mathematics which has received more attention in the last few decades due to improved computational tools. 


A binary  form  of degree  $d$  is a homogeneous  polynomial $f(X,Y)$ of  degree $d$ in two  variables over $k$.  Let  $V_d$ be the $k$-vector space
of binary  forms of degree $d$.  The group $GL_2(k)$ of  invertible  $2  \times 2$  matrices  over  $k$  acts on
$V_d$  by coordinate  change.    Any  genus $g\geq 2$ superelliptic  curve over  $k$ has a  projective equation of the
form  \cref{super-eq},  where $f$ is degree $d$  a binary form   of  non-zero discriminant.
Two curves  are isomorphic  if and  only if  the corresponding  binary forms   are conjugate under $GL_2(k)$.
Therefore the moduli  space  of superelliptic  curves  is the affine  variety whose  coordinate ring  is
the  ring of $GL_2(k)$-invariants in the coordinate ring  of the set of elements of $V_d$ with non-zero
discriminant.

Generators for this  and similar invariant rings in  lower degree were constructed by Clebsch, Bolza  and others
in  the last  century using complicated calculations.  For the case of sextics,  Igusa \cite{Ig} extended  this to
algebraically closed  fields of  any characteristic using  techniques  of modular forms and algebraic  geometry. In \cite{vishi} Igusa's result is proved in an elementary way using methods of geometric reductivity.

Hilbert  \cite{Hi}  developed some  general, purely  algebraic tools  in invariant theory. Combined with the linear reductivity of  $GL_2(k)$ in characteristic 0, this permits a  more
conceptual proof  of the  results of  Clebsch \cite{Cl} and Bolza  \cite{Bo}. After  Igusa's  paper appeared, the concept  of geometric reductivity  was developed by Mumford \cite{Mu1}, Haboush \cite{Ha}  and others. Haboush's theorem states that for any semisimple algebraic group  $G$ over $k$, and for any linear representation of $G$ on a $k$-vector space $V$, given $v \in V$ with $v\not = 0$ that is fixed by the action of $G$, there is a $G$-invariant polynomial $F$ on $V$, without constant term, such that $F(v) \not =0$. The polynomial $F$ can be taken to be homogeneous, and if the characteristic is $p>0$ the degree of the polynomial can be taken to be a power of $p$.  In particular, it was proved  that reductive algebraic  groups  in   any characteristic   are geometrically reductive. This allows  the application  of  Hilbert's  methods  in  any characteristic. For example, Hilbert's finiteness theorem   was  extended   to  any  characteristic  by  Nagata  \cite{Na}. Here, we follow the same approach for binary sextics and octavics. The proofs are  elementary in characteristic 0, and extend to characteristic  $p >  5$  by  quoting  the respective  results using geometric reductivity. 

\subsection{Invariants of Binary Forms}


Let $k$ denote an algebraically closed field.

\subsubsection{Action of $GL_2(k)$ on binary forms.}
Let $k\, [X, Y]$  be the  polynomial ring in  two variables and  let $V_d$ denote  the  $d+1$-dimensional subspace of $k\, [X, Y]$  consisting  of homogeneous polynomials.
\begin{equation}\label{eq1}
f(X,Y) = a_0X^d + a_1X^{d-1}Y +  \dots  + a_dY^d
\end{equation}
of  degree $d$. Elements  in $V_d$  are called  {\it binary  forms} of degree $d$.
We let $GL_2(k)$ act as a group of automorphisms on $ k\, [X,Y] $ as follows: if
\[ g = \begin{pmatrix} a & b \\ c & d \end{pmatrix}  \in GL_2(k) 
\]
 then
\begin{equation}\label{eq2a}
g (X) = aX + bY    \quad \text{ and } \quad g (Y) = cX + dY
\end{equation}
This action of $GL_2(k)$  leaves $V_d$ invariant and acts irreducibly on $V_d$.

\begin{rem}\label{rem1}
It is well  known that $SL_2(k)$ leaves a bilinear  form (unique up to scalar multiples) on $V_d$ invariant. This
form is symmetric if $d$ is even and skew symmetric if $d$ is odd.
\end{rem}

Let $A_0$, $A_1$,  \dots ,   $A_d$ be coordinate  functions on $V_d$. Then the coordinate  ring of $V_d$ can be
identified with $ k\, [A_0  ,    \dots  ,   A_d] $. For $I \in k\, [A_0,  \dots  ,   A_d]$ and $g \in GL_2(k)$,
define $I^g \in k\, [A_0,  \dots  ,   A_d]$ as follows
\begin{equation}\label{eq3}
{I^g}\, (f) = I \, ( g ( f ) )
\end{equation}
for all $f \in V_d$. Then  $I^{gh} = (I^{g})^{h}$ and  \eqref{eq3} defines an action of $GL_2(k)$ on $k\, [A_0,
\dots  ,   A_d ]$.

\begin{defn}
Let  $\mathcal R_d$  be the  ring of  $SL_2(k)$ invariants  in $k\, [A_0,  \dots  ,   A_d]$, i.e., the ring of all
$I \in k \, [A_0,  \dots  ,   _d]$ with $I^g = I$ for all $g \in SL_2(k)$.
\end{defn}

Note that if $I$ is an invariant, so are all its homogeneous components. So $\mathcal  R_d$ is  graded by  the
usual degree  function on  $k\, [A_0,  \dots  ,   A_d]$.

Since $k$ is algebraically closed, the binary form $f(X,Y)$ in   \eqref{eq1} can be factored as
\begin{equation} \label{eq4}
f(X,Y)  = (y_1  X  -  x_1 Y) \cdots (y_d  X  - x_d  Y) = \displaystyle \prod_{1 \leq  i \leq  d} \det
\left(\begin{pmatrix} X & x_{i} \\ Y & y_i
\end{pmatrix} \right)
\end{equation}
The points  with homogeneous coordinates $(x_i, y_i)  \in \mathbb P^1$ are  called the  roots  of the  binary form
\eqref{eq1}.  Thus  for $g  \in GL_2(k)$ we have
$$g\left (f(X,Y) \right )  = ( \det(g))^{d}  (y_1^{'}  X -  x_1^{'}  Y) \cdots (y_d^{'} X  - x_d^{'} Y),$$
where
\begin{equation}
\begin{pmatrix}   x_i^{'}  \\  y_i^{'}  \end{pmatrix} = g^{-1}  \begin{pmatrix} x_i\\ y_i \end{pmatrix}.
\end{equation}
%
%
The  \textbf{null cone}  $N_d$ of  $V_d$  is the  zero set  of all  homogeneous elements in $\mathcal R_d$ of positive degree.
%
\begin{lem}\label{lem1b}
Let $\ch (k)  = 0$  and $\Omega_s$  be the subspace  of $k\, [A_0,   \dots  ,   A_d]$ consisting of homogeneous
elements of degree $s$. Then there is a $k$-linear  map 
\[R :  k\, [A_0,  \dots  ,    A_d] \to \mathcal R_d,  \]
  with the following properties:

\smallskip

(a) $R(\Omega_s) \subseteq  \Omega_s$ for all $s$

(b) $R(I) = I$ for all $I \in \mathcal R_d$

(c) $R(g(f)) = R(f)$ for all $f \in k\, [A_0,  \dots  ,   A_d]$
\end{lem}

\begin{proof} $\Omega_s$ is a  polynomial module of degree $s$  for $SL_2(k)$. Since $SL_2(k)$  is linearly  reductive
in $\ch (k)  = 0$,  there exists  a $SL_2(k)$-invariant  subspace  $\Lambda_s$  of  $\Omega_s$  such  that
$\Omega_s = (\Omega_s \cap  \mathcal R_d) \bigoplus \Lambda_s$. Define $R : k\, [A_0,   \dots  ,   A_d] \to
\mathcal R_d $ as $R(\Lambda_s)  = 0$ and $R_{|\Omega_s \cap \mathcal R_d} = id$. Then $R$ is $k$-linear and the
rest of the proof is clear from the definition of $R$.
\end{proof}

\noindent The map $R$ is called the {\bf Reynold's operator}.

\begin{lem}\label{lem2}
Suppose $\ch (k) = 0$.  Then every maximal ideal in $ \mathcal R_d$ is contained in a maximal ideal of $k\, [A_0,
\dots  ,   A_d]$.
\end{lem}

\begin{proof}
If $\mathcal I$ is a maximal ideal in $ \mathcal R_d $ which generates the unit ideal of
 $ k\, [A_0, $ $\dots  ,   A_d]$, then there exist $m_1,   \dots ,   m_t \in \mathcal I$ and $f_1$, $f_2$,  \dots
, $f_t \in k\, [A_0,  \dots  ,   A_d]$ such that
$$1 = m_1 f_1 +  \dots  + m_t f_t$$
Applying the Reynold's operator to the above equation we get
$$1 = m_1 \, R(f_1) +  \dots  + m_t \, R(f_t)$$
But  $R(f_i)  \in \mathcal  R_d$  for all  $i$.  This  implies $1  \in \mathcal I$, a contradiction.
\end{proof}

The following is known as the Hilbert's Finiteness  Theorem.
\begin{thm}\label{thm1a}  Suppose $\ch (k)  = 0$.  Then $\mathcal R_d$ is finitely generated over $k$.
\end{thm}

\begin{proof}
Let $\mathcal I_0$ be the ideal in $k\, [A_0,  \dots  ,   A_d]$ generated by all homogeneous invariants of
positive degree. Because $k\, [A_0,  \dots  ,  A_d]$ is Noetherian,  there exist finitely many  homogeneous
elements $J_1,
 \dots  ,   J_r$  in $\mathcal  R_d$   such  that  $\mathcal  I_0  =  (J_1,  \dots  ,  J_r)$.
We prove $\mathcal R_d =  k\, [J_1,  \dots  ,   J_r]$. Let $J \in \mathcal R_d$  be homogeneous  of degree $d$. We
prove $J  \in k\, [J_1, \dots ,   J_r]$ using induction on $d$.  If $d = 0$, then $J \in k \subset k\, [J_1, \dots
,   J_r]$. If $d > 0$, then
\begin{equation}\label{eq_6}
J = f_1 \, J_1 +  \dots  + f_r \, J_r
\end{equation}
with $f_i  \in k\, [A_0,  \dots  ,    A_d]$ homogeneous and $deg(f_i)  < d$ for all $i$. Applying the Reynold's
operator to  \eqref{eq_6} we have
$$J = R(f_1) J_1 +  \dots  + R(f_r) J_r$$
then by  Lemma 1 $R(f_i)$ is  a homogeneous element  in $\mathcal R_d$ with $deg(R(f_i))  < d $  for all $i$  and
hence by induction  we have $R(f_i) \in k\, [J_1,  \dots   ,   J_r]$ for all $i$. Thus $J  \in k\, [J_1,  \dots  ,
J_r]$.

\end{proof}

If $k$ is of arbitrary characteristic, then $SL_2(k)$ is geometrically reductive,  which is a  weakening of linear
reductivity; see Haboush \cite{Ha}. It suffices to prove Hilbert's finiteness theorem in any characteristic; see
Nagata \cite{Na}. The following theorem is also due to Hilbert.

\begin{thm}\label{thm2a}
Let $I_1$,  $I_2$,  \dots ,    $I_s$ be homogeneous  elements in  $ \mathcal R_d$ whose common zero set equals the
null cone $\mathcal N_d$. Then $ \mathcal R_d$  is finitely  generated as a  module over $k\, [I_1,   \dots  ,
I_s]$.
\end{thm}

\begin{proof}
Consider first the case  $\ch (k)  = 0$. By   \cref{thm1a}  we have $\mathcal R_d  = k\, [J_1, J_2,  \dots  ,    J_r]$ for  some
homogeneous invariants  $J_1$,  \dots   ,  $J_r$. Let $\mathcal I_0$  be the maximal ideal  in $ \mathcal  R_d$
generated by all homogeneous elements  in $ \mathcal R_d$ of  positive degree. Then the theorem follows if $I_1$,
\dots , $I_s$ generate an ideal $\mathcal I$ in $ \mathcal R_d$ with $rad(\mathcal I)=\mathcal I_0$. For if this
is the case, we have an integer $q$ such that
\begin{equation}\label{eq5}
 J_i ^{q} \in \mathcal I, \quad  \textit{    for all   } i
\end{equation}
Set \[ S:= \{ J_1 ^{i_1} J_2^{i_2} \dots J_r^{i_r} \,  | \, 0 \leq i_1,  \dots  ,   i_r <  q \}.\]  Let $\mathcal  M$
be the $k\, [I_1,  \dots   I_s]$-submodule in $\mathcal R_d$  generated by $S$.  We prove $\mathcal R_d  =
\mathcal M$. Let $J \in \mathcal  R_d$  be homogeneous.  Then $J  = J^{'}  + J^{''}$  where  $J^{'} \in  \mathcal
M$, $\, \, J^{''}$ is a $k$-linear combination of  $J_1^{i_1} J_2 ^{i_2}   \dots  J_r ^{i_r}$ with  at least one
$i_{\nu} \geq q$ and $deg(J) = deg(J^{'})  = deg(J^{''})$. Hence  \eqref{eq5} implies $J^{''} \in \mathcal I$
and so we have
\[J^{''} = f_1 \, I_1 + \cdots  + f_s \, I_s \]
where $f_i \in \mathcal R_d$ for all $i$. Then 
\[ \deg (f_i) < \deg (J^{''}) = \deg(J),
\]
 for all $i$. Now by induction on degree of $J$ we may assume $f_i \in \mathcal M$ for all $i$. This implies $J^{''} \in \mathcal M$ and  hence  $J
\in  \mathcal  M$.  Therefore  $\mathcal M  =  \mathcal R_d$.  So  it  only  remains  to prove  $rad(\mathcal  I) =  \mathcal I_0$. This  follows from  Hilbert's Nullstellensatz and  the following claim.

\medskip

\noindent  \textbf{Claim:}  $\mathcal I_0$ is the only maximal ideal containing $I_1,  \dots  ,   I_s$.

\medskip

Suppose $\mathcal I_1 $ is  a maximal ideal  in $ \mathcal  R_d$ with $I_1,  \dots   ,   I_s \in  \mathcal I_1$.
Then from Lemma  2 we know there exists a  maximal ideal  $\mathcal J$  of $k\, [A_0,  \dots   ,   A_d]$  with $
\mathcal I_1 \subset \mathcal J$.  The point in $V_d$ corresponding to $\mathcal  J$ lies  on the  null  cone
$\mathcal N_d$ because  $I_1,  \dots   ,   I_s$ vanish  on  this point.  Therefore  $\mathcal I_0  \subset
\mathcal J$,  by definition of  $\mathcal N_d$. Therefore  $\mathcal J \cap \mathcal R_d$ contains both the
maximal ideals $\mathcal I_1$ and $\mathcal I_0$. Hence,   $\mathcal I_1 = \mathcal J  \cap \mathcal R_d =
\mathcal I_0$.

Next we consider the case   $\ch (k) = p>0$. The same proof works if \cref{lem2} above holds. Geometrically  this means the morphism  $\pi : V_d \to
V_d$ // $SL_2(k)$ corresponding to the  inclusion $\mathcal R_d \subset k\, [A_0,  \dots  ,   A_d]$ is surjective.
Here $ V_d$ // $SL_2(k)$ denotes the affine variety corresponding  to the  ring $\mathcal R_d$  and is  called the
\textbf{categorical quotient}. $\pi$ is surjective  because $SL_2(k)$ is geometrically  reductive. The proof  is by
reduction modulo $p$, see Geyer \cite{Ge}.
\end{proof}


\subsubsection{Symbolic method}

We will use the symbolic method of classical theory to construct covariants of binary forms. First we  recall some facts about the symbolic notation. Let
\[f(X,Y):=\sum_{i=0}^n
\begin{pmatrix} n \\ i
\end{pmatrix}
a_i X^{n-i} \, Y^i, \quad  and \quad g(X,Y) :=\sum_{i=0}^m
  \begin{pmatrix} m \\ i
\end{pmatrix}
b_i X^{n-i} \, Y^i
\]
be binary forms of  degree $n$ and $m$ respectively. We define the $r$-{\it transvection}
\[(f,g)^r:= \frac {(m-r)! \, (n-r)!} {n! \, m!} \, \,
\sum_{k=0}^r (-1)^k
\begin{pmatrix} r \\ k
\end{pmatrix} \cdot
\frac {\partial^r f} {\partial X^{r-k} \, \,  \partial Y^k} \cdot \frac {\partial^r g} {\partial X^k  \, \,
\partial Y^{r-k} },
\]
see Grace and Young \cite{GY} for  details.

The following result gives relations among the invariants of binary forms and it is known as the \textbf{Gordon's formula}.  It is the basis for most of the classical results on invariant theory.  

\begin{thm}
Let $\phi_i$, $i=0, 1, 2$  be covariants of order $m_i$ and $e_i, e_j, m_k$ be three non-negative integers such that $ e_i+e_j \leq m_k$,   for distinct $i, j, k$.  The following holds
\begin{small}
\begin{equation}
\begin{split}
& \sum_i  \frac  {C_i^{e_1} \cdot  C_i^{m_1-e_0-e_2 }}  {C_i^{m_0+m_1+1-2e_2-i}   } \, \left(  \left(\phi_0 \, \phi_1 \right)^{e_2+1}, \phi_2  \right)^{e_0+e_1-i}  \\
 =& \sum_i  \frac  {C_i^{e_2} \cdot  C_i^{m_2-e_0-e_1 }}  {C_i^{m_0+m_2+1-2e_1-i}   } \, \left( \left(\phi_0 \, \phi_2 \right)^{e_1+1}, \phi_1      \right)^{e_0+e_2-i}, 
\end{split} 
\end{equation}
\end{small}
where $e_0=0$ or $e_1+e_2=m_0$.  
\end{thm}

This result has been used by many XIX century mathematicians to compute algebraic relations among invariants, most notably by Bolza for binary sextics and by Alagna for binary octavics.  It provides algebraic relations among the invariants in a very similar manner that the Frobenious identities do for theta functions of hyperelliptic curves. Whether there exists some explicit relation among both formulas is unknown.  

\subsubsection{Binary sextics}

Let $f(x, z)$ be a binary sextic defined over a field $k$,  $\ch k =0$,  given by
\begin{equation}\label{eq_1}
f(x,z)       =      \sum_{i=0}^6 a_i x^{6-i} z^i             =  (z_1x-x_1z)(z_2x-x_2z) \dots (z_6x-x_6z)
\end{equation}
Consider the  following covariants
\begin{equation}\label{Y-transv}
\begin{aligned}
& \Delta  = \left( (f, f)_4, (f, f)_4  \right)_2,     &  Y_1  = \left(f,  (f, f)_4  \right)_4 \\
& Y_2  = \left( (f, f)_4, Y_1  \right)_2,    &  Y_3  = \left( (f, f)_4, Y_2  \right)_2  \\
\end{aligned}
\end{equation}
The \textbf{Clebsch invariants} $A, B, C, D$  are defined as follows
\begin{equation}
 A  = (f, f)_6, \; \; B  = \left( (f, f)_4,  (f, f)_4  \right)_4, \; \; C  = \left( (f, f)_4, \Delta \right)_4, \; \;   D  = \left(  Y_3, Y_1 \right)_2,   
\end{equation}
see Clebsch \cite{Cl} or Bolza \cite{Bo}*{Eq.~(7), (8), pg. 51}  for details. 
%
%
%


\noindent \textbf{Root differences:}    Let $f(x, z)$ be a binary sextic as above and set
$D_{ij}:=   \begin{pmatrix} x_i & x_j \\  z_i & z_j \end{pmatrix}$.
For $\tau  \in SL_2(k)$,  we have
\[ \tau (f) =  (z_1^{'} x  - x_1^{'} z)   \dots  (z_6^{'}  x - x_6^{'}  z), \quad  \textit{ with } \quad
 \begin{pmatrix}   x_i^{'}  \\ z_i^{'}  \end{pmatrix}  = \tau^{-1} \, \begin{pmatrix} x_i\\ z_i \end{pmatrix}.
\]
Clearly $D_{ij}$ is  invariant under this action of $SL_2(k)$ on $\mathbb P^1$. Let $\{i, j,  k, l, m, n  \}$ = $\{ 1, 2, 3$,  $4, 5, 6 \}$. Treating $a_i$ as  variables, we construct the following elements in the ring of invariants $\mathcal R_6$
%
\begin{small}
\begin{equation}\label{j-invariants}
\begin{split}
\IA  & =  a_0^2 \,  \prod_{fifteen}   (12)^2 (34)^2 (56)^2  = \displaystyle \sum_{\substack {i<j,k<l,m<n}}  D_{ij}^2D_{kl}^2D_{mn}^2 \\
& \\
\IB  & =   a_0^4 \, \prod_{ten}  (12)^2 (23)^2 (31)^2 (45)^2 (56)^2 (64)^2 = \sum_{\substack {i<j,j<k, \\ l<m,m<n}}    D_{ij}^2D_{jk}^2D_{ki}^2D_{lm}^2D_{mn}^2D_{nl}^2 \\
 & \\
\IC  & =  a_0^6 \,  \prod_{sixty}   (12)^2 (23)^2 (31)^2 (45)^2 (56)^2 (64)^2  (14)^2 (25)^2 (36)^2  \\
 & = \sum_{ \substack {i<j,j<k, l<m, m<n \\  i<l' , j < m', k<n' \\  l',m',n' \in \{ l, m, n\} } }
 D_{ij}^2 D_{jk}^2 D_{ki}^2 D_{lm}^2 D_{mn}^2 D_{nl}^2  D_{{il}^{'}}^2 D_{{jm}^{'}}^2 D_{{kn}^{'}}^2 \\
 & \\
\ID  & =  a_0^{10} \prod_{i<j} (i j)^2   \\
\end{split}
\end{equation}
\end{small}
These    invariants, sometimes called \textbf{integral invariants},   are defined in \cite{Ig}*{pg. 620} where they are denoted by $A, B, C, D$.  Incidentally even Clebsch invariants which are defined next are also denoted by $A, B, C, D$ by many authors.  



%
%

To quote Igusa \textit{"if we restrict to integral invariants, the discussion will break down in characteristic 2 simply because Weierstrass points behave badly under reduction modulo 2"}; see \cite{Ig}*{pg. 621}.  Next we define invariants which will work in every characteristic. 

In \cite{Ig}*{pg. 622} Igusa defined what he called \textbf{basic arithmetic invariants}, which are now commonly known as  \textbf{Igusa invariants}
\[ 
J_2 = \frac 1 {2^3} \IA, \;   J_4=  \frac 1 {2^5 \cdot 3} (4J_2^2-\IB),
J_6 = \frac 1 {2^6 \cdot 3^2} (8J_2^3-160J_2 J_4 -\IC), \;          J_{10}= \frac 1 {2^{12}} \ID  
\]
While most of the current literature on genus 2 curves uses invariants $\IA, \IB, \IC, \ID$, which are now most commonly labeled as $I_2, I_4, I_6, I_{10}$, Igusa went to great lengths in \cite{Ig} to define $J_2, J_4, J_6, J_{10}$ and to show that they also work in characteristic 2. 
\begin{lem}\label{lem3}
$J_{2i}$ are homogeneous elements in $\mathcal R_6$ of degree $2i$,  for $i$ = 1,2,3,5.
\end{lem}

\begin{lem}\label{lem4}
A sextic has a root of multiplicity exactly three if and only if the basic invariants take the form
\begin{equation}\label{eq6}
J_2 = 3r^{2}, \quad J_4 =  81r^{4}, \quad J_6 = r^{6},  \quad J_{10} = 0.
\end{equation}
for some $ r \neq 0 $.
\end{lem}

\begin{lem}\label{lem5}
A sextic has a root of multiplicity at least four if and only if the basic invariants vanish simultaneously.
\end{lem}

Both of the above lemmas are useful when we study semistable and stable genus 2 curves. 
\begin{lem}\label{lem6}
$\mathcal R_6$ is finitely generated as a module over $k\, [I_2, I_4, I_6, I_{10}]$.
\end{lem}

\begin{cor}\label{cor1a} (\textbf{Clebsch-Bolza-Igusa})
Two binary sextics $f$ and $g$ with $I_{10} \neq 0$ are $GL_2(k)$ conjugate if and only if there exists an $r \neq
0$ in $k$ such that for every $i$ = 1, 2, 3, 5 we have
\begin{equation}\label{eq9}
I_{2i}(f) = r^{2i} \, I_{2i}(g)
\end{equation}
\end{cor}

See \cite{vishi} for a proof. We will use \cref{eq9} when we consider the moduli space of binary sextics as a weighted moduli space.

\subsubsection{Binary octavics}  Next we will construct covariants and invariants of binary octavics.  They were first constructed by van Gall who showed that there are 70 such covariants; see von Gall \cite{vG}.  Let $f(X,Y)$ denotes a binary octavic as below:
\begin{equation}
f(X,Y) =   \sum_{i=0}^8 a_i X^i Y^{8-i} = \sum_{i=0}^8
\begin{pmatrix} n \\ i
\end{pmatrix}    b_i X^i Y^{n-i}
\end{equation}
where $b_i=\frac {(n-i)! \, \, i!} {n!} \cdot a_i$,  for $i=0, \dots , 8$. We define the following
covariants:
\begin{equation}
\begin{split}
&g=(f,f)^4, \quad k=(f, f )^6, \quad h=(k,k)^2, \quad m=(f,k)^4, \\
&   n=(f,h)^4, \quad p=(g,k)^4, \quad q=(g, h)^4.\\
\end{split}
\end{equation}

\noindent Then, the following 
\begin{small}
\begin{equation}\label{def_J}
\begin{aligned}
&  J_2= 2^2 \cdot 5 \cdot 7 \cdot (f,f)^8,     &  J_3 = \frac 1 3 \cdot  2^4 \cdot 5^2 \cdot 7^3 \cdot (f,g )^8,   &  \; J_4= 2^9 \cdot 3 \cdot 7^4 \cdot (k,k)^4,  \\
&  J_5= 2^9 \cdot 5 \cdot 7^5 \cdot (m,k)^4,  &   J_6 = 2^{14} \cdot 3^2 \cdot 7^6 \cdot (k,h )^4,                 &  \;  J_7= 2^{14} \cdot 3 \cdot 5 \cdot 7^7 \cdot (m,h )^4,  \\
& J_8= 2^{17} \cdot 3 \cdot 5^2 \cdot 7^9 \cdot  (p,h)^4,    &  J_9= 2^{19} \cdot 3^2 \cdot 5 \cdot 7^9 \cdot  (n,h)^4, &   \; J_{10}=   2^{22} \cdot 3^2 \cdot 5^2 \cdot 7^{11} (q,h)^4         
\end{aligned}
\end{equation}
\end{small}
are $SL_2(k)$- invariants.     Notice that these invariants are scaled up to multiplication by a constant for computational purposes only; see    \cite{shi1} and \cite{hyp-3} for further details. 

\begin{lem}\label{lem_3} 
For each binary octavic $f(X, Y)$, its invariants defined in \cref{def_J}  are primitive homogeneous polynomials $J_i\in \Z [a_0, \dots , a_8]$  of degree $i$, for $i=2, \dots , 10$.
Let $f^\prime=g (f)$, where
\[ g =
\begin{pmatrix} a &b \\ c & d
\end{pmatrix}
\in GL_2(k),  
\]
and denote the corresponding $J_2, \dots , J_{10}$ of $f^\prime$ by $J_2^\prime, \dots , J_{10}^\prime$. Then,
$$J_i^\prime= ( \Delta^4)^i \, J_i$$ where $\Delta=ad-bc$ and $i=2, \dots , 10$.
\end{lem}

\proof The first claim is immediate from the definition of the covariants and invariants.  Let $f$ and
$f^\prime$ be two binary octavics as in the hypothesis. One can check the result computationally. \qed

There are 68 invariants defined this way as discovered by van Gall \cites{vG, vG2} in 1880.  Indeed, van Gall claimed 70 such invariants, but as discovered in XX-century there are only 68 of them.  In particular, $J_{14}$ is the discriminant of the binary octavic.  In articles in 1892 and 1896 R. Alagna determined the algebraic relations among such invariants; see \cites{Al, Al1} for details. 

Next we want to show that the ring of invariants $\cR_8$ is finitely generated as a module over $k[J_2,
\dots,J_7]$. First we need some auxiliary lemmas.

\begin{lem}\label{lem_4} 
If $J_i=0$, for $i=2, \dots 7$, then the $f(X,Y)$ has a  multiple root.
\end{lem}

\proof Compute $J_i=0$, for $i=2, \dots 7$. These equations imply that
\[Res( f(X,1), f^\prime (X, 1), X)=0,\]
where $f^\prime $ is the derivative of $f$. This proves the lemma. \qed

\begin{thm} \label{thm_5}   The following hold true for any octavic.

i)   An octavic has a root of multiplicity exactly four if and only if the basic invariants take the form
\begin{equation}\label{J_i}
\begin{split}
J_2 &= 2  \cdot r^2, 
\quad       J_3=  2^2 \cdot 3 \cdot  r^3,  
\quad J_4 = 2^6 \cdot   r^4, 
\quad J_5 = 2^6   \cdot r^5,\\
 &J_6=  2^9 \cdot    r^6, 
 \quad J_7= 2^9 \cdot   r^7, 
 \quad J_8=  2^{11}\cdot 3^2 \cdot        r^8,
\end{split}
\end{equation}
for some  $ r \neq 0 $.  Moreover, if the octavic has equation 
\[ f(x, y)=x^4 (a x^4+b x^3 y+c x^2 y^2+d x y^3+e y^4), \]
then  $r=e$.

ii) \label{lem_root} An octavic has a root of multiplicity 5 if and only if
\[ J_i=0, \ \  for \ \ i=2, \dots , 8.\]

\end{thm}

\begin{rem}
An alternative proof of the above can provided using the $k$-th subresultants of $f$ and its derivatives. Two
forms have $k$ roots in common if and only if the first  $k$ subresultants vanish.
This is equivalent to $J_2= \dots = J_7=0$.
\end{rem}

\begin{thm}\label{thm_6}
$\cR_8$ is finitely generated as a module over $k[J_2, \dots,J_7]$.
\end{thm}

\begin{cor}
$J_2, \dots , J_7$ are algebraically independent over $k$ because $\cR_8$ is the coordinate ring of the     5-dimensional variety $V_8$//$SL_2(k)$.
\end{cor}

In \cite{hyp-3} the following theorem was proved that determines explicitly the relation among the invariants. 

\begin{thm} \label{eq-main-g-3}  
The  invariants  $J_2, \dots , J_8$ satisfy  the following equation
\begin{equation}\label{shaska}
 J_8^5 + \frac { I_8} {3^4 \cdot 5^3 } J_8^4 + 2 \cdot \frac { I_{16} } {3^8\cdot 5^6}   J_8^3 + \frac {I_{24}} {2 \cdot 3^{12} \cdot 5^6}  J_8^2  + \frac { I_{32}} {3^{16} \cdot 5^{10} }  J_8 + \frac { I_{40}} {2^2 \cdot 3^{20} \cdot 5^{12}}  =0, 
 \end{equation}
where  $I_8, I_{16}, I_{24}, I_{32}, I_{40}$ are expressed in terms of the coefficients in the Appendix in \cite{hyp-3}
\end{thm}

We suggest the following problem.

\begin{prob}
Express all invariants $I_8, I_{16}, I_{24}, I_{32}, I_{40}$ in terms of the transvectants of the binary octavics. 
\end{prob}

We also have a similar result for superelliptic curves. 

\begin{thm} Two superelliptic curves $C$ and $C^\prime$ in Weierstrass form, given by affine equations  
\[  C:   Z^n=f(X, 1)   \textit{    and   } C^\prime:  z^n=g(X, 1) \]
with $\deg f = \deg g =8$ are isomorphic over $k$ if and only if there exists some $\l \in k\setminus \{ 0\}$ such that 
\[ J_i (f) = \l^i \cdot J_i(g), \textit{   for   }  \,\,  i=2, \dots , 7,   \]  
and $J_2, \dots J_8$ satisfy the  \eqref{shaska}.
Moreover, the isomorphism $C  \to C^\prime$  is given by
\[ 
\begin{bmatrix} X \\ Y \end{bmatrix}  \to M \cdot \begin{bmatrix} X \\ Y \end{bmatrix}
\]
where $M  \in GL_2(k) $ and $\l = \left( \det M \right)^4$.
\end{thm}

Using \cref{eq-main-g-3} one can build a database of superelliptic curves $y^n=f(x)$, for $\deg f = 8$. This was done in \cite{beshaj-polak} for genus 3 hyperelliptic curves.


\subsection{Discriminant of a superelliptic curve} 
An important  invariant is the discriminant of the binary form.  In the classical way, the discriminant is defined as 
$ \D =  \prod_{i \neq j } (\a_i - \a_j)^2$, 
where $\a_1, \dots \a_d$ are the roots of $f(x, 1)$.   It is a well-known result that it can be expressed in terms of the transvections. For example, for binary sextics we have $\D = J_{10}$  and for binary octavics  $\D (f) = J_{14}$.

The discriminant of a degree $d$ binary form $f(X, Z)\in k[X, Z]$ is and $SL_2 (k)$-invariant of degree $2d-2 $.   For any $M \in GL_2 (k)$ and any degree $d$ binary form $f$ we have
\[ \D (f^M) = \left(  \det M \right)^{d (d-1) } \, \D (f) \,.\]
The concept of a minimal discriminant is classical concept in number theory, starting with the binary quadratics. The minimal discriminant 
for elliptic curves was studied by Tate and others in the 1970-s; see \cite{ta-75} and generalized by Lockhart in \cite{lockhart} for hyperelliptic curves.  We will consider superelliptic curves with minimal discriminant or with minimal set of invariants in \cref{sect-10}. 

  
\subsection{Dihedral invariants of superelliptic curves with extra automorphisms}  
%
For curves with extra automorphisms we have additional invariants which are simpler in form and easier to compute.  These invariants were introduced in \cite{g_sh} for hyperelliptic curves and generalized in  \cite{AK} for superelliptic curves.

We will say that the superelliptic curve is  in \textbf{normal form} if and only if it is given by an equation of the form
\[y^n = x^s + \sum_{i=1}^{d/\delta}a_ix^{\delta \cdot i} + 1.    \]
To parametrize families of the superelliptic curves  that admit an extra automorphism of order $\d$, we determine the set of possible coefficients $\{ a_{s / \d-1}, \cdots, a_1\}$ of this normal form up to a change of coordinate in $x$.   The condition $\tau (x)= \zeta x$, implies that $\bar{\tau}$ fixes the places $0, \infty$. Moreover we can change the defining equation by a morphism $\g \in \pgl_2(k)$ of the form $ \g: x \rightarrow mx$ or $ \g: x \rightarrow \frac{m}{x}$  so that the new equation is again in normal form.  Substituting  
$a_0= (-1)^{d/s} \prod_{i=1}^{d/s} \b_i^s \,,$
we obtain
\[(-1)^{s/ \d} \prod_{i=1}^{s/ \d} \g(\b_i)^\d =1 \,,\]
whence $m^s= (-1)^{s/\d}$. Then,  $x$ is determined up to a coordinate change by the subgroup $D_{s/ \d}$ generated by
\[ \t_1: x \rightarrow \epsilon  x, \thinspace \t_2 : x \rightarrow \frac{1}{x} \,, \]
where $\epsilon$ is a primitive $s/\d$-root of unity; see \cite{g_sh} for details. 
The action of $D_{s/\d}$ on the parameter space $k(a_1,\dots, a_{s/\d} )$ is given by 
\[
\begin{split}
 \t_1 : &  \, a_i  \rightarrow \epsilon^{\d i}  a_i,   \text{for} i= 1, \dots s/\d \,,\\
 \t_2 : & \, a_i  \rightarrow   a_{d/ \d-i},       \text{for} i= 1,  \dots [s/\d] \,. \\
 \end{split}
 \]
Notice that if $s/\d=1$ then the above actions are trivial, therefore the normal form determines the equivalence class.  If $s/\d=2$ then 
$$ \t_1(a_1) = -a_1,   \tau_1 (a_2)= a_2, \tau_2= 1 $$
and the action is not dihedral but cyclic on the first vector.

\begin{lem}  
Assume that $s/\d> 2$. The fixed field $k(a_1,a_2,\cdots a_{s/\d})^{D_{s/\d}}$ is the same as the function field of the variety $\L_{n,s,\d}$.
\end{lem}

\begin{lem} Let $r:= s/\d >2$. The elements 
$$ \u_i:= a_1^{r-i}a_1 + a_{r-1}^{r-i}a_{r-i},  \textbf{ for } i=1, \dots ,r$$
 are  invariants under the action of the group $D_{s/\d}$  defined as above. 
\end{lem}

The elements $\u_i$ are called the \textbf{dihedral invariants}. 
 
 \begin{thm}
 Let $\u=(\u_1, \dots , \u_r )$ be the $r$-tuple of \ss-invariants. Then 
 \[ k(\L_{s,n,\d})= k(\u_1, \dots, \u_r).\]
 \end{thm}
 
In \cref{sect-11} we will sill show how to determine an equation of the curve in terms of these dihedral invariants.

\section{Weighted moduli spaces and their heights}\label{sect-9}

Another way of identifying isomorphism classes of superelliptic curves is by using $SL_2(k)$-invariants. From Hilbert's basis theorem the coordinate ring of degree $d$ binary forms is finitely generated.  Assume for example that $J_{q_0}, \dots , J_{q_n}$ are the generators.  Then two superelliptic curves $\X$ and $\X^\prime$ are isomorphic if and only if 
\[ J_{q_i} (\X) = \l^{q_i} J_{q_i} (\X^\prime), \quad \text{for } \quad i=0, \dots , n.\]
Hence, the isomorphism classes of superelliptic curves correspond to tuples $(J_{q_0}, \dots , J_{q_n})$ up to "multiplication" by a constant.  But these are exactly points in the weighted projective spaces, which motivates this section.

\subsection{Introduction to weighted moduli spaces}
Let $K$ be a field   and  $(q_0, \dots , q_n) \in \Z^{n+1}$ a fixed tuple of positive integers called \textbf{weights}.   Consider the action of $K^\star = K \setminus \{0\}$ on $\A^{n+1} (K)$ as follows
\begin{equation}\label{equivalence}
 \lambda \star (x_0, \dots , x_n) = \left( \l^{q_0} x_0, \dots , \l^{q_n} x_n   \right) 
\end{equation}
for $\l\in K^\ast$.  The quotient of this action is called a \textbf{weighted projective space} and denoted by   $\wP^n_{(q_0, \dots , q_n)} (K)$. 
The space $\wP_{(1, \dots , 1)} (K)$ is the usual projective space.  The space $\wP_w^n$ is called \textbf{well-formed} if   
\[ \gcd (q_0, \dots , \hat q_i, \dots , q_n)   = 1, \quad \text{for each } \;  i=0, \dots , n. \]
While most of the papers on weighted projective spaces are on well-formed spaces, we do not assume a well-formed space here.   We will denote a point $\p \in \wP_w^n (K)$ by $\p = [ x_0 : x_1 : \dots : x_n]$.  For more on weighted projective spaces one can check \cite{MR879909}, \cite{MR627828}, \cite{MR2852925}, \cite{igor} among many others.

\subsection{Graded rings} \label{graded-rings} 

In projective spaces, by means of the Veronese embedding, we could embed the same variety in different projective spaces. It turns out that we can do the same for varieties embedded in weighted projective spaces.

As above we let $k$ be a field.  Let $R = \oplus_{ i \geq 0} R_i$ be a graded ring.  We further assume that 
\begin{itemize}
\item[(i)] $R_0=k$ is the ground field

\item[(ii)]  $R$ is finitely generated as a ring over $k$

\item[(iii)]  $R$ is an integral domain
\end{itemize}

Consider the polynomial ring $k[x_0, \dots , x_n]$ where each $x_i$ has weight $\wt x_i = q_i$.  Every polynomial is a sum of monomials $x^m= \prod x_i^{m_i}$ with weight $\wt (x^m) = \sum m_i q_i$.  A polynomial $f$ is \textbf{weighted homogenous of weight $m$} if every monomial of $f$ has weight $m$.  

An ideal in a graded ring  $I \subset R$ is called \textbf{graded} or \textbf{weighted homogenous} if $I = \oplus_{n\geq 0} I_n$, where $I_n = I\cap R_n$.  Hence, $R = k[x_0, \dots , x_n]/I$, where $\deg x_i = q_i$ and $I$ is a homogenous prime ideal. 

\subsection{Construction of $\Proj R$}
To the prime ideal $I$ corresponds an irreducible affine variety $CX= \Spec R = V_a (I) \subset \A^{n+1}$.

\begin{defi}\label{w.h.p}
A polynomial $f(x_0, \dots , x_n)$ is called \textbf{weighted homogenous} of degree $d$ if  it satisfies the following
\[  f(\l^{q_0} x_0, \l^{q_1} x_1, \dots , \l^{q_n} x_n) = \l^d f(x_0, \dots , x_n). \]
\end{defi}
Notice that the condition $f(P)=0$ is defined on the equivalence classes of \eqref{equivalence}. We define the quotient 
$V_a (I)\setminus\{0\}$ by the above equivalence by $V_h (I)$, where $h$ stands for homogenous.  Then, we denote 
$X= \Proj R = V_h (I) \subset \wP_{\w}^n(k)$.  It is a projective variety. Notice that $CX$ above is the \textbf{affine cone} over the projective variety $V_h (I)$. 

Next we will define truncated rings and determine the role that they play in the Veronese embedding. 

\subsection{Truncated rings}
Define the $d$'th truncated ring $R^{[d]} \subset R$ by
\[ R^{[d]} = \bigoplus_{d |n} R_n  = \bigoplus_{i \geq 0} R_{di}, \]
Hence, $R^{[d]}$ is a graded ring and the elements have degree $di$ in $R$ and  degree $i$ in $R^{[d]}$.   If $R$ is a graded ring then its subring $R^{[d]}$ is called the $d$-th Veronese subring.

For example, let $R = k [x, y]$ with $wt(x) = wt (y) = 1$. Then,
\[ R^{[2]} = \bigoplus_{i \geq 0} R_{2i} = \bigoplus_{i \geq 0} \left \{f(x, y) \in k[x, y]  \left | \frac{}{} \right. \mbox{deg }(f) = 2i \right \}. \]
Notice that the even degree polynomials in $k[x, y]$ are generated by $x^2$, $xy$, and $y^2$ hence we have that 
\[ R^{[2]} = k[x^2, xy, y^2] \cong k [ u, v, w ] \big / \< uw -v^2\rangle \]
Now, if we consider the projective spaces we have that 
\[ \Proj \, \,( k[x, y] )= \P_{(1, 1)} =  \P^1\]
while
\[ \Proj \, \, (k [ u, v, w ] \big / \< uw -v^2\rangle  )=  V( uw -v^2 ) \subseteq \P_{(1, 1, 1)} = \P^2 \,.\]
Hence, we have
\[\P^1 (k) =  \Proj ( k[x, y] ) \cong \Proj \, \,( k[x, y]^2 ) \subseteq \P^2 (k). \]
This is exactly the degree-2 Veronese embedding of \, $\P^1(k) \hookrightarrow \P^2 (k)$.  The truncation of graded rings in this case corresponds to the degree-$2$ Veronese embedding.

The proof of the following lemma can be found in \cite{igor}. 

\begin{lem} \label{proj_iso}
Let $R$ be a graded ring and $d \in \N$. Then,
\[ \Proj   R \cong \Proj   R^{[d]}  \,.\]
\end{lem} 
 
Using the above \cref{proj_iso} we can find a closed embedding of a weighted projective space $\wP_w$ into an ordinary projective space $\P^N$ with big enough $N$. There is a very ampleness condition that was described by Delorme in \cites{MR0404277, MR0404278}.

\begin{prop} 
Consider the weighted polynomial ring $R = k [x_0, \dots, x_n]$ , where the positive integers $q_0, \dots, q_n$ are the weights of $x_0, \dots, x_n$ and  $d= \gcd (q_0, \dots , q_n)$.  The following are true: 

i)  $R^{[d]}  = R$.  Thus, 
\[ \wP^n_{(q_0, \dots, q_n) }(R)  = \wP^n _{ \left(\frac{q_0}{d}, \dots, \frac{q_n}{d} \right)} (R).\]  

ii) Suppose that $q_0, \dots, q_n$ have no common factor, and that $d$ is a common factor of all $a_i$ for $i \neq j$ (and therefore coprime to $a_j$). Then the $d$'th truncation of $R$ is the polynomial ring 
\[ R^{[d]} = k [x_0, \dots, x_{j-1}, x_j^d, x_{j+1}, \dots, x_n].\]
Thus, in this case 
\[ \wP^n_{(q_0, \dots, q_n) } (R)=   \wP^n _{\left(\frac{q_0}{d}, \dots, \frac{q_{j-1}}{d}, q_j, \frac{q_{j+1}}{d}, \dots, \frac{q_n}{d}\right)}(R^{[d]}).\]

In particular by passing to a truncation $R^{[d]}$ of $R$ which is a polynomial ring generated by pure powers of $x_i$, we can always write any weighted projective space as a well formed weighted projective space. 
\end{prop}

 \begin{proof}
  i) If $d | q_i$ for all $i = 0, \dots, n$ then the degree of every monomial is divisible by $d$ and so part i) is obvious. Hence, the truncation does not change anything. 
 
 ii)  Since $d | q_i$ for every $i \neq j$ then $x_i \in \R^{[d]}$ for every $i \neq j$. But the only way that $x_j$ can occur in a monomial with degree divisible by $d$ is as a $d$'th power. Given
\[R = k [x_0, \dots, x_j, \dots,  x_n]\]
then 
\[R^{[d]} = k [x_0, \dots, x_j^d, \dots,  x_n]\]
and  
 \[\begin{split} 
 \wP^n_{(q_0, \dots, q_n) }(R) &= \Proj \, \,  k_w [x_0, \dots, x_j, \dots,  x_n] \iso\Proj \, \,  k_{w/d} [x_0, \dots, x_j^d, \dots,  x_n] \\
 & = \wP^n _{\left(\frac{q_0}{d}, \dots, \frac{q_{j-1}}{d}, q_j, \frac{q_{j+1}}{d}, \dots, \frac{q_n}{d}\right)}(R^{[d]}).
 \end{split}\]
 This completes the proof. 
 \end{proof}
 
Hence, the  above result  shows that any weighted projective space is isomorphic to a well formed weighted projective space.

\subsection{Heights on the weighted projective space}\label{sec-4}

Let $K$ be an algebraic number field and $[K:\Q]=n$. With $M_K$ we will denote the set of all absolute values in $K$.   For $v \in M_K$, the \textbf{local degree at $v$}, denoted $n_v$ is 
\[n_v =[K_v:\Q_v]\]
where $K_v, \Q_v$ are the completions with respect to $v$. 

The following are true for any number field $K$; see \cite{silv-book}*{pg. 171-172} for proofs.
Let $L/K$ be an extension of number fields, and let $v \in M_K$ be an absolute value on $K$. Then
\[\sum_{\substack{w \in M_L\\w|v}}[L_w:K_v]= [L:K]\]
is known as the \textbf{degree formula}. For  $x \in K^\star$ we have the \textbf{product formula}  
\begin{equation}\label{prod.formula}
\prod_{v\in M_K}  |x|^{n_v}_v  =  1.   
\end{equation}
%
Given a point $\p \in \P^n(\overline \Q)$ with   $\p=[x_0, \dots, x_n]$, the \textbf{field of definition} of $\p$ is 
\[\Q(\p)=\Q   \left(  \frac  {x_0} {x_j}, \dots , \frac {x_n} {x_j}   \right)\] 
for any $j$ such that $x_j \neq 0$.  Next we try to generalize some of these concepts for the space $\wP_{\w} (K)$, where $K$ is a number field.  

In \cite{mandili} and \cite{b-g-sh} was introduced the concept of weighted height, which we will briefly describe below. 

Let $\w=(q_0, \dots , q_n)$ be a set of heights and $\wP^n(K)$ the weighted  projective space  over  a number field $K$.   Let  $\p \in \wP^n(K)$ a point such that  $\p=[x_0, \dots , x_n]$. We define the  \textbf{multiplicative heigh}t of $P$   as  
\begin{equation}\label{def:height}
\wh_K( \p ) := \prod_{v \in M_K} \max   \left\{   \frac{}{}   |x_0|_v^{\frac {n_v} {q_0}} , \dots, |x_n|_v^{\frac {n_v} {q_n}} \right\}
\end{equation}
The \textbf{logarithmic height} of the point $\p$ is defined as follows
\[\wh^\prime_K(\p) := \log \wh_K(\p)=   \sum_{v \in M_K}   \max_{0 \leq j \leq n}\left\{\frac{n_v}{q_j} \cdot  \log  |x_j|_v \right\}.\]
Next we will give some basic properties of heights functions.
\begin{prop}Let $K$ be a number field and $\p\in \wP^n(K)$ with weights $w = (q_0, \dots, q_n)$. Then the following are true:

i) The height $\wh_K(\p)$ is well defined, in other words it does not depend on the choice of  coordinates of $\p$

ii) $\wh_K(\p) \geq 1$.
\end{prop}

Moreover, we have the following (see \cite{b-g-sh} for details. 
 
\begin{prop}\label{lem_1}
Let $\p \in \wP^n(K)$.  Then the following are true:

i) If $K=\Q$,  
\begin{equation}
\wh_\Q (\p)=  \max_{0 \leq j \leq n}\left\{\frac{}{}|x_j|^{1/q_j}_\infty \right\}.
\end{equation}
%

ii) Let $L/K$ be a finite extension. Then,
\begin{equation}
\wh_L(\p)=\wh_K(\p)^{[L:K]}.
\end{equation}
\end{prop}

\subsubsection{Absolute heights}

Using \cref{lem_1}, part ii),  we can define the height on $\wP^n(\overline \Q)$. The height of a point on $\wP^n(\overline \Q)$ is called the \textbf{absolute (multiplicative) weighted height} and is the function 
\[
\begin{split}
\awh: \wP^n(\bar \Q) & \to [1, \infty)\\
\awh(\p)&=\wh_K(\p)^{1/[K:\Q]},
\end{split}
\]
where  $\p \in \wP^n(K)$, for any $K$.  The \textbf{absolute (logarithmic) weighted height} on $\wP^n(\overline \Q)$  is the function 
\[
\begin{split}
\awh^\prime: \wP^n(\bar \Q) & \to [0, \infty)\\
\awh^\prime (\p)&= \log \, \wh (\p)= \frac 1 {[K:\Q]} \awh_K(\p).
\end{split}
\]

\begin{lem}\label{lem_galois_conj}
The height is invariant under Galois conjugation. In other words, for  $\p \in \wP^n(\overline \Q)$ and $\sigma \in G_{ \Q}$ we have $\wh (\p^\sigma) = \wh (\p)$. 
\end{lem}

\proof   Let $\p =[x_0, \dots, x_n] \in \wP^n(\overline \Q)$. Let $K$ be a finite Galois extension of $\Q$ such that $\p \in \wP^n(K)$. Let $\sigma \in G_\Q$. Then $\sigma$ gives an isomorphism 
\[\sigma: K \to K^\sigma\]
and also identifies the sets $M_K$, and $M_{K^\sigma}$ as follows
\[
\begin{split}
\sigma: M_K &\to M_{K^\sigma}\\
v &\to v^\sigma
\end{split}
\]
Hence, for every $x \in K$ and $v \in M_K$, we have $|x^\sigma|_{v^\sigma} = |x|_v$.   Obviously $\sigma$ gives as well an isomorphism 
\[\sigma: K_v \to K^\sigma_{v^\sigma}\]
Therefore $n_v= n_{v^\sigma}$, where $n_{v^\sigma} = [ K^\sigma_{v^\sigma}: \Q_v]$. Then 
\[
\begin{split}
\wh_{K^\sigma}(P^\sigma) &= \prod_{w \in M_{K^\sigma}} \max_{0 \leq i \leq n} \left\{\frac{}{}  |x_i^\sigma|_{w}^{n_{w}/q_i}\right\}\\
&=  \prod_{v \in M_{K}} \max_{0 \leq i \leq n} \left\{\frac{}{}  |x_i^\sigma|_{v^\sigma}^{n_{v^\sigma}/q_i}\right\}=  \prod_{v \in M_{K}} \max_{0 \leq i \leq n} \left\{\frac{}{}  |x_i|_v^{n_v/q_i}\right\}=\wh_K(\p)
\end{split}
\]
This completes the proof.
\qed
 %
 
The following is the equivalent of  Northcott's theorem for weighted projective spaces.

\begin{thm}\cite{b-g-sh}  \label{thm_finite}
Let $c_0$ and $d_0$ be constants and $\wP_w^n(\overline \Q)$ the weighted projective space with weights  $w = (q_0, \dots, q_n)$. Then the set 
\[\{\p \in \wP_w^n(\overline \Q): H(\p) \leq c_0 \text{ and } [\Q(\p):\Q] \leq d_0\}\]
contains only finitely many points. In particular for any number field $K$
\[\{\p \in \wP_w^n(K): \wh_K(\p) \leq c_0 \} \]
is a finite set.
\end{thm} 

The next result is the analogue of Kronecker's theorem for  heights on projective spaces.

\begin{lem}
Let $K$ be a number field,  and let $\p=[x_0: \dots: x_n] \in \wP^n_w(K)$, where $\w = (q_0, \dots, q_n)$. Fix any $i$ with $x_{i} \neq 0$. Then $\wh (\p)= 1$ if  the ratio $x_j/\xi_{i}^{q_j}$, where $\xi_i$ is the $q_i$-th root of unity of $x_i$, is a root of unity or zero for every $0 \leq j \leq n$ and $j \neq i$.
\end{lem}

\proof   Let  $\p=[x_0: \dots: x_i: \dots: x_n] \in \wP^n(K)$. Assume $x_i \neq 0$. Adjoin the $q_i$-th root of unity to $x_i$. Hence, let $x_i = \xi_i^{q_i}$  so that $wt(\xi_i) = 1$. Without loss of generality we can divide the coordinates of $\p$ by $\xi_i^{q_j}$, for $j \neq i$,  and then  we have 
$$\p=\left [\frac{x_0}{\xi_i^{q_0}}, \dots , 1, \dots, \frac{ x_n}{\xi_i^{q_n}}\right].$$ 
For simplicity let $\p=[y_0: \dots: 1: \dots: y_n]$.  If $y_l$   is a root of unity for every $0 \leq l \leq n$ and $l \neq i$ then $|y_l|_v=1$ for every $v \in M_K$. Hence, $\wh (\p)=1$.
\qed



\subsection{Polynomials in weighted projective spaces}

Next we give a brief description  of a weighted variety and then define the height on a weighted variety.  For more details on weighted projective varieties see \cites{igor, MR879909} among others.

As it turns out, in the same way  as in ordinary projective spaces, evaluating a polynomial at a point it's not well defined, but checking if a point is a zero of a polynomial is.  We will make this precise below. As above we let $k$ be a field. We define the polynomial ring in $n+1$ variables with weights $w =(q_0, \dots, q_n)$ as  $k_w [x_0, \dots, x_n] $ such that  $wt (x_i) =q_i$. 

This changes the grading of the ring but does not change the underlying $k$-algebra structure. So, $k_w[x_0, \dots, x_n]$ is a Noetherian ring.  We will write $k_w[x_0, \dots, x_n]_d \subset k_w[x_0, \dots, x_n]$, where  $w = (q_0, \dots, q_n)$,  to  mean the additive group of all weighted homogenous polynomials of degree $d$.

\begin{defi} Let $f(x_0, \dots, x_n) \in k[x_0, \dots, x_n]$ where $wt (x_i) = q_i$, for $i=0, \dots, n$.   A polynomial $f(x_0, \dots , x_n)$   is called a \textbf{weighted homogenous polynomial of degree $d$} if each monomial in $f$ is weighted of degree $d$, i.e. 
\[ f (x_0, \dots , x_n) = \sum_{i=1}^m a_i \prod_{j=0}^n x_j^{d_j}, \, \,  \, a_i \in k  \, \text{ and } \,  m \in \N \]
and for all $ 0 \leq i \leq n$, we have that 
\[\sum_{i=1}^n q_id_j = d \,.\]
\end{defi}

%
%
%

Consider the point $P=(a_0, \dots, a_n) \in \wP^n_w$ and $f(x_0, \dots, x_n) \in k_w[x_0, \dots, x_n]_d$. By definition $P =   (\lambda^{q_0} a_0, \dots, \lambda^{q_n}a_n)$ for any $\lambda \in \G_m$, and particularly we can assume that $\lambda  \neq 1$, then we have that
\[f(\lambda^{q_0} a_0, \dots, \lambda^{q_n}a_n) = f(a_0, \dots, a_n) \, \, \text{ if and only if } f(a_0, \dots, a_n) = 0 \,.\]
Thus, it is well defined to write $f(P) = 0$ for some $f(x_0, \dots, x_n) \in k_w[x_0, \dots, x_n]_d$ and $P \in \wP_w^n$. We say that an ideal is a \textbf{weighted homogenous ideal} if and only if every element of $f\in I$ can be written as 
\[ f = \sum_{i=0}^{\deg f} f_i\]
with $f_i \in k_w[x_0, \dots, x_n]_i\cap I$.  Given $I \triangleleft k_w[x_o,  \dots, x_n]$,  a weighted homogenous ideal, define the \textbf{weighted projective variety} by
\[V (I)  = \left  \{ P \in \wP_w^n \, \left |\, \frac{}{} \right. f(P) = 0 \, \, \text{ for all } \, \, f \in I \right \} \,.\]
Conversely, given $V \subset \wP_w^n$ define the \textbf{ideal associated to $V$} by
\[ I(V) = \left\{ f \in k_w[x_0, \dots, x_n] \, \left | \frac{}{} \right. \, f(P) = 0 \, \, \text{ for all } \, \,  p \in V   \right \} \,. \]
 In the next lemma we prove that $I(V)$ it is actually an ideal.

\begin{lem}  Let $V \subset \wP_w^n$ and define $I(V)$ as above  
\[ I(V) = \left\{ f \in k_w[x_0, \dots, x_n] \, \left | \frac{}{} \right. \, f(P) = 0 \, \, \text{ for all } \, \,  p \in V   \right \} \]
then $I(V)$ is a radical weighted homogenous ideal. 
\end{lem}

\begin{proof} Let $f$ and $g$ be two polynomials in $I(V)$. Then, $f(P) = g(P) = 0$ for all points $P \in V$, i.e. they both vanish at all points $P$ in the variety $V$ then, so does $f+g$ and $fh$ where $h$ is any polynomial in $I(V)$. Therefore, $I(V)$ is an ideal. 

Since, $k_w[x_0, \dots, x_n]$ is Noetherian then $I(V)$ is finitely generated, say  
\[ 
I(V)= \< f_1, \dots, f_n\rangle .
\]
But, $f_i \in k_w[x_0, \dots, x_n]$ for all $i$ and therefore every $f_i$ is weighted homogenous. Hence $I(V)$ is weighted homogenous since it is generated by finitely many weighted homogenous polynomials. 

Lastly let us prove that $I(V)$ is radical. Let $f^r \in I(V)$. Then, for all points $P \in V$ we have that $f^r (P) = 0$. But since $f \in k_w[x_0, \dots, x_n]$, which is an integral domain, then $f^r (P) = (f(P)^r = 0$ implies that $f(P) = 0$ for all $P\in V$. Therefore, $I(V)$ is radical. 

\end{proof} 

A weighted projective variety  is said to be irreducible if it has no non-trivial decomposition into subvarieties.  Weighted projective varieties are projective varieties.  Hence, we can define a Zariski topology for weighted projective varieties  $\wP_w^n$ which is given by defining the closed sets of $\wP_w^n$ to be those of the form $V(I)$ for   weighted homogenous ideal $I \subset k_w[x_0, \dots, x_n]$.

%
Let   $f(x_0, \dots , x_n)$   be a weighted homogenous polynomial of degree $d$, then each monomial in $f$ is weighted of degree $d$, i.e. 
\[ f (x_0, \dots , x_n) = \sum_{i=1}^m a_i \prod_{j=0}^n x_j^{d_j}, \, \,  \, a_i \in k  \, \text{ and } \,  m \in \N \]
and for all $ 0 \leq i \leq n$, we have that 
\[\sum_{i=1}^n q_id_j = d.\]
%

We  use lexicographic ordering to order the terms in a given polynomial, and 
\[ x_1 > x_2 > \dots >x_n. \]
  The \textbf{multiplicative height of $f$} is defined as follows 
\[\wh_K(f)= \prod_{v \in M_K}  |f|_v^{n_v}\]
where 
\[|f|_v := \max_j\left\{\frac{}{} |a_j|_v^{1/q_j} \right\}\]
for any absolute value $v$.  Hence, the \textbf{multiplicative height of a polynomial} is the height of its coefficients taken as coordinates in the weighted projective space.  The   \textbf{absolute multiplicative  height} is defined as follows
\[
\begin{split}
 H: \P^n(\Q) & \to [1, \infty)\\
 H(f)&=\wh_K(f)^{1/[K:\Q]}. 
 \end{split}
\]
 

\begin{thm}\label{pol_finite} 
Let  $F(x, y ) \in K_w [x, y]$, where $w = (q_0, q_1)$, be a given weighted homogenous polynomial. Then,  there are only finitely many polynomials $G(x,y)\in K_w [x, y]$ such that $\wh_K(G) \leq \wh_K(F)$.  
\end{thm}

\proof
Let 
\[F(x, y)=\sum_{\substack{i=(i_0, i_1)\in I \\ d= i_0\cdot q_0+i_1 \cdot q_1} } a_ix^{i_0}y^{i_1}\]
be a  polynomial with coefficients in $K$ and  fix an ordering $x > y$.  Let  $\wh_K(F)=c$. By  definition  
\[\wh_K(F) =  \prod_{v \in M_K}  |f|_v^{n_v}= \prod_{v \in M_K} \max_i\left\{\frac{}{} |a_i|_v^{n_v}\right\} = \wh_K[a_0   \dots, a_i, \dots ]_{i\in I}.\]
But, $P=[a_0   \dots, a_i, \dots ]_{i\in I}$ is a point in $\P^{s}$ where $s$ is the number of monomials of degree $d$ in 2 variables. Hence,  $s=\left( 
\begin{array}{c}
d+1 \\ 
d
\end{array}
\right). $
From \cref{thm_finite} we have  that  for any constant  $c$ the set 
\[\{P \in \P^s(K): \wh_K(\p) \leq c \}\]
is finite. Hence there are finitely many polynomials $G(x, y)$ with content 1 corresponding to points $P$ with height   
$\wh_K(G) \leq c=\wh_K(F)$.

\qed


Next we see an application of the weighted projective spaces which was the main motivation for the definition of the weighted gcd's and the height in such spaces.

\subsection{Space of binary forms as weighted projective spaces}

It turns out that the space of degree-$d$ binary forms is a weighted projective space.

\def\RR{\mathcal R}

To start, let us determine what happens to the invariants when we change the coordinates, in other words when we act on the binary form $g(x, y)$ via $M \in GL_2( k)$. 
Let $I_0, \dots  , I_n$ be the generators of $\RR_d$ with degrees $q_0, \dots , q_n$ respectively.  
We denote the tuple of invariants by $\I := (I_0, \dots , I_n)$.  The following result is fundamental to our approach. 

\begin{prop}\label{prop-2}
For any two binary formal $f $ and $g$,   and    $M\in GL_2 (k)$,   
$g=f^M$  if and only if 
\[ 
\left( I_0 (g), \dots   I_i (f), \dots , I_n (g) \right) = \left( \l^{q_0} \, I_0 (f), \dots ,    \l^{q_i}\,  I_i (f), \dots , \l^{q_n} \, I_n (f),    \right), 
\] 
where   $\l = \left( \det M \right)^{\frac d 2}$.
\end{prop}

\proof  Let  $f(x, y) = \sum_{i=0}^d a_i x^i y^{d-i}$ be a degree $d\geq 2$ binary form and $I_s$ be an invariant of degree $s$ in $\RR_d$, say
\[ 
I_s = \sum  a_0^{\a_0} \ldots a_d^{\a_d}
\]
where $\a_i=0, \ldots , s$. When we evaluate $I (f^M) = I (f(ax+by, cx+dy)$ we have

\qed

\begin{cor} Let $I_0,  I_1,   \dots,  I_n$ be the generators of the ring of invariants $\RR_d$ of degree $d$ binary forms. 
A $k$-isomorphism class of a binary form $f$ is determined by  the point 
\[ \I (f) := \left[ I_0 (f), I_1 (f), \dots , I_n(f) \right]  \in \wP_\w^n  (k). \]
Moreover $f=g^M$ for some $M\in GL_2 (K)$ if and only if  $ \I(f)  = \lambda\star \I(g)$, 
for $\lambda = \left( \det A \right)^{\frac d 2}$.
\end{cor}

Since the isomorphism class of any superelliptic curve, given by
\begin{equation}\label{w-eq-super} 
\X : z^m y^{d-m}= f(x, y) \,,
\end{equation}
 is determined by the equivalence  class of binary form $f(x, y)$ we denote the set of invariants of $\X$ by $\I (\X) : = \I (f)$. Therefore, we have:
 
\begin{cor}
Let $\X$ be a superelliptic curve given by \cref{w-eq-super}. The $\bar k$-isomorphism class of $\X$ is determined by the weighted moduli point  $\p := \left[ \I (f) \right]  \in \wP_\w^n  (k)$. 
\end{cor}

\section{Minimal models}\label{sect-10}

Let $k$ be an algebraic number field   and  $\O_k$ its ring of integers.  The isomorphism class of a  smooth, irreducible algebraic curve $C$ defined over $\O_k$ its determined by its set invariants which are  homogenous polynomials in terms of the coefficients of $\X$.  When $\X$ is a superelliptic curve  then these invariants are generators of the invariant ring of binary forms of fixed degree.



If $C$ is a hyperelliptic curve over $k$, then 
 the discriminant of $C$ is a polynomial given in terms of the coefficients of the curve. Hence,    it is an ideal in the ring of integers $\O_k$. The valuation of this ideal is a positive integer.   A classical question is to find an equation of the curve such that this valuation is minimal, in other words the discriminant is minimal.  

The simplest example is for   $C$ being an elliptic curve.  There is an extensive theory  of the minimal discriminant ideal $\dis_{C/K}$.  Tate \cite{ta-75}  devised an algorithm how to determine the Weierstrass equation of an elliptic curve with minimal discriminant as part of his larger project of determining Neron models for elliptic curves.  An implementation of this approach for elliptic curves was done by Laska in \cite{la-82}.  Tate's approach was extended to genus 2 curves by Liu \cite{liu} for genus 2, and to all hyperelliptic curves by Lockhart \cite{lockhart}. 

\subsection{Minimal discriminants over local fields}
Let $K$ be a local field, complete with respect to a valuation $\v$.  Let $\O_K$ be the ring of integers of $K$, in other words 
$\O_K = \{ x\in K \, | \, \v (x) \geq 0\}$. 
We denote by $\O_K^\ast$ the group of units of $\O_K$ and by $\m$ the maximal ideal of $\O_K$.  Let $\pi$ be a generator  for $\m$ and $k=\O_K / \m$ the residue field. We assume that $k$ is perfect and denote its algebraic closure by $\bar k$. 

Let $\X_g$ be a superelliptic curve of genus $g \geq 2$ defined over $K$ and $P$ a $K$-rational point on $\X_g$.     By a suitable change of coordinates we can assume that all coefficients of $\X_g$ are in $\O_K$.      Then, the discriminant $\D \in \O_K$.  In this case we say that the equation of $\X_g$ is \textbf{integral}.

An equation for $\X_g$ is said to be a \textbf{minimal equation}  if it is integral and $\v (\D)$ is minimal among all integral equations of $\X_g$. The ideal $I=\m^{\v (\D)}$ is called the \textbf{minimal discriminant} of $\X_g$. 
 

\subsection{Minimal discriminants over global fields}
Let us assume now that $K$ is an algebraic number field with field of integers $\O_K$.  Let $M_K$ be the set of all inequivalent absolute values on $K$  and $M_K^0$ the set of all non-archimedean absolute values in $M_K$. 
We denote by $K_\v$ the completion of $K$ for each $\v \in M_K^0$ and by $\O_\v$ the valuation ring in $K_\v$. Let $\p_v$ be the prime ideal in $\O_K$ and $\m_v$ the corresponding maximal ideal in $K_\v$. Let $(\X, P)$ be a superelliptic curve of genus $g\geq 2$ over $K$.  

If $\v \in M_K^0$ we say that $\X$ is \textbf{integral at $\v$} if $\X$ is integral when viewed as a curve over $K_\v$.  We say that $\X$ is \textbf{minimal at $\v$} when it is minimal over $K_\v$. 

An equation of $\X$ over $K$ is called \textbf{integral} (resp. \textbf{minimal}) over $K$ if it is integral (resp. minimal) over $K_\v$, for each $\v \in M_K^0$. 
 
Next we will define the minimal discriminant over $K$ to be the product of all the local minimal discriminants. For each $\v \in M_K^0$ we denote by $\D_\v$ the minimal discriminant for $(\X, P)$ over $K_\v$.  The \textbf{minimal discriminant} of $(\X, P)$ over $K$ is the ideal 
\[ \D_{\X / K} = \prod_{\v \in M_K^0} \m_\v^{\v (\D_\v) } \,.\]
We denote by $\fa_\X$ the ideal   $ \fa_\X = \prod_{\v \in M_K^0} \p_\v^{\v (\D_\v) }$. 
%
\begin{thm}
Let $(\X_g, P)$ be a superelliptic curve  over $\Q$. Then its global minimal  discriminant $\D\in \Z$  is unique (up to multiplication by a unit).  There exists a minimal Weierstrass equation corresponding to this $\D$.
\end{thm}
%
 
\begin{rem} 
In general ($K$ an algebraic number field) with class number $> 1$, then the curve may not have a minimal Weierstrass equation.
\end{rem}

\subsubsection{Elliptic curves and Tate's algorithm}
Let $E$  be an elliptic curve defined over a number field $K$ with equation
\begin{equation}\label{w-eq-ell}     y^2 + a_1 xy + a_3  y = x^3 + a_2 x^2 + a_4  x +a_6. \end{equation}
For simplicity we assume that $E$ is defined over $\Q$; the algorithm works exactly the same for   any algebraic number field $K$. 

We would like to find an equation
\begin{equation}\label{w-eq-2}     y^2 + a_1^\prime xy + a_3^\prime y = x^3 + a_2^\prime x^2 + a_4^\prime x +a_6^\prime. \end{equation}
such that the discriminant $\D^\prime$ of the curve in  \cref{w-eq-2} is minimal.  Since we want the new equation to have integer coefficients, the only transformations we can carry out are
\[ x = u^2 x^\prime + r, \qquad y = u^3 y^\prime + u^2 s x^\prime + t\]
for $u, r, s, t \in \Z$ and $u \neq 0$.  The coefficients of the two equations are related as follows:
\[
\begin{split} 
& ua_1^\prime  = a_1 + 2s,                          \\
& u^3 a_3^\prime  = a_3 + r a_1 + 2 t,               \\          
& u^2 a_2^\prime  = a_2 - sa_1 + 3 r - s^2,             \\
\end{split}
\qquad
\begin{split}
                   & u^4 a_4^\prime  = a_4 - s a_3 + 2 r a_2 - (t+rs)a_1 + 3 r^2 - 2st    \\
      & u^6 a_6^\prime  = a_6 + r a_4 + r^2 a_2 + r^3- ta_3 - rta_1 - t^2   \\          
    & u^{12} \D^\prime  =  \D                                              \\
\end{split}
\]
The version of the algorithm below is due to M. Laska; see \cite{la-82}. \\

\noindent \textsc{Step 1:} Compute the following
\[
\begin{split}
 c_4 & = (a_1^2+4a_2)^2 - 24(a_1a_3 +2a_4), \\
  c_6  & = - (a_1^2+4a_2)^3 + 36(a_1^2 + 4 a_2)(a_1a_3 + 2 a_4) - 216 (a_3^2 +4a_6) 
\end{split}
\]
\noindent  \textsc{Step 2:} Determine the set $S$ of integers $u \in \Z $ such that there exist $x_u$, $y_u \in \Z$ such that 
$ u^4 = x_u c_4$ and $ u^6 y_u = c_6$.  Notice that   $S$ is a finite set. \\

\noindent  \textsc{Step 3:} Choose the largest $u \in S$, say $u_0$ and factor it as $u_0 = 2^{e_2} \, 3^{e_3} \, v$, where $v$ is relatively prime to 6.  \\

\noindent  \textsc{Step 4:}  Choose 
\[
a_1^\prime, a_3^\prime \in \left\{ \sum_{i=1}^n \alpha_i w_i \, | \, \alpha_i = 0 \, \textbf{ or }  1 \, \right\} \, \textbf{ and } \,
a_2^\prime \in \left\{ \sum_{i=1}^n \alpha_i w_i \, | \, \alpha_i = -1, 0 \, \textbf{ or }  1 \, \right\}
\]
subject to the following conditions:
\[ (a_1^\prime)^4 \equiv x_u \mod 8, \quad (a_2^\prime)^3 \equiv - (a_1^\prime)^6 - y_u \mod 3.  \]

\noindent  \textsc{Step 5:}  Solve the following equations for $a_4^\prime$ and $a_6^\prime$
\[
\begin{split}
 x_u & = ({a_1^\prime}^2+4{a_2^\prime})^2 - 24({a_1^\prime}{a_3^\prime} +2{a_4^\prime}), \\
 y_u  & = - ({a_1^\prime}^2+4{a_2^\prime})^3 + 36({a_1^\prime}^2 + 4 {a_2^\prime})({a_1^\prime}{a_3^\prime} + 2 {a_4^\prime}) - 216 ({a_3^\prime}^2 +4{a_6^\prime}) 
\end{split}
\] 

\noindent  \textsc{Step 6:} Solve the equations for $s, r, t$ successively 
\[ u a_1^\prime = a_1 + 2s, \quad u^2 a_2^\prime = a_2 - s a_1 + 3 r - s^2, \quad u^3 a_3^\prime = a_3 + r a_1 + 2t \]
For these values of $a_1^\prime, \dots , a_6^\prime$  the  \cref{w-eq-2}  is the desired result.  

For a complete version of the algorithm see \cite{la-82}.


\subsection{Superelliptic curves with minimal weighted moduli point}

Now we will consider the minimal models of curves over $\O_k$.  Let $\X$ be as in \cref{w-eq-super}  and $\p = [\I(f)]\in \wP_\w^n (k)$.  
Let us assume that for a prime $p\in \O_k$, we have $\nu_p \left( \wgcd (\p) \right) = \alpha$.  If we use the transformation $x\to \frac x {p^\beta} x$, for $\beta \leq \alpha$, then from \cref{prop-2}  the invariants will be transformed according to
\[ \frac 1 {p^{ \frac d 2 \beta }}\star \I(f) \,.\]
To ensure that the moduli point $\p$ still has integer coefficients we must pick $\beta$ such that $p^{\frac {\beta d} 2}$ divides  $p^{\nu_p (x_i)}$ for $i= 0, \dots , n$.   Hence, we must pick $\beta$ as the maximum integer such that $\beta \leq \frac 2 d \nu_p (x_i)$, for all $i=0, \dots , n$. 
%
The transformation 
\[ (x, y) \to \left( \frac x {p^\beta} , y \right), \]
has a corresponding Jacobian matrix $M=\begin{bmatrix} \frac 1 {p^\beta} & 0 \\ 0 & 1 \end{bmatrix}$ with $\det M = \frac 1 {p^\beta}$.  Hence, \cref{prop-2} implies that the moduli point $\p$ changes according to
$  \p \to \left( \frac 1 {p^\beta} \right)^{d/2} \star \p $,  
which is still an integer tuple.  We can repeat this this for all primes $p$ dividing $\wgcd (\p)$.     Notice that the new point is not necessarily normalized in $\wP_\w^n (k)$ since $\beta $ is not necessarily equal to $\alpha$.  
This motivates the following definition. 

\begin{defi}
Let $\X$ be a superelliptic curve defined over an integer ring $\O_k$ and $\p \in \wP_\w^n (\O_k)$ its corresponding weighted moduli point. We say that $\X$ has a \textbf{minimal model} over $\O_k$ if for every prime $p \in \O_k$ the \textbf{valuation of the tuple} at $p$
\[
\val_p (\p) := \max \left\{  \nu_p (x_i)   \text{ for all }  i=0, \ldots n    \right\},
\]
is minimal, where $\nu_p (x_i)$ is the valuation of $x_i$ at the prime $p$. 
\end{defi}
%
The following is proved in \cite{super-min}.

\begin{thm}\label{thm-4}  Minimal models of superelliptic curves exist.  
In particular, the equation given by $\X :$ $ z^m y^{d-m}= f(x, y)$ is a minimal model over $\O_k$,   if   for every prime $p \in \O_k$ which divides  
$ p \, | \,  \wgcd \left( \I(f) \right)$,  the valuation $\val_p $  of $\I(f)$ at $p$ satisfies  
\begin{equation}\label{val}
\val_p (\I(f))  <   \frac d 2  \,  q_i,
\end{equation}
for all $i=0, \dots , n$.  Moreover,  then for  $\lambda=\wgcd (\I (f)) $ with respect  the weights
$\left( \left\lfloor  \frac {dq_0} {2} \right\rfloor,\ldots, \left\lfloor  \frac {dq_n} {2} \right\rfloor \right) $,    
the transformation 
\[ (x, y, z) \to \left(      \frac x \lambda, y, \lambda^{\frac d m} z \right) \] 
 gives the minimal model of  $\X$ over $\O_k$. If $m|d$ then this isomorphism is defined over $k$. 
\end{thm}

   Let $\X$ be a superelliptic curve given by \cref{w-eq-super} over $\O_k$ and   $\p = \I(f) \in \wP_\w^n (\O_k)$ with weights $\w=(q_0, \dots , q_n)$.  Then $\p \in \wP_\w^n (\O_k)$ and  exists $M\in SL_2 (\O_k)$ such that $M= \begin{bmatrix}  \frac 1 \lambda  & 0 \\ 0 & 1  \end{bmatrix} $ and $\lambda$ as in the theorem's hypothesis, and \cref{val} holds; see \cite{super-min} for details. 

Let us also determine how the equation of the curve $\X$ changes when we apply the transformation by $M$. We have 
%
\[ 
z^m y^{d-m} = f  \left(\frac x \lambda , y \right) =  a_d \frac {x^d} {\lambda^d} + a_{d-1} \frac {x^{d-1}} {\lambda^{d-1}} y + \cdots + a_1 \frac x \lambda y^{d-1} + a_0 y^d \,,
\]
or, equivalently,
\begin{equation}\label{eq-2}
\X^\prime : \;  \lambda^d  z^m y^{d-m} = a_d  x^d + \lambda a_{d-1}  x^{d-1}  y + \cdots + \lambda^{d-1} a_1 x   y^{d-1} + \lambda^d a_0  y^d \,.
\end{equation}
This equation has coefficients in $\O_k$.  Its weighted moduli point is 
\[ \I (f^M) = \frac 1 {\lambda^{\frac d 2}}\star \I(f), \]
and satisfies \cref{val}.   It is a twist of the curve $\X$ since $\lambda^d$ is not necessary a $m$-th power in $\O_k$. The isomorphism of the curves over the field  $k \left(\lambda^{\frac d m}\right)$ is given by 
\[ (x, y, z) \to \left( \frac x \lambda, y, \lambda^{\frac d m} z   \right) \,.
\]
If $m|d$ then this isomorphism is defined over $k$  and $\X^\prime$ has equation 
\[ 
\X^\prime : \;   z^m  y^{d-m} =  a_d  x^d + \lambda a_{d-1}  x^{d-1}  y + \cdots + \lambda^{d-1} a_1 x   y^{d-1} + \lambda^d a_0  y^d \,.
\]
Thus, we have the following:

\begin{cor}
There exists a  curve $\X^\prime$ given in \cref{eq-2}  isomorphic to $\X$ over the field $K:=k \left( \wgcd( \p)^{\frac d m}\right)$ with minimal   $SL_2 (\O_k)$-invariants.    Moreover, if $m  |  d$ then $\X$ and $\X^\prime$ are $k$-isomorphic.
\end{cor}

An immediate consequence of the above is that in the case of hyperelliptic curves we have $m=2$ and $d=2g+2$.  Hence,  the curves $\X$ and $\X^\prime$ are always isomorphic over $k$.  We have the following:

\begin{cor} Given a hyperelliptic curve  defined over a ring of integers $\O_k$. There exists a  curve $\X^\prime$  $k$-isomorphic to $\X$   with minimal   $SL_2 (\O_k)$-invariants.  
\end{cor}

We give a detailed account of superelliptic curves with minimal invariants in \cite{super-min}.

\section{Field of moduli}\label{sect-11}


Let $\X$ be a genus $g$ projective, irreducible, algebraic curve defined over $k$, say given as the common zeroes of the polynomials $P_{1},\ldots, P_{r}$, and let us denote by $G=\Aut (\X)$ the full automorphism group of $\X$.  
If $\sigma \in {\rm Gal}(k)$, then $X^{\sigma}$ will denote the curve defined as the common zeroes of the polynomials $P_{1}^{\sigma},\ldots,P_{r}^{\sigma}$, where $P_{j}^{\sigma}$ is obtained from $P_{j}$ by applying $\sigma$ to its coefficients. In particular, if $\tau$ is also a field automorphism of $k$, then $X^{\tau \sigma}=(X^{\sigma})^{\tau}$. For details we refer to \cite{h-sh}.


A subfield $k_{0}$ of $k$ is called a \textbf{field of definition} of $\X$ if there is a curve ${\mathcal Y}$, defined over $k_{0}$, which is isomorphic to $\X$. It is clear that every subfield of $k$ containing $k_{0}$ is also a field of definition of it. In the other direction, a subfield of $k_{0}$ might not be a field of definition of $\X$.
Weil's descent theorem \cite{Weil} provides sufficient conditions for a subfield $k_{0}$ of $k$ to be a field of definition. Let us denote by ${\rm Gal}(k/k_{0})$ the group of field automorphisms of $k$ acting as the identity on $k_{0}$.

\begin{thm}[Weil's descent theorem \cite{Weil}]
Assume that for every $\sigma \in {\rm Gal}(k/k_{0})$ there is an isomorphism $f_{\sigma}:\X \to \X^{\sigma}$ so that
$$f_{\tau\sigma}=f_{\sigma}^{\tau} \circ f_{\tau}, \quad \forall \sigma, \tau \in {\rm Gal}(k/k_{0}).$$
Then there is a curve  ${\mathcal Y}$, defined over $k_{0}$, and there is an isomorphism $R:\X \to {\mathcal Y}$, defined over a finite extension of $k_{0}$, so that $R=R^{\sigma} \circ f_{\sigma}$, for every $\sigma \in {\rm Gal}(k/k_{0})$.
\end{thm}

Clearly, the sufficient conditions in Weil's descent theorem are trivially satisfied if $\X$ has non-trivial automorphisms. This is the generic situation for $\X$ of genus at least three.

\begin{cor}\label{coro:weil}
If $\X$ has trivial group of automorphisms and for every $\sigma \in {\rm Gal}(k/k_{0})$ there is an isomorphism $f_{\sigma}:\X \to \X^{\sigma}$, then $\X$ can be defined over $k_{0}$.
\end{cor}


The notion of field of moduli was originally introduced by Shimura for the case of abelian varieties and later extended to more general algebraic varieties by Koizumi.
If $G_{\X}$ is the subgroup of ${\rm Gal}(k)$ consisting of those $\sigma$ so that $\X^{\sigma}$ is isomorphic to $\X$, then the fixed field $M_{\X}$ of $G_{\X}$ is called \textbf{the field of moduli} of $\X$.
As we are assuming that $k$ is algebraically closed and of characteristic zero, we have that $G_{\X}$ consists of all automorphisms of ${\rm Gal}(k)$ acting as the identity on $M_{\X}$.

Every curve of genus $g \leq 1$ can be defined over its field of moduli. If $g \geq 2$, then  there are known examples of curves which cannot be defined over their field of moduli.  A direct consequence of   \cref{coro:weil} is the following.

\begin{cor}\label{corotrivial}
Every curve with trivial group of automorphisms can be defined over its field of moduli.
\end{cor}

As a consequence of Belyi's theorem \cite{Belyi}, every quasiplatonic curve $\X$ can be defined over $\overline{\mathbb Q}$ (so over a finite extension of ${\mathbb Q}$).

\begin{thm}[Wolfart \cite{Wolfart}]\label{Wolfart}
Every quasiplatonic curve can be defined over its field of moduli (which is a number field).
\end{thm}

\subsection{Two practical sufficient conditions}
When the curve $\X$ has a non-trivial group of automorphisms, then Weil's conditions (in Weil's descent theorem) are in general not easy to check. Next we consider certain cases for which it is possible to check for $\X$ to be definable over its field of moduli.

\medskip   \noindent \textbf{Sufficient condition 1: unique subgroups}
Let $H$ be a subgroup of $\Aut(\X)$. In general there might other different subgroups $K$ which are isomorphic to $H$ and with $\X/K$ and $\X/H$ having the same signature. For instance, the genus-two curve $\X$ defined by $y^{2}=x(x-1/2)(x-2)(x-1/3)(x-3)$ has two conformal involutions, $\tau_{1}$ and $\tau_{2}$, whose product is the hyperelliptic involution. The quotient $\X/\langle \tau_{j}\rangle$ has genus one and exactly two cone points (of order two). 
We say that $H$ is 
is \textbf{unique} in $\Aut(\X)$ if it is the unique subgroup of $\Aut(\X)$ isomorphic to $H$ and with quotient orbifold of same signature as  $\X/H$. Typical examples are (i) $H=\Aut(\X)$ and (ii) $H$ being the cyclic group generated by the hyperelliptic involution for the case of hyperelliptic curves. 
If $H$ is unique in $\Aut(\X)$, then it is a normal subgroup; so we may consider the reduced group $\bAut(\X)=\Aut(\X)/H$, which is a group of automorphisms of the quotient orbifold $\X/H$. In \cite{HQ} the following sufficient condition for a curve to definable over its field of moduli was obtained;

\begin{thm}\label{thm1}
Let $\X$ be a curve of genus $g \geq 2$ admitting a subgroup $H$ which is unique in $\Aut(\X)$ and so that $\X/H$ has genus zero.  If the reduced group of automorphisms $\bAut(\X)=\Aut(\X)/H$ is different from trivial or cyclic, then $\X$ is definable over its field of moduli.
\end{thm}

If $\X$ is a hyperelliptic curve, then a consequence of the above is the following result.

\begin{cor}\label{cor1b}
Let $\X$ be a hyperelliptic curve with extra automorphisms and reduced automorphism group $\bAut (\X)$ not isomorphic to a cyclic group.  Then, the field of moduli of $\X$  is a field of definition. 
\end{cor}

\medskip   \noindent \textbf{Sufficient condition 2: Odd signature}
Another sufficient condition of a curve $\X$ to be definable over its field of moduli, which in particular contains the case of quasiplatonic curves, was provided in \cite{AQ}. We say that $\X$ has \textbf{odd signature} if $\X/\Aut(\X)$ has genus zero and in its signature one of the cone orders appears an odd number of times.

\begin{thm}\label{thm2}
Let $\X$ be a curve of genus $g \geq 2$. If $\X$ has odd signature, then it can be defined over its field of moduli.
\end{thm}

\subsection{The locus of curves with prescribed group action, moduli dimension of families}
Fix an integer $g\ge2$ and a finite group $G$. Let $C_1, \dots , C_r$ be nontrivial conjugacy classes of $G$. Let $\bC=(C_1, \dots ,C_r)$, viewed as an unordered tuple, where repetitions are allowed. We allow $r$ to be zero, in which case $\bC$ is empty. Consider pairs $(\X, \mu)$, where $\X$ is a curve and $\mu: G  \to  \Aut(\X)$ is an injective
homomorphism. We will suppress $\mu$ and just say $\X$ is a curve with $G$-action, or a $G$-curve. Two $G$-curves $\X$ and $\X'$ are called equivalent if there is a $G$-equivariant  conformal isomorphism $\X\to \X'$.
We say a $G$-curve $\X$ is \textbf{of ramification type} $(g, G, \bC)$ (for short, of type $(g,G,\bC)$)
if 
\begin{enumerate}
\item[i)] $g$ is the genus of $\X$, 
\item[ii)] $G<\Aut(\X)$,
\item[iii)]  the points of the quotient $\X/G$ that are ramified in the cover $\X\to \X/G$ can be labeled as $p_1, \dots ,p_r$ such that $C_i$ is the conjugacy class in $G$ of distinguished inertia group
generators over $p_i$ (for $i=1, \dots ,r$). 
\end{enumerate}

If $\X$ is a $G$-curve of type $(g,G,\bC)$, then the genus $g_0$ of $\X/G$ is given by the Riemann-Hurwitz formula
$$2(g-1)=2|G|(g_{0}-1)+|G|\sum_{j=1}^{r}(1-|C_{j}|^{-1}).$$

Define $\H=\H(g,G,\bC)$ to be the set of equivalence classes of $G$-curves of type $(g,G,\bC)$. By covering space theory, $\H$ is non-empty if and only if $G$ can be generated by elements $\a_1,\b_1,\dots ,\a_{g_0},\b_{g_0},\g_1,\dots ,\g_r$ with $\g_i\in C_i$ and $\prod_{j}\ [\a_j,\b_j]\  \prod_{i} \g_i =  1$, where $[\a,\b]=\ \a^{-1}\b^{-1}\a\b$.

Let $\M_g$ be the moduli space of genus $g$ curves, and $\M_{g_0,r}$ the moduli space of genus $g_0$ curves with $r$ distinct marked points, where we view the marked points as unordered. Consider the map 
\[ \Phi:\ \H\ \to \ \M_{g},\]
forgetting the $G$-action, and the map $\Psi:\ \H \ \to \ \M_{g_0,r} $ mapping (the class of) a $G$-curve $\X$ to the class of the quotient curve $\X/G$ together with the (unordered) set of branch points $p_1, \dots , p_r$. 
If $\H \ne \emptyset$, then $\Psi$ is surjective and has finite fibers, by covering space theory. Also $\Phi$ has finite fibers, since the automorphism group of a
curve of genus $\ge2$ is finite. The set $\H$ carries a structure of quasi-projective variety (over $\C$) such that the maps $\Phi$ and $\Psi$ are finite morphisms. 
If $\H\ne\emptyset$, then all (irreducible) components of $\H$ map surjectively to $\M_{g_0,r}$ (through a finite map), hence they all have the same dimension
\[ \d     (g,G,\bC):= \ \ \dim\ \M_{g_0,r} \ \ = \ \ 3g_0-3+r\]  
Let $\M(g,G,\bC)$ denote the image of $\Phi$, i.e., the locus of genus $g$ curves admitting a $G$-action of type $(g,G,\bC)$. 
 Since $\Phi$ is a finite map, if this locus is non-empty, each of its (irreducible) components has dimension $\d (g, G,\bC)$.
 \cref{Wolfart} can be stated as follows:
\begin{thm}
If $\d (g, G,\bC) = 0$, then every curve in $\M (g, G, \bC) $ is defined over its field of moduli.
\end{thm}
The last part of the above is due to the fact that $\delta=0$ ensures that the quotient orbifold $\X/G$ must be of genus zero and with exactly three conical points, that is, $\X$ is a quasiplatonic curve.

\subsection{Field of moduli of superelliptic curves}\label{Sec:FOMSC}
%
Let $\X$ be a superelliptic curve of level $n$ with $G=\Aut(\X)$. By the definition, there is some $\tau \in G$, of order $n$ and central, so that the quotient $\X / \< \tau \>$ has genus zero, that is, it can be identified with the projective line, and all its cone points have order $n$.  As, in this case, the cyclic group $H=\< \tau \> \cong C_{n}$  is normal subgroup of $G$, we may consider the quotient group $\G \, := \, G/H$, called the \textit{reduced automorphism group of $\X$ with respect to $H$}; so  $G$ is a degree $n$ central extension of $\G$.

In the particular case that $n=p$ is a prime integer, Castelnuovo-Severi's inequality \cite{CS} asserts that for $g>(p-1)^{2}$ the cyclic group $H$ is unique in $\Aut(\X)$. 
The following result shows that the superelliptic group of level $n$ is unique:

\begin{thm}\label{unicidad}
A superelliptic curve of level $n$ and genus $g \geq 2$ has a unique superelliptic group of level $n$.
\end{thm}

\proof    Let $\X$ be a superelliptic curve of level $n$ and assume that $\langle \tau \rangle$ and $\langle \eta \rangle$ are two different superelliptic groups of level $n$.
The condition that the cone points of both quotient orbifolds $\X/\langle \tau \rangle$ and $\X/\langle \eta \rangle$ are of order $n$ asserts that a fixed point of a non-trivial power of $\tau$ (respectively, of $\eta$) must also be a fixed point of $\tau$ (respectively, $\eta$). In this way, our previous assumption asserts that no non-trivial power of $\eta$ has a common fixed point with a non-trivial power of $\tau$. In this case, the fact that $\tau$ and $\eta$ are central asserts that $\eta \tau=\tau \eta$ and that $\langle \tau, \eta \rangle \cong C_{n}^{2}$ (see also \cite{Sa}). 

Let $\pi: \X \to \P^{1}_{k}$ be a regular branched cover with $\langle \tau \rangle$ as deck group. Then the automorphism $\eta$ induces a automorphism $\rho \in \pgl_{2}(k)$ (also of order $n$) so that $\pi \eta = \rho \pi$. As $\rho$ is conjugated to a rotation $x \mapsto \omega_{n} x$, where $\omega_{n}^{n}=1$, we observe that it has exactly two fixed points. This asserts that $\eta$ must have either $n$ or $2n$ fixed points (forming two orbits under the action of $\langle \tau \rangle$). As this is also true by interchanging the roles of $\tau$ and $\eta$, the same holds for the fixed points of $\tau$. It follows that the cone points of $\pi$ consists of (i) exactly two sets of cardinality $n$ each one or (ii) exactly one set of cardinality $n$, and each one being invariant under the rotation $\rho$. Up to post-composition by a suitable transformation in $\pgl_{2}(k)$, we may assume these in case (i) the $2n$ cone points are given by the $n$ roots of unity and the $n$ roots of unity of a point different from $1$ and $0$ and in case (ii) that the $n$ cone points are the $n$ roots of unity. In other words, $\X$ can be given either as
\[ \X_{1}: \; y^{n}=(x^{n}-1)(x^{n}-a^{n}), \quad a \in k-\{0,1\}\]
or as the classical Fermat curve
\[ \X_{2}: \;y^{n}=x^{n}-1 \]
and, in these models, 
\[ \tau(x,y)=(x,\omega_{n} y), \quad \eta(x,y)=(\omega_{n} x, y).\]
As the genus of $\X_{1}$ is at least two, we must have that $n \geq 3$. But such a curve also admits the order two automorphism
\[ \gamma(x,y)=  \left(\frac{a}{x},\frac{ay}{x^{2}}  \right)\]
which does not commute with $\eta$, a contradiction to the fact that $\eta$ was assumed to be central.
In the Fermat case, the full group of automorphisms is $C_{n}^{2} \rtimes S_{3}$ and it may be checked that it is not superelliptic.

\qed

The group  $\G$ is a subgroup of the group of automorphisms of a genus zero field, so $\G <  \pgl_2(k)$ and $\G$ is finite. It is a classical result that every finite subgroup of $\pgl_2 (k)$ (since we are assuming $k$ of characteristic zero) is either the trivial group or isomorphic to one of the following: $C_m $, $ D_m $, $A_4$, $S_4$, $A_5$. All automorphisms groups of superelliptic curves and their equations were determined in \cite{Sa} and \cite{Sa-sh}. 
Determining the automorphism groups $G$, the signature $\bC$ of the covering $\X \to \X/G$, and the dimension of the locus $\M(g,G,\bC)$  for superelliptic curves is known; see \cite{Sa}. 
We have seen in  \cref{unicidad} that its superelliptic group of level $n$ is unique.
As a consequence of  \cref{thm1}, we obtain the following fact concerning the field of moduli of superelliptic curves:

\begin{thm}\label{teounico}
Let $\X$ be a superelliptic curve of genus $g \geq 2$ with superelliptic group $H \cong C_{n}$. If the reduced group of automorphisms $\bAut(\X)=\Aut(\X)/H$ is different from trivial or cyclic, then $\X$ is definable over its field of moduli.
\end{thm}

As a consequence of the above, we only need to consider the case when the reduced group $\G=G/H$ is either trivial or cyclic. As a consequence of  \cref{thm2} we have:

\begin{thm}\label{teoAQ}
Let $\X$ be a superelliptic curve of genus $g \geq 2$ with superelliptic group $H \cong C_{n}$ so that $\G=G/H$ is either trivial or cyclic. If $\X$ has odd signature, then it can be defined over its field of moduli.
\end{thm}

As a consequence, the only cases we need to take care are those superelliptic curves with reduced group $\G=G/H$ being either trivial or cyclic and with $\X/G$ having not an odd signature. 

\subsection{Superelliptic curves of genus at most 10}
We proceed, in each genus $2 \leq g \leq 10$, to describe those superelliptic curves which are definable over their field of moduli. Observe that in the cases left (which might or might not  be definable over their field of moduli) the last column in \cref{equations} provides an algebraic model $y^{n}=f(x)$, where $f(x)$ is defined over the algebraic closure and not necessarily over a minimal field of definition. The branched regular covering $\pi:\X \to \P_{k}^{1}$ defined by $\pi(x,y)=x$ as deck group $H=\langle \tau(x,y)=(x,\epsilon_{n} y) \rangle \cong C_{n}$. 

\medskip   \noindent \textbf{\bf Genus $2$:}
This case  is well known since in this case for every curve $\X$ with  $|\Aut (\X) | > 2$ the field of moduli is a field of definition. 

\medskip   \noindent \textbf{\bf Genus $3$:}
There are 21 signatures  from which 12 of them are hyperelliptic and $3$ are trigonal.

\begin{lem}
Every superelliptic curve of genus $3$, other than Nr. 1 and 2 in \cref{g=3}, is definable over its field of moduli. 
\end{lem} 
 
\proof 
 If $\bAut (\X)$ is isomorphic to $A_4$ or $S_4$ then the corresponding locus consists of the curves $y^4= x^4+ 2x^2 + \frac 1 3 $ and $y^2=   x^8 + 14 x^4 + 1$ which are both defined over their field of moduli. 
If $\bAut (\X)$ is isomorphic to a dihedral group and $\X$ is not hyperelliptic, then $\Aut (\X)$ is isomorphic to $V_4 \times C_4$, $G_5$, $D_6 \times C_3$, and $G_8$.  These cases  $G_5$, $D_6 \times C_3$, and $G_8$ correspond to $y^4=  x^4-1$, $y^3= x (x^3-1)$, and $y^4 = x (x^2-1)$, which are all defined over the field of moduli.  
\begin{table}[h]
\caption{Genus $3$ curves No. 1 and 2 are the only one whose field of moduli is not necessarily a field of definition}

\begin{tabular}{|l|l|l|l|l|l|l|l|}
\hline
Nr. & $\overline G$             & G      &$n$  &$m$ & sig. & $\delta$ & Equation $y^n=f(x)$ \\
\hline \hline
\textbf{\color{blue}\color{blue}1} &$\{I\}$ & $C_2$           &  2 & 1 & $2^8$  & 5 &  $  x \left( x^6+ \sum_{i=1}^5 a_i x^i   +1 \right) $ \\
\textbf{\color{blue}\color{blue}2} & $C_2$ &  $ V_4$         &  2 & 2 & $2^6$       & 3 &  $ x^8+ a_1 x^2+ a_2 x^4+ a_3 x^6+1 $ \\ 
3 &$C_2$ &         $C_4$  &  2 & 2 & $2^3, 4^2 $         & 2 &   $ x \left(x^6+a_1 x^2+a_2 x^4+1 \right)$ \\
4 &$C_2$ &    $C_6$       &  3 & 2 & $2, 3^2,6$             & 1 &   $  x^4+a_1 x^2+1 $ \\
5 & $V_{4}$ & $V_4\times C_4$ &  4 & 2 & $2^3, 4$         & 1 &  $ x^4+a_1 x^2+1$ \\
\hline \hline
\end{tabular}
\label{g=3}
\end{table}
If $\bAut(\X)$ is isomorphic to a cyclic group, then in the cases when it is isomorphic to $C_{14}, C_{12}$ there are two cases which correspond to the curves $y^2= x^7+1$ and $y^3= x^4+1$. The left cases  are given in~\cref{g=3}. The curve No. 5 is definable over its field of moduli by  \cref{teounico}. All the other cases, with the exception of Nr. 1 and 2,  the curves are of odd signature, so they are definable over their field of moduli by  \cref{teoAQ}.
\qed

\medskip   \noindent \textbf{\bf Genus $4$:} We have the following:
\begin{lem}
Every superelliptic curve of genus 4, other than Nr. 1, 3 and 5 in \cref{g=4}, is definable over its field of moduli.
\end{lem}
\proof 

There is only one case when  the reduced automorphism group $\bAut(\X)$  is not isomorphic to a cyclic or a dihedral group, namely $\bG\iso S_4$.  In this case, the curve is $y^3 = x (x^4-1)$ and is defined over the field of moduli. If $\bG$ is isomorphic to a dihedral group, then there are only $6$ signatures which give the groups $D_6\times C_3$, $D_4\times C_3$, $D_{12}\times C_3$, $D_4\times C_3$, $D_8 \times C_3$, and $D_4 \times C_5$.  The groups $D_{12}\times C_3$, $D_8 \times C_3$, and $D_4 \times C_5$ correspond to curves $y^3 = x^6-1$, $y^3=x(x^4-1)$, and $y^5= x(x^2-1)$ respectively.  
The remaining three cases are given by Nrs. 7, 8 and 9 in \cref{g=4} which are definable over their field of moduli by  \cref{teounico}. 

\begin{table}[ht]
\caption{Genus $4$ curves No. 1, 3 and 5 are the only ones whose field of moduli is not necessarily a field of definition}

\begin{tabular}{|l|l|l|l|l|l|l|l|}
\hline
Nr. & $\overline G$                & G      &$n$  &$m$ & sig. & $\delta$ & Equation $y^n=f(x)$ \\
\hline \hline
\textbf{\color{blue}\color{blue}1} & & $C_2$               &  2     & 1 & $2^{10}$       &  7     &  $  x \left( x^8+ \sum_{i=1}^7 a_i x^i   +1 \right) $ \\
2 & &  $ V_4$             &  2     & 2 & $2^7$       &  4     &  $ x^{10}+  \sum_{i=1}^4 a_i x^{2i}+1 $ \\ 
\textbf{\color{blue}\color{blue}3} & $C_m$ & $C_4$               &  2     & 2 & $2^4, 4^2$  &  3     &  $ x (x^8+a_3 x^6+a_2 x^4+a_1 x^2+1) $  \\
4 & & $C_6$               &  2     & 3 & $2^3, 3, 6$ &  2     &  $ x^9+a_1 x^3+a_2 x^6+1 $             \\
\textbf{\color{blue}\color{blue}5} &  & $C_3$              &  3     & 1 &$3^6$        &  3     &  $ x (x^4+a_1 x+a_2 x^2+a_3 x^3+1) $  \\
6 &  & $C_2 \times C_3$   &  3     & 2 & $2^2, 3^3 $  & 2     &  $ x^6+a_2 x^4+a_1 x^2+1 $             \\
\hline
7 &  & $D_6 \times C_3$   &  3     & 3 & $2^2, 3^2 $  & 1     &  $ x^6+a_1 x^3+1 $                   \\
8 & $D_{2m}$ & $V_4 \times C_3 $  &  3     & 2 & $2^2,  3, 6$ & 1     &  $ (x^2-1)(x^4+a_1 x^2+1) $         \\
9 & &  $V_4 \times C_3$   &  3     & 2 & $2^2, 3, 6$ & 1     &  $ x (x^4+a_1 x^2+1)   $              \\
\hline \hline
\end{tabular}
\label{g=4}
\end{table}

If $\bAut(\X)$ is isomorphic to a cyclic group, then there are two signatures for each of the groups $C_{18}$ and $C_{15}$.  In each case, both signatures give the same curve, namely $y^2 = x^9+1$ and $y^3=x^5+1$ respectively.  The left cases are given by cases 1 to 6 in  \cref{g=4}. As all cases, with the exception of cases 1, 3 and 5, the curves are of odd signature; so definable over their field of moduli by  \cref{teoAQ}.
\qed

\section{Theta functions}\label{sect-12}
\def\T{\theta}
\def\C{\mathbb C}

\renewcommand{\ch}[2]
{\begin{bmatrix}
 #1 \\
 #2\\
\end{bmatrix}}

\renewcommand{\chr}[4]
{\begin{bmatrix}
 #1 & #2\\
 #3 & #4
\end{bmatrix}}

\renewcommand{\chs}[6]
{\begin{bmatrix}
 #1 & #2 & #3\\
 #4 & #5 & #6
\end{bmatrix}}

In this section we describe the theory of theta functions for hyperelliptic curves and steer the reader toward the theta functions for superelliptic curves in the light of recent developments in the area \cite{kop-1}, \cite{kop-2}, \cite{book}. 
 
%
An \textbf{algebraic function} $y(x)$ is a function which satisfies some equation 
\[ f(x, y(x)) =0, \]
where $f(x, y) \in \C[x, y]$ is an irreducible polynomial. 
Recall from calculus that $\int F(x) \, dx$,  for $F(x) \in \C (x)$,  can be integrated using partial fractions and expressing this as a sum of rational functions in $x$ or logarithms of $x$.    Also, the integral  
\[ \int F(x, y) \, dx, \] 
where $F \in \C (x, y)$ and $x, y \in \C(t)$, can be easily solved   by replacing for $x=x(t)$ and $y=y(t)$ this reduces to the previous case.   Similarly, we can deal with the case   
\[ \int F \left(x,  \sqrt{ax^2+bx+c} \right) \, dx. \]
Indeed, let $y=\sqrt{ax^2+bx+c}$.  Then, $y^2=ax^2+bx+c$ is the equation of a conic.  As such it can be parametrized as $x= x(t)$,  $y=y(t)$ 
and again reduces to the previous case.  
However, the integral   
\[ \int F \left(x,  \sqrt{ax^3+bx^2+cx+d} \right) \, dx \]  
can not be solved this way because  
\[ y^2= ax^3+bx^2+cx+d\] 
is not a genus 0 curve, and therefore can not be parametrized.    Such integrals are called \textbf{elliptic integrals}.  To solve them one needs to understand the concept of  \textbf{elliptic functions} which will be developed later.  It can be easily shown that these integrals can be transformed to the form 
\[ \int \frac {p(x)} {\sqrt{q(x)} } dx \]
where $p(x), q(x)$ are polynomials such that $\deg q = 3, 4$ and $q(x)$ is separable.  The term \textbf{elliptic} comes from the fact that such integrals appear in the computation of the length of an ellipse. 

%
A natural generalization of the elliptic integrals are the \textbf{hyperelliptic integrals} which are of the form 
$ \int \frac {p(x)} {\sqrt{q(x)} } dx $
where $p(x), q(x)$ are polynomials such that $\deg q \geq 5$ and $q(x)$ is separable. 
Naturally, the square root above can be assumed to be a n-th root.  We will call such integrals \textbf{superelliptic integrals}.  Hence, a superelliptic integral is of the form 
\[ \int \frac {p(x)} {\sqrt[n]{q(x)} } dx \]
where $n \geq 3$, $p(x), q(x)$ are polynomials such that $\deg q \geq 5$ and $q(x)$ is separable. 
What about the general case when 
$ \int R(x, y) \, dx$,   where $R \in \C (x, y)$ and $y$ is an algebraic function of $x$ given by some equation $F(x, y) =0$, 
for $F(x, y)\in \C [x, y]$?   An integral of this type is called an \textbf{Abelian integral}.  


There are several version of what is called the Abel's theorem in the literature. For original versions of what Abel actually stated and proved one can check the classic books \cite{baker} and \cite{Clebsch}.  For modern interpretations of Abel's theorem and its historical perspectives there are the following wonderful references \cite{gri},  \cite{gri-2} and \cite{kleiman}. In this short notes we will try to stay as close as possible to the original version of Abel.  
Let $y$ be an algebraic function of $x$ defined by an equation of the form 
\[ f(x, y) = y^n + A_1\, y^{n-1} + \cdots A_n=0,\]
with $A_0, \dots , A_n \in \C (x)$. Let $R (x, y) \in \C(x, y)$. 
\begin{thm}[Abel]
The sum \[ \int_{(a_1, b_1)}^{(x_1, y_1)} R(x, y) + \cdots + \int_{(a_m, b_m)}^{(x_m, y_m)} R(x, y) \]
for arbitrary $a_i, b_i$, is expressible as a sum of rational functions of the variables $(x_1, y_1)$, $\dots$, $(x_m, y_m)$ and logarithms of such rational functions with the addition of 
\[ - \int^{(z_1, s_1)} R(x, y)  - \cdots - \int^{(z_k, s_k)} R(x, y) \,,\]
where $z_i, s_i$ are determined by $x_i, y_i$ as the roots of an algebraic equation whose coefficients are rational coefficients of $x_1, y_1, \dots , x_m, y_m$. Moreover, $s_1, \dots , s_k$ are the corresponding values of $y$, for which any $s_i$ is determined as a rational function of $z_i$ and   $x_1, y_1, \dots , x_m, y_m$.
The number $k$ does not depend on $m$, $R(x, y)$, or the values $(x_i, y_i)$, but only on the equation 
\[ f(x, y) =0.\]
\end{thm}
For more details of this version of Abel's theorem and its proof see \cite{baker}*{pg. 207-235}.
A modern version of the Abel's theorem, which is found in most textbooks says that the Abel-Jacobi's map is injective; see  ~\cref{thm-abel} for details. A nice discussion from a modern view point was provided in \cite{kleiman}. 
The new idea of Jacobi was to consider integrals $\int_{c}^{w} R(x, y)$ as variables and to try to determine $w$ in terms of such variables. This idea led to the fundamental concept of theta functions, which will be formally defined in the next section.

First, consider the Abelian integrals  \[  z_i : =\int_{c_i}^{w_i} R(x, y) \]  for $i=1, \dots g$.  Consider $z_i$ as variables and express $w_i$ as functions of $z_i$,
\[ w_i = f(z_i).\]
This is known as the \textbf{Jacobi inversion problem}.  
\begin{exa}[Elliptic integrals]   Let be given the integral (i.e. $g=1$)
\[ \int_0^{w_1} \frac {dt}   {\sqrt{(1-t^2)(1-k^2t^2) } } = z_1\]
Then \[w_1= \sn (z_1)= \sn (u; k) = \frac  {\th_3(0) \th_1 (v)} {\th_2 (0) \th_0 (v)}, \]   where $u=v\, \pi \, \th_3^2 (0)$ and $\th_0, \th_1, \th_2, \th_3$ are the Jacobi theta functions; see \cite{baker} for details. 
\end{exa}
%
It was exactly the above case that motivated Jacobi to introduce the theta functions.  In terms of these functions, he expressed his functions $\sn u$, $\cn u$, and $\dn u$ as fractions with the same denominator whose zeroes form the common poles of $\sn u$, $\cn u$, and $\dn u$. 
For  $g=2$, G\"opel found similar functions, building on work of Hermite.  We will say more about this case in the coming sections. G\"opel and later Rosenhain notice that integrals of the first kind, which exist for $g=2$ become elliptic integrals of the first and third kind, when two branch points of the curve of $g=2$ coincide.  This case corresponds to the degenerate cases of the $\L_n$ spaces as described in \cite{shaska-thesis} and later in \cite{deg3}. Both G\"opel and Rosenhain, when developing theta functions for genus $g=2$, were motivated by the Jacobi inversion problem.  Weierstrass considered functions which are quotients of theta functions for the  hyperelliptic curves, even though it appears that he never used the term "theta functions".  
In their generality, theta functions were developed by Riemann for $g \geq 2$.  It is Riemann's approach that is found in most modern books and that we will briefly describe in the next section. Most known references for what comes next can be found in \cites{Igusa, Mu1, Mu2, Mu3}.
 
 
\subsection{Riemann's theta functions} 
%

\subsubsection{Introduction to theta functions of curves}
Let $\X$ be an irreducible, smooth, projective curve of genus $g \geq 2$ defined over the complex field $\C.$ We denote the moduli space of genus $g$ by $\M_g$ and  the hyperelliptic locus in $\M_g$ by $\H_g.$ It is well known that  $\dim \M_g = 3g-3$ and $\H_g$ is a $(2g-1)$ dimensional subvariety of $\M_g.$  
Choose a symplectic homology basis for $\X$, say
\[ \{ A_1, \dots, A_g, B_1, \dots , B_g\}\]
such that the intersection products $A_i \cdot A_j = B_i \cdot B_j =0$ and $A_i \cdot B_j= \d_{i j}$.
We choose a basis $\{ w_i\}$ for the space of holomorphic 1-forms such that $\int_{A_i} w_j = \d_{i j},$ where $\d_{i j}$ is the Kronecker delta. The matrix $\O= \left[ \int_{B_i} w_j \right] $ is  the \textbf{period matrix} of $\X$.
The columns of the matrix $\left[ I \ | \O \right]$ form a lattice $L$ in  $\C^g$ and the Jacobian  of $\X$ is $\J(\X)= \C^g/ L$. 
Fix a point $p_0 \in \X$.  Then, the Abel-Jacobi map is defined as follows
\[
\begin{split}
\mu_p : \X & \to \J (\X) \\
p & \to \left( \int_{p_0}^p w_1, \dots ,   \int_{p_0}^p w_g \right) \mod L \,.\\
\end{split}
\]
The Abel-Jacobi map can be extended to divisors of $\X$ the natural way. For example, for a divisor $D= \sum_i n_i P_i$ we define %
\[ \mu (D) = \sum_i n_i \mu (P_i). \]
The following two theorems are part of the folklore on the subject and their proofs can be found in all classical textbooks. 
\begin{thm}[Abel]\label{thm-abel}
The Abel-Jacobi map is injective.
\end{thm}

\begin{thm}[Jacobi]\label{thm-jacobi}
The Abel-Jacobi map is surjective
\end{thm}

We continue with our goal of defining  theta functions and theta characteristics.   Let 
\[ \HS_g =\{\t : \t \,\, \textbf{is symmetric}\,\, g \times g \, \textbf{matrix with positive definite imaginary part} \}\]
be the \textbf{Siegel upper-half space}. Then $\O \in \HS_g$. The group of all $2g \times 2g$ matrices $M \in GL_{2g}(\Z)$ satisfying
\[M^t J M = J  \,\,\,\,\,\,\,\, \textbf{with} \, \,\,\,\,\,\, J = \begin{pmatrix}  0 & I_g \\ -I_g & 0 \end{pmatrix} \]
is called the \textbf{symplectic group} and denoted  by $Sp_{2g}(\Z)$.
Let $M = \begin{pmatrix} R & S \\ T & U \end{pmatrix} \in Sp_{2g}(\Z) $ and $\t \in \HS_g$ where $R,$ $S,$ $T$ and $U$ are $g \times g$ matrices. 
$Sp_{2g}(\Z) $ acts transitively on $\HS_g$ as
\[ M(\t) = (R \t + S)(T \t + U)^{-1}. \]
\noindent Here, the multiplication is matrix multiplication. There is an injection
\[ \M_g \embd \HS_g/ Sp_{2g}(\Z) =: \A_g, \] 
where each curve $C$ (up to isomorphism) is mapped to its Jacobian in $\A_g.$
If $\ell$ is a positive integer, the principal congruence group of degree $g$ and of level $\ell$ is defined as a subgroup of $Sp_{2g}(\Z)$ by the condition $M \equiv I_{2g} \mod \ell.$ We shall denote this group by $Sp_{2g}(\Z)(\ell)$.
For any $z \in \C^g$ and $\t \in \HS_g$ the \textbf{Riemann's theta function} is defined as
\[ \T (z , \t) = \sum_{u\in \Z^g} e^{\pi i ( u^t \t u + 2 u^t z )  }\]
where $u$ and $z$ are $g$-dimensional column vectors and the products involved in the formula are matrix products. The fact that the imaginary part of $\t$ is positive makes the series absolutely convergent over every compact subset of $\C^g \times \HS_g$.

The theta function is holomorphic on $\C^g\times \HS_g$ and has quasi periodic properties,
\[ \T(z+u,\tau)=\T(z,\tau)\quad \textbf{and}\quad \T(z+u\tau,\tau)=e^{-\pi i( u^t \tau u+2z^t u )}\cdot  \T(z,\tau), \]
where $u\in \Z^g$; see \cite{Mu1} for details.  The locus 
\[ \Theta: = \{ z \in \C^g/L : \T(z, \O)=0 \}\]
is called the \textbf{theta divisor} of $\X$.
Any point $e \in \J (\X)$ can be uniquely written  as $e = (b,a) \begin{pmatrix} 1_g \\ \O \end{pmatrix}$ where $a,b \in \R^g$ are the characteristics of $e.$
We shall use  the notation $[e]$ for the characteristic of $e$ where $[e] = \ch{a}{b}.$ For any $a, b \in \Q^g$, the theta function with rational characteristic is defined as a translate of Riemann's theta function multiplied by an exponential factor
\begin{equation} \label{ThetaFunctionWithCharac} \T  \ch{a}{b} (z , \t) = e^{\pi i( a^t \t a + 2 a^t(z+b))} \T(z+\t a+b ,\t).\end{equation}
\noindent By writing out  \cref{ThetaFunctionWithCharac}, we obtain
\[ \T  \ch{a}{b} (z , \t) = \sum_{u\in \Z^g} e^{\pi i ( (u+a)^t \t (u+a) + 2 (u+a)^t (z+b) )  }. \]
The Riemann's theta function is $\T \ch{0}{0}.$ Theta functions with rational characteristics have the following properties:
\begin{equation}\label{periodicproperty}
\begin{split}
& \T \ch{a+n} {b+m} (z,\t) = e^{2\pi i a^t m}\T \ch {a} {b} (z,\t),\\
&\T \ch{a} {b} (z+m,\t) = e^{2\pi i a^t m}\T \ch {a} {b} (z,\t),\\
&\T \ch{a} {b} (z+\t m,\t) = e^{\pi i (-2b^t m -m^t \t m - 2m^t z)}\T \ch {a} {b} (z,\t)\\
\end{split}
\end{equation}
with $n,m \in \Z^n.$ All of these properties are immediately verified by writing them out.
A scalar obtained by evaluating a theta function with characteristic at $z=0$ is called a \textbf{theta constant} or \textbf{theta-nulls}. When the entries of column vectors $a$ and $b$ take values in $\{ 0,\frac{1}{2}\}$, then the characteristics $ \ch {a}{b} $ are called the \textbf{half-integer characteristics}. The corresponding theta functions with rational characteristics are called \textbf{theta characteristics}.

Points of order $n$ on $\J(\X)$ are called the $\frac 1 n$-\textbf{periods}. Any point $p$ of $\J(\X)$ can be written as $p = \t \,a + b. $ If $\ch{a}{b}$ is a $\frac 1 n$-period, then $a,b \in (\frac{1}{n}\Z /\Z)^{g}.$ The $\frac 1 n$-period $p$ can be associated with an element of $H_1(\X,\Z / n\Z)$ as follows: 
Let $a = (a_1,\cdots,a_g)^t,$ and $b = (b_1,\cdots,b_g)^t.$ We have
\[
\begin{split}
p   = \t a + b  & = \left(      \sum a_i \int_{B_i} \om_1, \cdots , \sum a_i \int_{B_i} \om_g \right)^t     +       
           \left(b_1 \int_{A_1} \om_1, \cdots  ,   b_g       \int_{A_g} \om_g \right) \\
           & = \left(\sum (a_i \int_{B_i} \om_1 + b_i\int_{A_i} \om_1 \right) , \cdots , \sum \left(a_i \int_{B_i} \om_g + b_i\int_{A_i} \om_g) \right)^t\\
           & = \left( \int_C \om_1, \cdots, \int_C \om_g \right)^t
\end{split}
\]
with $C = \sum a_i B_i + b_i A_i. $ We identify the point $p$ with the cycle $\bar{C} \in H_1(\X,\Z / n\Z)$ where $\bar{C} =\sum \bar{a_i} B_i + \bar{b_i} A_i,$  $\bar{a_i} = n a_i$ and $\bar{b_i} = n b_i$ for all $i$;  see  \cite{Accola} for more details. 

\subsection{Half-Integer Characteristics and the G\"opel Group} 
In this section we study the groups of half-integer characteristics. Any half-integer characteristic $\m \in\frac{1}{2}\Z^{2g}/\Z^{2g}$ is given by
\[
\m = \frac{1}{2}m = \frac{1}{2}
\begin{pmatrix} m_1 & m_2 &  \cdots &  m_g \\ m_1^{\prime} & m_2^{\prime} & \cdots & m_g^{\prime}  \end{pmatrix},
\]
where $m_i, m_i^{\prime} \in \Z.$ For $\m = \ch{m ^\prime}{m^{\prime \prime}} \in \frac{1}{2}\Z^{2g}/\Z^{2g},$ we define $e_*(\m) = (-1)^{4 (m^\prime)^t m^{\prime \prime}}.$ We say that $\m$ is an \textbf{even} (resp. \textbf{odd}) characteristic if $e_*(\m) = 1$ (resp. $e_*(\m) = -1$). For any curve of genus $g$, there are $2^{g-1}(2^g+1)$ (resp., $2^{g-1}(2^g-1)$ ) even theta functions (resp., odd theta functions). Let $\an$ be another half-integer characteristic. We define
\[
 \m \, \an = \frac{1}{2} \begin{pmatrix} t_1 & t_2 &  \cdots &  t_g \\ t_1^{\prime} & t_2^{\prime} & \cdots &
t_g^{\prime}
\end{pmatrix}
\]
where $t_i \equiv (m_i\, + a_i)  \mod 2$ and $t_i^{\prime} \equiv (m_i^{\prime}\, + a_i^{\prime} ) \mod 2.$
We only consider  characteristics $\frac{1}{2}q$ in which each of the elements
$q_i,q_i^{\prime}$ is either 0 or 1. We use the following abbreviations:
\[
\begin{split}
&|\m| = \sum_{i=1}^g m_i m_i^{\prime},  \quad \quad \quad \quad \quad \quad \quad \quad \quad
|\m, \an| = \sum_{i=1}^g (m_i^{\prime} a_i - m_i a_i^{\prime}), \\
& |\m, \an, \bn| = |\an, \bn| + |\bn, \m| + |\m, \an|, \quad \quad \binom{\m}{\an} = e^{\pi i \sum_{j=1}^g m_j
a_j^{\prime}}.
\end{split}
\]
\indent The set of all half-integer characteristics forms a group $\G$ which has $2^{2g}$ elements. We say that two half integer characteristics $\m$ and $\an$ are \textbf{syzygetic} (resp., \textbf{azygetic}) if $|\m, \an| \equiv 0 \mod 2$ (resp., $|\m, \an| \equiv 1 \mod 2$) and three half-integer characteristics $\m, \an$, and $\bn$ are syzygetic if
$|\m, \an, \bn| \equiv 0 \mod 2$.
A \textbf{G\"opel group} $G$ is a group of $2^r$ half-integer characteristics where $r \leq g$ such that every two characteristics are syzygetic. The elements of the group $G$ are formed by the sums of $r$ fundamental characteristics; see \cite{baker}*{pg. 489} for details. Obviously, a G\"opel group of order $2^r$ is isomorphic to $C^r_2$. The proof of the following lemma can be found on   \cite{baker}*{pg.  490}.
\begin{lem}
The number of different G\"opel groups which have $2^r$ characteristics is
\[ \frac{(2^{2g}-1)(2^{2g-2}-1)\cdots(2^{2g-2r+2}-1)}{(2^r-1)(2^{r-1}-1)\cdots(2-1)}. \]
\end{lem}
If $G$ is a G\"opel group with $2^r$ elements, it has $2^{2g-r}$ cosets. The cosets are called \textbf{G\"opel systems}
and are denoted by $\an G$, $\an \in \G$. Any three characteristics of a G\"opel system are syzygetic. We can find a
set of characteristics called a basis of the G\"opel system which derives all its $2^r$ characteristics by taking only
combinations of any odd number of characteristics of the basis.
\begin{lem}
Let $g \geq 1$ be a fixed integer, $r$ be as defined above and $\sigma = g-r.$ Then there are
$2^{\sigma-1}(2^\sigma+1)$ G\"opel systems which only consist of even characteristics and there are
$2^{\sigma-1}(2^\sigma-1)$ G\"opel systems which consist of odd characteristics. The other $2^{2\sigma}(2^r-1)$ G\"opel
systems consist of as many odd characteristics as even characteristics.
\end{lem}
\proof The proof can be found on \cite{baker}*{pg. 492}. \qed
\begin{cor}\label{numb_systems}
When $r=g,$ we have only one (resp., 0) G\"opel system which consists of even (resp., odd) characteristics.
\end{cor}
Consider $s=2^{2\sigma}$ G\"opel systems which have  distinct characters and denote them by
\[\an_1 G,\an_2 G,\cdots,\an_s G.\] 
We have the following lemma.
\begin{lem}
It is possible to choose $2\sigma+1$ characteristics from $\an_1, \an_2,\cdots, \an_s,$  say $\bar{\an}_1,$
$\bar{\an}_2,$ $\cdots,$ $\bar{\an}_{2\sigma+1}$, such that every three of them are azygetic and all have the same
character. The above $2\sigma+1$ fundamental characteristics are even (resp., odd) if $\sigma \equiv 1,0 \mod 4$
(resp.,$\equiv 2,3 \mod 4$).
\end{lem}
\noindent The proof of the following lemma can be found on \cite{baker}*{pg. 511}.
\begin{lem}
For any half-integer characteristics $\an$ and $\hn,$ we have the following:
\begin{equation}\label {Bakereq1}
\T^2[\an](z_1,\t) \T^2[\an \hn](z_2,\t) = \frac{1}{2^{g}} \sum_\en  e^{\pi i |\an \en|} \binom{ \hn}{ \an \en}
\T^2[\en](z_1,\t)\T^2[\en \hn](z_2,\t),
\end{equation}
where the sum runs over all half-integer
characteristics.
\end{lem}
We can use this relation to get identities among half-integer thetanulls.  We know that we have $2^{g-1}(2^g+1)$ even characteristics. As the genus increases, we have multiple choices for $\en.$ In the following, we explain how we reduce the number of possibilities for $\en$ and how to get
identities among thetanulls.
First we replace $\en$ by $\en \hn$ and $z_1=z_2= 0$ in  \cref{Bakereq1}. \cref{Bakereq1} can then be written
as follows:
\begin{equation}\label {Bakereq2}
\T^2[\an] \T^2[\an \hn] = 2^{-g} \sum_\en  e^{\pi i |\an \en \hn|} \binom{ \hn }{ \an \en \hn} \T^2[\en] \T^2[\en \hn].
\end{equation}
We have $e^{\pi i |\an \en \hn|}\binom{ \hn }{ \an \en \hn} = e^{\pi i |\an \en|}\binom{ \hn }{ \an \en} e^{\pi i |\an
\en, \hn|}.$ Next we put $z_1=z_2= 0$ in  \cref{Bakereq1} and add it to  \cref{Bakereq2} and obtain the following
identity:
\begin{equation}\label {Bakereq3}
2\T^2[\an] \T^2[\an \hn] = 2^{-g} \sum_\en  e^{\pi i |\an \en|} (1 + e^{\pi i|\an \en, \hn|}) \T^2[\en] \T^2[\en \hn].
\end{equation}
If $|\an \en, \hn| \equiv  1 \mod 2$, the corresponding terms in the summation vanish. Otherwise $1 + e^{\pi i|\an \en,
\hn|} = 2.$ In this case, if either $\en$ is odd or $\en \hn$ is odd, the corresponding terms in the summation vanish
again. Therefore, we need $|\an \en, \hn| \equiv 0 \mod 2$ and $|\en| \equiv |\en \hn| \equiv 0 \mod 2,$ in order to
get nonzero terms in the summation. If $\en^*$ satisfies $|\en^*| \equiv |\en^* \hn^*| \equiv 0 \mod 2$ for some
$\hn^*,$ then $\en^*\hn^*$ is also a candidate for the left hand side of the summation. Only one of such two values
$\en^*$ and $\en^* \hn^*$ is taken. As  a result, we have the following identity among thetanulls
\begin{equation} \label{theta-eq1}
\T^2[\an] \T^2[\an \hn] = \frac{1}{2^{g-1}} \sum_\en  e^{\pi i |\an \en|} \binom{ \hn }{ \an \en} \T^2[\en]\T^2[\en
\hn],
\end{equation}
where $\an, \hn$ are any characteristics and $\en$ is a characteristics such that $|\an \en, \hn| \equiv 0 \mod 2,$
$|\en| \equiv |\en \hn| \equiv 0 \mod 2$ and $\en \neq \en \hn.$

By starting from the  \cref{Bakereq1} with $z_1 = z_2$ and following a similar argument to the one above, we can
derive the identity,
\begin{equation}\label{eq2}
\T^4[\an] + e^{\pi i |\an, \hn|} \T^4[\an \hn] = \frac{1}{2^{g-1}} \sum_\en  e^{\pi i |\an \en|} \{ \T^4[\en] + e^{ \pi
i |\an, \hn|} \T^4[\en \hn]\}
\end{equation}
where $\an, \hn$ are any characteristics and $\en$ is a characteristic such that $|\hn| + |\en, \hn| \equiv 0 \mod 2,$
$|\en| \equiv |\en \hn| \equiv 0 \mod 2$ and $\en \neq \en \hn.$
\begin{rem}
$|\an \en ,\hn| \equiv 0 \mod 2$ and $|\en \hn| \equiv |\en| \equiv 0 \mod 2$ implies $|\an, \hn| + |\hn| \equiv 0 \mod
2.$
\end{rem}
We use  \cref{theta-eq1} and  \cref{eq2} to get identities among theta-nulls.
%

\subsection{Hyperelliptic curves and their theta functions}
A hyperelliptic curve $\X,$ defined over $\C,$ is a cover of order two of the projective line $\P^1.$  
 Let $\X \longrightarrow \P^1$ be the degree 2 hyperelliptic projection.
We can assume that $\infty$ is a branch point.
Let  $B := \{\a_1,\a_2, \cdots ,\a_{2g+1} \}$
be the set of other branch points and let $S = \{1,2, \cdots, 2g+1\}$ be the index set of $B$ and $\e : S \longrightarrow
\frac{1}{2}\Z^{2g}/\Z^{2g}$  be a map defined as follows:
%
\[
\e(2i-1)  = \begin{bmatrix}
              0 & \cdots & 0 & \frac{1}{2} & 0 & \cdots & 0\\
              \frac{1}{2} & \cdots & \frac{1}{2} & 0 & 0 & \cdots & 0\\
            \end{bmatrix}, 
\quad             
 \e(2i)  =\begin{bmatrix}
              0 & \cdots & 0 & \frac{1}{2} & 0 & \cdots & 0\\
              \frac{1}{2} & \cdots & \frac{1}{2} & \frac{1}{2} & 0 & \cdots & 0\\
            \end{bmatrix}
\]
%
where the nonzero element of the first row appears in $i^{th}$ column.
We define  $\e(\infty) $ to be $
\begin{bmatrix}
              0 & \cdots & 0 & 0\\
              0 & \cdots & 0 & 0\\
            \end{bmatrix}$.
For any $T \subset B $, we define the half-integer characteristic as
\[ \e_T = \sum_{a_k \in T } \e(k) .\]
Let $T^c$ denote the complement of $T$ in $B.$ Note that $\e_B \in \Z^{2g}.$ If we view $\e_T$ as an element of
$\frac{1}{2}\Z^{2g}/\Z^{2g}$ then $\e_T= \e_{T^c}.$ Let $\triangle$ denote the symmetric difference of sets, that is $T
\triangle R = (T \cup R) - (T \cap R).$ It can be shown that the set of subsets of $B$ is a group under $\triangle.$ We
have the following group isomorphism:
\[ \{T \subset B\,  |\, \#T \equiv g+1 \mod 2\} / T \sim T^c \cong \frac{1}{2}\Z^{2g}/\Z^{2g}.\]
For $\gamma = \ch{\gamma ^\prime}{\gamma^{\prime \prime}} \in \frac{1}{2}\Z^{2g}/\Z^{2g}$, we have
\begin{equation}\label{parityIdentity}
 \T [\gamma] (-z , \t) = e_* (\gamma) \T [\gamma] (z , \t).\end{equation}
It is known that for hyperelliptic curves, $2^{g-1}(2^g+1) -  \binom{2g+1 }{ g}$ of the even thetanulls are zero.
The following theorem provides a condition for the characteristics in which theta characteristics become zero. The
proof of the theorem can be found  in \cite{Mu2}.
\begin{thm}\label{vanishingProperty}
Let $\X$ be a hyperelliptic curve, with a set $B$ of branch points. Let $S$ be the index set as above and $U $ be the
set of all odd values of $S$. Then for all $T \subset S$ with even cardinality, we have $ \T[\e_T] = 0$  if and only if
$\#(T \triangle U) \neq g+1$, where $\T[\e_T]$ is the theta constant corresponding to the characteristics $\e_T$.
\end{thm}
When the characteristic $\gamma$ is odd, $e_* (\gamma)=1.$ Then from  \cref{parityIdentity} all odd thetanulls
are zero. There is a formula which satisfies half-integer theta characteristics for hyperelliptic curves called
\textbf{Frobenius' theta formula}.
\begin{lem}[Frobenius]\label{Frob}
For all $z_i \in \C^g$, $1\leq i \leq 4$ such that $z_1 + z_2 + z_3 + z_4 = 0$ and for all $b_i \in \Q^{2g}$, $1\leq i
\leq 4$ such that $b_1 + b_2 + b_3 + b_4 = 0$, we have
\[ \sum_{j \in S \cup \{\infty\}} \epsilon_U(j) \prod_{i =1}^4 \T[b_i+\e(j)](z_i) = 0, \]
where for any $A \subset B$,
\[
\epsilon_A(k) =
          \begin{cases}
           1 & \textit {if $k \in A$}, \\
           -1 & \textit {otherwise}.
          \end{cases}
\]
\end{lem}
\proof See \cite{Mu1}*{pg.107}. \qed

A relationship between thetanulls and the branch points of the hyperelliptic curve is given by Thomae's formula:
\begin{lem}[Thomae]\label{Thomae}
%
For all sets of branch points $B=\{\a_1,\a_2, \cdots ,\a_{2g+1} \},$ there is a constant $A$ such that for all $T\subset B,$ $\# T$ is even, \\
\[\T[\eta_T](0;\t)^4 =(-1)^{\#T \cap U} A \prod_{\substack{i<j \\ i,j \in T \triangle U}} (\a_{i} - \a_{j}) \prod_{\substack{i<j \\ i,j \notin T \triangle U}} (\a_{i} -
\a_{j})  \] \\
\noindent where $\eta_T$ is a non singular even half-integer characteristic corresponding to the subset $T$ of branch
points.
\end{lem}
See \cite{Mu1}*{pg. 128} for the description of $A$ and \cite{Mu1}*{pg. 120} for the proof. Using Thomae's formula and Frobenius' theta identities we express the branch points of the hyperelliptic curves in terms of even thetanulls.   In \cite{sh-1} and \cite{sh-2} it is shown how such relations are computed for genus $g=2, 3$. 
 
\subsection{Superelliptic  curves and their  theta functions}

Generalizing the theory of theta functions of hyperelliptic curves to all cyclic covers of the projective line has been the focus of research of the last few decades.  The main efforts have been on generalizing the Thomae's formula to such curves.  In the literature of Riemann surfaces such curves are called for historical reasons the $C_n$ curves. 
As a more recent development and new developments on this topic see \cite{book}.


\section{Jacobian varieties}\label{sect-13}
Let $\X$ be a  smooth, irreducible, algebraic  curve of  genus $g\geq 2$, defined over a field $K$.   Let $S_d$ denote the symmetric group of permutations.  Then $S_d$ acts on $\X^d$ as follows:
\begin{equation}
\begin{split}
S_d \times \X^d & \rightarrow \X^d \\
\left( \sigma , \left(  P_1, \dots , P_d \right) \right) & \rightarrow \left(\dots ,  P_i^{\sigma}, \dots \right) \\
\end{split}
\end{equation}
We denote the orbit space of this action by $\Sym^d (\X)$. 
Denote by   $\Div^d (\X)$  the set of degree $v$ divisors in $\Div (\X)$ and by $\Div^{+, d} (\X)$ the set of positive ones in $\Div^d (\X)$.
\begin{lem}
$\Div^{+ \, d} (\X) \cong \Sym^d (\X)$. 
\end{lem}

Let 
$ j : \X^d \hookrightarrow \P^{(n+1)d -1 } $
 be the Segre embedding. Let $R:= \C [\X^d] $ be the homogenous coordinate ring   of $\X^d$.  Then $S_d$ acts on $R$ by permuting the coordinates.  This action preserves the grading.  
Then $j$ is equivariant under the above action.     Hence,    the ring of invariants $R^{S_d}$  is finitely generated by homogenous polynomials $f_0, \dots , f_N$ of degree $M$. Thus, we have 
\[
\C [ f_0, \dots , f_N] \subset \{ f \in R^{S_d } \, \text{ such that }  \, M | \deg f \} \subset R^{S_d } \,.
\]
Hence, every element in $\C [ f_0, \dots , f_N]$ we can express it as a vector in $\P^N$ via the basis $\{f_0, \dots , f_N\}$.
Then we have an embedding 
\[ \Sym^d (\X) \hookrightarrow \P^N \]
with the corresponding  following diagram:
\[
\xymatrix{
		\X^d   \ar[d]     \ar@{^{(}->}[r]^j            &       \P^{(n+1)d -1}     \ar[d]           \\
        \Sym^d (\X)   \ar@{^{(}->}[r]      &        \P^N      \\
        }
\]
Thus, any divisor $D \in \Div^{+ \, d} (\X)$ we identify with its correspondent point in $\Sym^d (\X)$ and then express it in coordinates in $\P^N$.  
The variety $\Sym^d (\X)$ is smooth because $\Sym^d (\X) \setminus\{\Delta = 0 \}$ is biholomorphically to an open set in $\C^d$.

The known result which we will use in our approach is  the following:
\begin{thm} Let $\X$ be a genus $g\geq 2$ curve. The map 
\[ 
\begin{split}
\phi : & \Sym^g (\X) \longrightarrow \Jac \X \\
& \sum P_i \longrightarrow \sum P_i - g \infty \\
\end{split}
\]
is surjective.  In other words, for every divisor $D$ of degree zero, there exist $P_1, \dots , P_g$ such that 
$D$ is linearly equivalent to $\sum_{i=1}^g P_i - g\infty$. 
\end{thm}

See \cite{tata-2}*{pg. 3.30}.   The simplest case (hyperelliptic case)   of the above construction was suggested by Jacobi and worked out by Mumford in \cite{tata-2}.  We explained it briefly below.


\subsection{Hyperelliptic curves}
We would like to see how the above construction applies to hyperelliptic curves. Let's start with a hyperelliptic curve $\X$ with affine equation
\[ y^2 = f(x) = \prod_{i=1}^{2g+1} (x-\a_i) \]
defined over a field $k$. Then $\X$ has a point at infinity and $(x)_\infty= 2 \cdot \infty$ and $(y)_\infty = (2g+1) \cdot \infty$.

Denote by   $\Div^d (\X)$  the set of degree $v$ divisors in $\Div (\X)$ and by $\Div^{+, d} (\X)$ the set of positive ones in $\Div^d (\X)$. Then $\Div^{+, d}_0 (\X)$ is the set
\[ 
\begin{split}
\Div^{+, d}_0 (\X)  = 
& \left\{       D \in \Div^{+, d} (\X) \, |   \;  \text{ if } \; D= \sum_{i=1}^d P_i, \, \text{ then } \, P_i \neq \infty,   \text{ for all } \, i  \right.   \\
& \left.   \; \text{ and } P_i \neq \tau P_j, \; \text{ for } i\neq j   \right\}    \\
\end{split}
\]
where $\tau$ is the hyperelliptic involution.     Let $D \in \Div^{+, d}_0 (\X)$  given by
$ D= \displaystyle\sum_{i=1}^d \, \p_i$, 
where $\p_i = (\l_i, u_i)$.   By $x( \p_i)$ we denote the value of $x$ at $\p_i$.  Thus $x(\p_i)=\l_i$ and $y(\p_i)=u_i$. 
We follow the idea of Jacobi \cite{jacobi-1846} explained in details in \cite{tata-2} and define 
%
%
\begin{equation}
\begin{split}
U (x)  & = \prod_{i-1}^d  (x-\l_i) \\
\end{split}
\end{equation}
We want to determine a unique polynomial $V(x)$ of degree $< d-1$ such that 
\[ 
V (\l_i) = u_i, \; \; 1 \leq i \leq d. 
\]
Then we have:

\begin{lem}
The unique polynomial $V(x)$ of degree $\leq d-1$ such that 
\[ 
V (\l_i) = u_i, \; \; 1 \leq i \leq d
\]
is given by 
\begin{equation}
V(x) = \displaystyle\sum_{i=1}^d  u_i \; \displaystyle\frac { \displaystyle\prod_{j\neq i} (x- \l_i)   }   {  \displaystyle\prod_{j\neq i} (\l_i - \l_j)  } 
\end{equation}
Moreover,
$ U(x) \, | \, (f(x) - V(x)^2)$.
\end{lem}




Let $W(x)$ be defined as follows
\begin{equation}
\begin{split}
W(x) & = \frac 1 {U(x)} \left( f(x) - V(x)^2  \right),  \\
\end{split}
\end{equation}
which from the above is a polynomial.  
Then we have the following:

\begin{prop}  
There is a bijection between $\Div^{+, d}_0 (\X)$ and triples $(U, V, W)$ such that 
%
%
$U$ and $W$ are monic and $\deg V \leq d-1$, $\deg U = d$, $\deg W = 2g+1 -d$. 
\end{prop}

\proof
See \cite{Mum}*{Prop. 1.2}.
\qed

Polynomials $U(x)$, $V(x)$, and $W(x)$ are called \textbf{Jacobi polynomials}. 
%
%
%
%
Take  a genus $g \geq 2$ hyperelliptic curve $\X$ with at least one rational Weierstrass point  given by the affine Weierstrass  equation
\begin{equation}\label{hyp}
W_\X : \; y^2 + h(x) \, y = x^{2g+1} + a_{2g} x^{2g} + \dots + a_1 x + a_0 
\end{equation}
over $k$. We  denote the prime divisor corresponding to  $P_\infty =(0:1:0)$  by $\mathfrak {p}_\infty$.   The affine coordinate ring of $W_\X$ is  
\[  
\O   =  k[x,y]/( y^2 + h(x),      y -( x^{2g+1} + a_{2g} x^{2g} + \dots + a_1 x + a_0)    )
\]
and so prime divisors $\p$ of degree $d$ of  $\X$ correspond to prime ideals $P\neq 0$ with $[\O/P:k]=d$.    Let $\omega $ be the hyperelliptic involution of $\X$. It operates on $\O$ and on $\Spec(\O)$ and fixes exactly the prime ideals which   "belong" to Weierstrass points, i.e. split up in such points over $\bar{k}$.
 
Following Mumford \cite{Mum} we introduce polynomial coordinates for points in $\Jac_k (\X)$. The first step is to normalize representations of divisor classes. In each divisor class $c\in \Pic^0(k)$ we find a unique \textbf{reduced} divisor
\[ D=n_1\p_1+\cdots +n_r \p_r - d\, \p_\infty \]
with 
\[ \sum_{i=1}^r n_i\deg (\p_i) = d\leq g,\]
 $\p_i \neq \omega(\p_j)$ for $i\neq j$ and $\p_i\neq \p_\infty$ (we use Riemann-Roch and the fact that $\omega$ induces $-id_{J_\X}$).

Using the relation between divisors and ideals in coordinate rings, we obtain that
\[ n_1\p_1+\cdots +n_r \p_r\]
corresponds to an ideal $I\subset \O$ of degree $d$ and the property that if the prime ideal $P_i$ is such that  both $P$ and $\omega(P)$ divide $I$ then it belongs to a Weierstrass point.  The  ideal $I$ is a free $\O$-module of rank $2$ and we have
\[I=k[x]u(x)+k[x](v(x)-y).\]
$u(x)$, $v(x)\in k[x]$, $u$ are monic of degree $d$, $\deg(v)<d$,  and 
\[  
u \, | \, \left(  v^2+h(x)v-f(x) \right).
\]
Moreover, $c$ is uniquely determined by $I$, $I$ is uniquely determined by $(u,v)$ and so we can \textbf{take $(u,v)$ as coordinates for $c$}.  Polynomials $u$ and $v$ are determined by the following:

\begin{thm}[Mumford representation] \label{Mum-rep}
Let $\X$ be a hyperelliptic curve of genus $g\geq 2$ with affine equation 
\[ y^2 + h(x)\, y  \, = \,  f(x), \]
where $h, f \in k[x]$, $\deg f = 2g +1 $, $\deg h \leq g$.

 Every non-trivial group element $c \in \Pic^0_\X (k)$ can be represented in a unique way by a pair of polynomials $u, v \in k[x]$, such that 

i) $u$ is a monic,

ii) $\deg v < \deg u \leq g$,

iii) $u \, | \, v^2+ vh -f$.
\end{thm}

How does one find the polynomials $u,v$?   We can assume without loss of generality that $k=\bar{k}$ and identify prime divisors $\p_i$ with points $P_i=(x_i,y_i)\in k\times k$. Taking the reduced divisor $D =n_1\p_1+\cdots + n_r \p_r - d \p_\infty$ with $r=d\leq g$, we have
\[  u(x)= \prod_{i=1}^r(x-x_i)^{n_i}. \]
Since $(x-x_i)$ occurs with multiplicity $n_i$ in $u(x)$ we must have for $v(x)$ that 
\[  \left( \frac d {dx}  \right)^j \left[ v(x)^2 + v(x) \, h(x) - f(x)  \right]_{x=x_i} =0, \]
and one determines $v(x)$ by solving this system of equations; see \cite{frey-shaska} for details.


Take the divisor classes represented by $[(u_1,v_1)]$  and $[(u_2,v_2)]$ and  in "general position". Then the product is represented by the ideal $I\in\O$ 
given by 
\[ \< u_1u_2, u_1(y-v_2), u_2(y-v_1), (y-v_1)(y-v_2) \>.\]
We have to determine a base, and this is done by Hermite reduction.  The resulting ideal is of the form $\<u'_3(X), v'_3(X)+w'_3(X)Y\>$ but not necessarily reduced. To reduce it one uses recursively the fact that $u \, \mid \, (v^2-hv-f)$.     Generalization of this  procedure   is called \textbf{Cantor's algorithm}; see \cite{frey-shaska} for details.  


Another approach to describing addition in the Jacobians of hyperelliptic curves is to use approximation by rational functions;  see \cite{Leitenberger}. This is analogous to the geometric method used for elliptic curves. 

For simplicity we assume that $k=\bar{k}$.  Let $D_1$ and $D_2$ be reduced divisors on $\Jac_k \X $ given by 
\begin{equation} 
\begin{split}
D_1  & = \p_1 + \p_2 + \dots + \p_{h_1} - h_1 \p_\infty,  \\
D_2  & = \mathfrak{q}_1 + \mathfrak{q}_2 + \dots + \mathfrak{q}_{h_2} - h_2\p_\infty,  \\
\end{split}
\end{equation}
where  $\p_i$ and $\mathfrak{q}_j$ can occur with multiplicities,  and $0 \leq h_i \leq g$, $i=1, 2$. 
As usual we denote by $P_i$ respectively $Q_j$ the points on $\X$ corresponding to $\p_i$ and $\mathfrak{q}_j$.

Let $g(X)= \frac {b(X)} {c(X)} $ be the unique  rational function going through the points $P_i$, $Q_j$. In other words we are determining $b(X)$ and $c(X)$ such that $h_1+h_2-2r$ points $P_i$, $Q_j$ lie on the curve
\[ Y \, c(X) - b(X) \, =  \, 0.\] 
This rational function is uniquely determined and has the form
\begin{equation}\label{pol} 
Y= \frac {b(X)} {c(X)}    = \frac     {b_0 X^p + \dots b_{p-1} X + b_p} {c_0 X^q + c_1 X^{q-1} + \dots + c_q} 
\end{equation}
where 
\[
p=\frac {h_1+h_2+g-2r-\epsilon} 2, \; \; q = \frac {h_1+h_2-g-2r-2+\epsilon} 2,
\]
$\epsilon$ is the parity of $h_1+h_2+g$.    By replacing $Y$ from  \cref{pol} in  \cref{hyp} we get a polynomial of degree $\max \{ 2p, \, 2q (2g-1) \}$, which gives $h_3\leq g$ new roots apart from the $X$-coordinates of $P_i, Q_j$.   Denote the corresponding points on $\X$ by $R_1, \dots , R_{h_3}$ and   $\bar R_1, \dots , \bar R_{h_3}$ are the   corresponding symmetric points with respect to the $y=0$ line. 
Then, we define 
\[ D_1 + D_2 = \bar R_1 + \dots \bar R_{h_3} - h_3 \O.\]
For details we refer the reader to \cite{Leitenberger}.  

\begin{rem}
For $g=1, 2$ we can take $g(X)$ to be a cubic polynomial. 
\end{rem}

\subsubsection{Curves of genus $2$}\label{gen-2-add}
Let $\X$ be a genus 2 curve defined over a field $k$ with a rational Weierstrass point.  If $k \neq 2, 3$ the curve $\X$ is  birationally isomorphic  to an affine plane curve with equation
\begin{equation} 
Y^2 = a_5 X^5 + a_4 X^4+ a_3 X^3 + a_2 X^2 + a_1 X + a_0.
\end{equation}
Let $\p_\infty$ be the prime divisor corresponding to the point at infinity. Reduced divisors in generic position  are  given by 
\[ D = \p_1 + \p_2 - 2 \p_\infty \,,\]
where $P_1(x_1, y_1) $, $P_2(x_2, y_2)$ are points in $\X(k)$ (since $k$ is algebraically closed) and $x_1\neq x_2$.  For any two divisors $D_1=\p_1 + \p_2 - 2\p_\infty$ and $D_2 = \mathfrak{q}_1 + \mathfrak{q}_2 - 2\p_\infty$ in reduced form, we determine the cubic polynomial 
\begin{equation} Y= g(X) = b_0 X^3 + b_1 X^2 + b_2 X + b_3,\end{equation}
going through the points $P_1 (x_1, y_1)$, $P_2(x_2, y_2)$, $Q_1(x_3, y_3)$, and $Q_2(x_4, y_4)$.   This cubic will intersect the curve $\X$ at exactly two other points $R_1$ and $R_2$ with coordinates 
\begin{figure}[htbp] 
   \centering
   \includegraphics[width=3in]{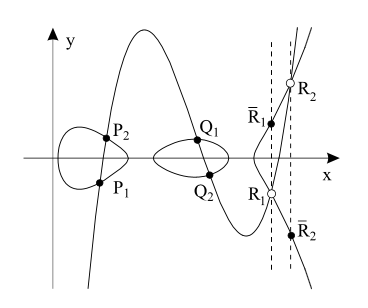} 
   \caption{A geometric interpretation of addition on a 2-dimensional  Jacobian.}
   \label{fig:genus-2}
\end{figure}
\begin{equation}
R_1  = \left( x_5,   g(x_5)    \right) \; \text{ and } \; R_2  =  \left( x_6,  g(x_6)    \right),  
\end{equation}
where $x_5$, $x_6$ are roots of the quadratic equation
\begin{equation}  
x^2 + \left( \sum_{i=1}^4 x_i   \right) x + \frac  {b_3^2-a_5} { b_0^2 \prod_{i=1}^4 x_i} = 0. 
\end{equation}
Let us denote by $\overline R_1 = (x_5, - g(x_5) )$ and $\overline R_2 = (x_6, - g(x_6) )$. 
Then,
\begin{equation} 
[D_1] \oplus [D_2] = [\overline R_1  + \overline R_2 - 2 \p_\infty].
\end{equation}


\begin{exa}
Let $\X$ be a genus 3 hyperelliptic curve with equation $y^2=f(x)$ where $\deg f = 7$.  Then   $g(x)=\frac {b(x) } {c(x)} $ must be such that 
\[ \left( b(x) \right)^2 - c(x) \cdot f(x) =0 ,\]
must have degree 9.  Hence, $\deg b = 4$ and $\deg c = 2$.  This is the first case where one has to use a rational function instead of a polynomial. 
\end{exa}

\subsection{Addition on superelliptic Jacobians, generalized Jacobi polynomials}

A natural question is the following:
\begin{prob}\label{jacobi-prob}
Is it possible to generalize the above procedure to a general curve?
\end{prob}
A complete answer to \cref{jacobi-prob} would be challenging for the very simple reason that in general we are not even able to write down an equation for a curve.  However,  from   \cref{sect-7} we know that we can write down precise equations for superelliptic curves.  Thus, the following question seems more reasonable:

\begin{prob}
Is it possible to generalize the Cantor's algorithm  to superelliptic curves?
\end{prob}

The main difference between hyperelliptic and superelliptic curves is that the hyperelliptic involution $\tau: (x, y) \to (x, -y)$ is now replaced by the order $n\geq 2$ automorphism $\a: (x, y) \to (x, \e_n \, y)$, where $\e_n$ is a primitive $n$-th root of unity. Hence, a naive extension of the Cantor's algorithm to superelliptic curves would be to determine the degrees of the $b(x)$ and $c(x)$ now for the curve $\X : y^n= f(x)$ so that the graph of $y= \frac {b(x)} {c(x)}$ intersect the curve $\C$ in exactly $g$ places $\bar R_1, \dots , \bar R_g$.   This is not possible.  Hence, we have to attempt a more general function, which is not necessary rational in one variable, in order to solve the problem.
A further discussion on this is intended in  \cite{kop-sh}.

There are a few issues that still are mysterious.  For example, there is a long and detailed discussion in \cite{Mum}, \cite{previato}, \cite{previato-87}. 
on what the Jacobi polynomials mean in terms of differential equations.  Do the corresponding \textbf{generalized Jacobi polynomials} for superelliptic curves have any significance in the theory of differential equations along the lines of \cite{previato}, et al. 

\begin{prob}
Investigate whether  Jacobi polynomials of hyperelliptic curves can be generalized to superelliptic curves and determine their significance in the theory of differential equations. 
\end{prob}
\subsection{Superelliptic Jacobians}

\newcommand\Ss{\mathcal{S}}
\def\AA{\mathcal A}
\def\BB{\mathcal B}
\newcommand\Fp{{\F}_p}
\newcommand\Qp{{\Q}_p}
\newcommand\Fq{\mathbb{F}_q}
\def\CC{\mathcal{C}}

A \textit{superelliptic Jacobian} is the Jacobian of a superelliptic curve. 
%
%
We assume that the reader is familiar with the basic definitions of Abelian varieties. For details one can check \cite{Mum} or \cite{frey-shaska}.

Let $\AA$, $\BB$ be abelian varieties over a  field $k$.  We denote the $\Z$-module of homomorphisms  $\AA \mapsto  \B$  by $\Hom( \AA, \B)$  and the ring of endomorphisms $\AA \mapsto \AA$ by $\End \AA$. In the context of Linear Algebra it can be more convenient  to   work with the $\Q$-vector spaces $\Hom^0 (\AA, \BB):= \Hom(\AA, \BB) \otimes_\Z \Q$, and $\End^0 \AA:= \End \AA\otimes_\Z \Q$. Determining $\End \AA$ or $\End^0 \AA$ is an interesting problem on its own; see  \cite{Oort}.


A homomorphism $f : \AA \to  \BB$ is called an \textbf{isogeny} if $\Img f = \BB$  and $\ker f$ is a finite group scheme. If an isogeny $\AA \to \BB$ exists we say that $\AA$ and $\BB$ are \textbf{isogenous}.  
%
The degree of an isogeny $f : \AA \to \B$ is the degree of the function field extension
\[ \deg f  := \left[k(\AA) : f^\star k(\B) \right].\]
It is equal to the order of the group scheme $\ker (f)$, which is, by definition, the scheme theoretical inverse image $f^{-1}(\{0_\AA\})$.

The group of $\bar{k}$-rational points has order 
\[ 
\#(\ker f)(\bar{k}) = [  k(\AA) : f^\star k(\BB)  ]^{sep},
\]
 where $[k(A) : f^\star k(\BB)]^{sep}$ denotes the degree of the maximally separable extension in $k(\AA)/ f^\star k(\BB)$.   $f$ is a \textbf{separable isogeny} if and only if 
\[ \# \ker f(\bar{k}) = \deg f.\]
%
%
The following result   should be compared with the well known result for quotient groups of abelian groups.

\begin{lem}\label{noether} 
For any Abelian variety $\AA/k$ there is a one to one correspondence between the finite subgroup schemes $\K \leq \AA$ and   isogenies $f : \AA \to \BB$, where $\BB$ is determined up to isomorphism.   Moreover, $\K = \ker f$ and $\BB = \AA/\K$. 
\end{lem}

 Isogenous Abelian varieties have  isomorphic endomorphism rings.

\begin{lem} 
If  $\AA$ and $\B$ are isogenous then $\End^0 (\AA) \iso \End^0 (\B)$. 
\end{lem}

\begin{lem}
If $\AA$ is a  absolutely simple Abelian variety then every endomorphism  not equal $0$ is an isogeny. 
\end{lem}

We can assume that $k=\bar{k}$. Let $f$ be a nonzero  isogeny   of $\AA$. Its kernel $\ker f$ is a subgroup scheme of $\AA$ (since it is closed in the Zariski topology because of continuity and under $\oplus$ because of homomorphism). It contains $0_\AA$ and so its connected component, which is, by definition, an Abelian variety. 

Since $\AA$ is simple and $f\neq 0$ this component is equal to $\{0_\AA\}$. But it has finite index in $\ker f$ (Noether property) and so $\ker f$ is a finite group scheme.

The ring of endomorphisms of generic Abelian varieties is "as small as  possible". For instance, if $\chara (k)  =0$, then $\End(\AA)=\Z$ in general. If $k$ is a finite field, the Frobenius endomorphism  will generate a larger ring, but again, in the generic case. Determining endomorphism rings of superelliptic Jacobians is an interesting problem. 
A concrete result is  the following \cite{zarhin-1}:
\begin{thm}[Zarhin]\label{zarhin-thm}
Let $\CC$ be a hyperelliptic curve with affine equation $y^2=f(x)$, $n=\deg f$, and $f \in \Q[x]$. If $\Gal (f) $ is isomorphic to $A_n$ or $S_n$  then $\End_{\overline \Q} \, (\Jac \CC ) \iso \Z$.
\end{thm}

The theorem is actually true over any number field $K$. See \cite{zarhin-2} for detailed results on endomorphisms of Jacobians of hyperelliptic and superelliptic curves. 
It is an interesting task to find Abelian varieties with larger endomorphism rings. This leads to the theory of real and complex multiplication.  For instance, the endomorphism ring of the Jacobian of the Klein quartic contains an order in a totally real field of degree $3$ over $\Q$.

An abelian variety $\cA/k$  is said to have \textbf{complex multiplication} over $k$ if $\End_k^0 (\cA)$ is larger than $\Z$. Normally we say that an  Abelian variety with complex multiplication by CM; see next section for more details.

\subsection{Jacobians of genus 2 curves}

For $\chara k \neq 2$, a point $\p $ in the moduli space $\M_2$ is determined by the tuple $(J_2, J_4, J_6, J_{10})$, for discriminant $D:=J_{10} \neq 0$. In the case of $\chara k =2$ another invariant $J_8$ is needed. 

  For every $D:=J_{10} > 0$ there is a Humbert hypersurface $H_D$ in $\M_2$ which parametrizes curves $\CC$ whose Jacobians admit an optimal action on $\O_D$; see \cite{HM95}.   Points on $H_{n^2}$ parametrize curves whose Jacobian admits an $(n, n)$-isogeny to a product of two elliptic curves.

  For every quaternion ring $R$  there are irreducible curves $S_{R, 1}$, $\dots$, $S_{R, s}$ in $\M_2$ that parametrize curves whose Jacobians admit an optimal action of $R$. Those $S_{R, 1}$, $\dots$ , $S_{R, s}$ are called  \textbf{Shimura curves}.


We have the following:

\begin{prop}   
$\Jac (\CC)$ is a geometrically simple Abelian variety if and only if it is not $(n, n)$-decomposable for some $n>1$. 
\end{prop}

The endomorphism rings of Abelian surfaces can be determined by the Albert's classification and results in \cite{Oort}. We summarize in the following:   

\begin{prop}  
The endomorphism ring $\End_{\overline \Q}^0  \, (\Jac \CC )$ of an abelian surface is either 
$\Q$, a real quadratic field, a CM-field of degree 4, a non-split quaternion algebra over $\Q$, $F_1 \oplus F_2$,  where each $F_i$ is either $\Q$ or an imaginary quadratic field, the Mumford-Tate group $F$,  where $F$ is either $\Q$ or an imaginary quadratic field.   
\end{prop}

\begin{rem}  Genus 2 curves with extra involutions have endomorphism ring larger than $\Z$. 
Let $\CC$ be a   genus 2 curve defined over $\Q$. If $\Aut (\CC)$ is isomorphic to the Klein 4-group $V_4$, then  $\CC$ is isomorphic to a curve $\CC^\prime$ with equation 
\[ y^2 = f(x)= x^6 - a x^4+b x^2-1.\]
We denote  $u=a^3+b^3$ and $v=ab$. The    discriminant   
\[\Delta_f = - 2^6 \cdot \left(  27-18v+4u-u^2   \right)^2,\]  
is not a complete square in $\Q$ for any values of $a, b \in \Q$.  In this case $\Gal_{\Q} (f) $ has order 24.  There is a twist of this curve, namely $y^2=f(x)=x^6+a^\prime x^4+ b^\prime x^2 +1$, in which case $\Delta_f$ is a complete square in $\Q$ and $\Gal_{\Q} (f) $ has order 48. In both cases, from \cref{zarhin-thm} we have that   $\End_{\overline \Q} (\Jac \CC^\prime ) \neq  \Z$.
\end{rem}


Next, we turn our attention to determining the endomorphism ring of abelian surfaces. Let us first recall a few facts on characteristic polynomials of Frobenius for abelian surfaces. The Weil $q$-polynomial arising in genus 2 have the form 
\begin{equation}\label{weil-g-2}
 f (T ) = T^4 - aT^3 + (b+2q) T^2 - aq T + q^2,
\end{equation}
for $a, b \in \Z$ satisfying the inequalities 
\[ 2 |a| \sqrt{q} - 4q \leq b \leq \frac 1 4 a^2 \leq 4q .\]
We follow the terminology from  \cite{bhls}.  Let $\CC$ be a curve of genus 2 over $\F_q$ and $\J=\Jac \CC$.  Let $f$ be the Weil polynomial of $J$ in \cref{weil-g-2}. We have that $\# \CC (\F_q)=q+1-a$, $\#J (\F_q)=f(1)$ and it lies in the genus-2 Hasse interval 
\[ \H_q^{(2)} = \left[  (\sqrt{q}-1)^4, (\sqrt{q}+1)^4    \right] \,. \]
In \cite{bhls} are constructed decomposable $(3, 3)$-Jacobians with a given number of rational points by glueing two elliptic curves together.

Next we briefly summarize some of the results obtained in \cite{lombardo} for $\End_K (\AA)$ in terms of the characteristic polynomial of the Frobenius. We let $K$ be a number field and $M_K$ be the set of norms of $K$.   Let $\AA$ be an abelian surface defined over $K$  and $f_v$ the characteristic Frobenius for every norm $v \in M_K$. 


\begin{lem}
Let $v$ be a place of characteristic $p$ such that $\AA$ has good reduction.  Then $\AA_v$ is ordinary if and only if the characteristic polynomial of the Frobenius 
\[ f_v (x) = x^4 + a x^3 + b x^2 + ap x + p^2, \]
satisfies $b \not\equiv 0 \mod p$.   
\end{lem}

Then from   \cite{lombardo}*{Lemma~4.3}   we have  the following.

\begin{lem}
Let $\AA$ be an absolutely simple abelian surface. The endomorphism algebra $\End_{\bar K}^0 (\AA)$ is non-commutative (thus a division quaternion algebra) if and only if for every $v \in M_K$, the polynomial $f_v (x^{12})$ is a square in $\Z[x]$. 
\end{lem}

The following gives a condition for geometrically reducible abelian surfaces. 


\begin{prop}[\cite{lombardo}]
i) If $\AA/K$ is geometrically reducible then for all $v\in M_k$ for which $\AA$ has good reduction the polynomial $f_v (x^{12})$ is reducible in $\Z[x]$. 

ii) If $\CC$ is a  smooth, irreducible genus 2 curve with affine equation $y^2=f(x)$  such that $f(x) \in K[x]$ is an irreducible polynomial of degree 5   then $\Jac \CC$ is absolutely irreducible. 
\end{prop}

\subsection{Decomposition of superelliptic Jacobians}

Let $\X$ be a superelliptic curve and $\s \in \Aut (\X_g)$ such that its projection $\bs \in \bAut (\X)$ has order  $m\geq 2$ and equation $y^n=f(x)$.  We can choose a coordinate in $\P^1$ such that $\bs (x) = x^m$. Since $\s$ permutes the Weierstrass points of $\X$ and it has two fixed points then the equation of the curve will be $y^n= f(x^m)$ or $y^n=x f(x^m)$; see \cite{b-sh-sh} for details of this part. 

Assume that $\X$ has equation 
\begin{equation}\label{first_eq_orig}
 y^n = f(x^m) := x^{\d m } + a_1 x^{(\d-1) m}  \dots + a_{\d-1} x^m + 1 .
\end{equation} 
We assume that $\bs$ lifts to $G$ to an element of order $m$.   Then, $\s( x, y) \to (\e_m x, y)$.  Denote by $\t: (x, y) \to (x, \e_n y)$ its superelliptic automorphism. Since $\t$ is central in $G$ then $\t\s=\s\t$.    We will denote by $\X_1$ and $\X_2$ the quotient curves $\X/\<\s\>$ and $\X/\< \t \s\>$ respectively.  The next theorem determines the equations of $\X_1$ and $\X_2$.  We denote by $K$ the function field of $\X$ and by  $F$ and $L$ the function fields of $\X_1$ and $\X_2$ respectively. 

\begin{thm}\label{thm_1}
Let $K$ be a genus $g\geq 2$  level $n$ superelliptic field and $F$ a  degree $m$ subfield fixed by $\s: (s, y) \to (\e_m s, y)$. 

i) Then,  $K=k(x, y)$ such that 
\begin{equation}\label{first_eq}
 y^n = f(x^m) := x^{\d m } + a_1 x^{(\d-1) m}  \dots + a_{\d-1} x^m + 1 .
\end{equation} 
for    $\D (f, x)  \neq 0$.

ii) $F = k(U, V) $ where $U=x^m$,     $V=y$  and 
\begin{equation}\label{ell_curve}
 V^n= f(U).
 \end{equation}

iii) There is another subfield $L= k(u, v) $ where  $u=x^m$,   $v=x^i y$,  and 
\begin{equation}\label{gen_2}
 v^n = u \cdot f(u),
 \end{equation}
for $m=\l n$ and $i=\l (n-1)$. 
\end{thm}

\xymatrixrowsep{7ex}
\xymatrixcolsep{5ex}
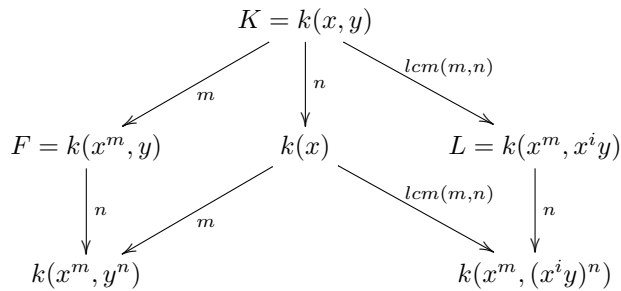
\begin{figure}[hp] 
\[
\xymatrix{ 
  & K=k(x, y) \ar@{->}[d]^{ n}  \ar@{->}[ld]^{  m}  \ar@{->}[rd]^{  lcm(m, n)} \\
F=k(x^m, y) \ar@{->}[d]^{n}         &        k(x)  \ar@{->}[ld]^{ m}  \ar@{->}[rd]^{lcm(m, n)}  & L=k(x^m, x^i y)   \ar@{->}[d]^{n}   \\
k(x^m, y^n)       &     &    k( x^m, (x^i y)^n)  }
\]
\caption{Lattice of subfields}
   \label{super-fig}
\end{figure}

\medskip

For the rest of this section we want to find necessary and sufficient conditions on $n$ and $m$ such that the Jacobian $\Jac  (\X)$ is isogenous to the product $\Jac  (\X_1) \times \Jac  (\X_2)$. First we focus on  hyperelliptic curves. 


\begin{thm} \label{lem_2}
Let $\X_g$ be a hyperelliptic curve. We denote its reduced automorphism group by $ \bAut (\X_g) \iso  C_m = \<\sigma \>$. Then $\X_g$ is isomorphic to a curve with equation
\[\X_g: Y^2=x^{\d  m} + a_1 x^{(\d-1)  m} + \dots + a_{\d-1} x^m + 1.\]
There exists subcovers $\pi_i : \X_g \to \X_i$, for $i=1, 2$ such that 
 \[ 
 \begin{split}
&  \X_1 : \quad Y^2= X^{\d} + a_1 X^{\d-1} + \dots + a_{\d-1} X + 1, \\
& \X_2 : \quad Y^2=X(  X^{\d} + a_1 X^{\d-1} + \dots + a_{\d-1} X + 1  ). \\
\end{split}
\] 
The  Jacobian of $\X$ is isogenous to the product
\[ \Jac  (\X) \iso  \Jac  (\X_1)   \times   \Jac  (\X_2) \]
%
if and only if the full automorphism group $\Aut (\X)$ is isomorphic to the Klein 4-group $V_4$. 
\end{thm}

\noindent Next, we generalize the previous theorem.

\begin{thm}\label{thm:superelliptic}
 Let $\X_g$ be a level $n$ superelliptic curve and $C_m = \< \bar \s \> \embd \bAut (\X_g)$, where  $m \geq 2$ and  the equation of $\X_g$ is  $ y^n = f(x^m)$,
with $\deg (f) = d = \d\, m$,  $d > n$.  Then there exist degree $m$ coverings $\pi : \X_g \to \X_i$, $i=1, 2$  where 
\[  \X_1 : \quad y^n=f(x) \quad and  \quad \X_2 : \quad  y^n= x f(x).  \]
Then,  
\[ \Jac  (\X) \iso \Jac  (\X_1 )  \times \Jac  ( \X_2 ) \] 
if and only if 
\begin{equation}\label{eq_m}
  \d (n-1) (m-2) = 1 -   \left(\gcd(\d+1, n) + \gcd(\d, n) - \gcd(\d m, n)    \right).
\end{equation} 
\end{thm}

\proof
 Let $\X_g$ be a superelliptic curve with and extra automorphism of order $m \geq 2$ and equation $y^n = f(x^m)$.  There is the superelliptic automorphism 
\[ \tau :   \, \,   (x, y) \to (x, \e_n y), \quad and \quad \bar \sigma :   \, \,  (x, y) \to (\e_m x, y) .  \]
   We denote by $\sigma$ the lifting of $\bar \sigma$ in $\Aut (\X)$.   Then, $\sigma\tau=\tau\sigma$.   

Let  $H_1 : =\<  \sigma \> $ and  $H_2 : = \< \s \tau \>$ be subgroups in $G$.  Then, $| H_1 | = n$ and $ | H_2 | = \lcm (n, m)$.    Thus, we have $H:=H_1 \times H_2 \embd G$. It is easy to check that $g \left(\X_g / (H_1 H_2) \right)=0$.

Moreover,    $\sigma $ and $\sigma \tau$ fix the curves 
 \[ \X_1 : \quad Y^n= X^{\d} + a_1 X^{\d-1} + \dots + a_{\d-1} X + 1\]
and \[ \X_2 : \quad Y^n=X(  X^{\d} + a_1 X^{\d-1} + \dots + a_{\d-1} X + 1  ).\] 
Let $g_1$ and $g_2$ denoted their genera respectively.  Then
\[ g_1 = 1 + \frac 1 2 \left( n \d - n - \d - \gcd(\d, n)   \right)  \]
and 
\[ g_2 = 1 + \frac 1 2 \left( n (\d+1) - n - (\d + 1)  - \gcd(\d + 1, n)   \right).  \]
Hence,  we have 
\[ g_1 + g_2 =  \frac 3 2 + n\d - \frac n 2 - \d - \frac 1 2 \left( \gcd( \d, n) + \gcd(\d+1, n) \right). \]
The genus of $\X$ is 
 \[ g = 1 + \frac 1 2 \left(  n \d m - n - \d m - \gcd(m \d, n)  \right). \]
Then, $g=g_1+g_2$ implies that 
\[  \d  (n-1) (m-2) = 1 -   \left(\gcd(\d+1, n) + \gcd(\d, n) - \gcd(\d m, n)    \right). \]
Thus,   
\[ \Jac  (X_g) \iso \Jac  (\X / H_1) \times \Jac  (\X / H_2) \]
which completes the proof. 
\endproof

\subsection{Jacobians with superelliptic components}
Next let us consider a family of non-hyperelliptic curves  whose Jacobians decompose into factors which are  superelliptic Jacobians.  
In \cite{Ya}  were studied a family of curves  in $\P^{s+2}$ given by the equations 
\begin{equation}\label{curve}
\left\{
\begin{aligned}
zw & = c_0x^2+c_1xw+c_2w^2 \\
y_1^r & = h_1 (z, w) := z^r+ c_{1, 1} z^{r-1} w + \cdots + c_{r-1, 1} zw^{r-1} + w^r,\\
 & \dots  \\
y_s^r & = h_s (z, w) :=  z^r+ c_{1, s} z^{r-1} w + \cdots + c_{r-1, s} zw^{r-1} + w^r,  \\
\end{aligned}
\right.
\end{equation}
where  $c_i \in k$, $i=0, 1, 2,$ and   $c_{i,j} \in k$ for $i=1, \cdots, r$, $ j=1, \cdots , s$. 
The variety $\X_{r, s}$ is an algebraic curve since the function field of $\X_{r, s}$ is a finite extension of $k(z)$. $\X_{r, s}$ is a complete intersection. 


Let $\X_{r, s}$ be as above. Assume that $\X_{r, s}$ is smooth and $c_0\neq 0$. Then     the genus of  $\X_{r, s}$ is  
\[ g(\X_{r, s}) = (r-1)(rs\cdot2^{s-1}-2^s+1).\] 
$r\geq3$ and $s\geq1$, then $\X_{r, s}$ is non-hyperelliptic.


Fix $r\geq 2$.  Let $\l$ be an integer such that $1  \leq \l\leq s$.    Define the superelliptic curve $C_{r, \l, m}$ as follows
\[ C_{r, \l, m}:  \qquad Y^r=\prod_{i=1}^{\l} h_{i}(X^m, 1),\]
for some $m\geq 2$.  The right side of the above equation has degree $d = rm \l$. Using  Lemma~\ref{lem_1} we have that 
\[ g( C_{r, \l, m} ) = 1 + \frac 1 2 \left(  r^2 m\l - r- m\l r - \gcd(\l r m, r)\right). \]
Hence,
\begin{equation} \label{eq_g}
g( C_{r, \l, m} ) = 1 + \frac r 2 \left((r-1)\l m -2  \right).
\end{equation}

In  \cite{b-sh-sh} automorphism groups of such curves were determined. We have  
 $\bAut \left( C_{r, \l, m} \right) \iso C_m$ or $\bAut \left( C_{2, \l, m} \right) \iso D_{2m}$.
%
%
From \cref{th14} we can now determine  the  automorphism group as follows.

 If    $\bAut \left( C_{r, \l, m} \right) \iso C_m$, then $G \cong C_{mn}$ or   $G$ is isomorphic to 
\begin{center}
$\left\langle \r, \s \right|\r^n=1,\s^m=1,\s\r\s^{-1}=\r^l \rangle$
\end{center}
where (l,n)=1 and $l^m\equiv 1$ (mod n). But if $(m,n)=1$, then $l=n-1$.  

If    $\bAut \left( C_{r, \l, m} \right) \iso D_{2m} $, then  
\begin{itemize}
\item[(1)] If n is odd then $G \cong D_{2m} \times C_n$.

\item[(2)] If n is even and $m$ is odd then $G \cong D_{2m} \times C_n$ or $G$ is isomorphic to the group with presentation 
\[ \left\langle \r, \s, \t \right|\r^n=1,\s^2=\r,\t^2=\r^{n-1},(\s\t)^m=\r^{\frac{n}{2}},\s\r\s^{-1}=\r,\t\r\t^{-1}=\r \rangle. \]

\item[(3)] If n is even and m is even then $G $ is isomorphic to one of the following groups $D_{2m} \times C_n$, $D_{2mn}$, or one of the following 
\begin{align*}
\begin{split}
G_1=& \left\langle \r, \s, \t \right|\r^n=1,\s^2=\r,\t^2=1,(\s\t)^m=1,\s\r\s^{-1}=\r,\t\r\t^{-1}=\r^{n-1} \rangle,\\
G_2=& \left\langle \r, \s, \t \right|\r^n=1,\s^2=\r,\t^2=\r^{n-1},(\s\t)^m=1,\s\r\s^{-1}=\r,\t\r\t^{-1}=\r \rangle,\\
G_3=& \left\langle \r, \s, \t \right|\r^n=1,\s^2=\r,\t^2=1,(\s\t)^m=\r^{\frac{n}{2}},\s\r\s^{-1}=\r,\t\r\t^{-1}=\r^{n-1} \rangle, \\
G_4=& \left\langle \r, \s, \t \right|\r^n=1,\s^2=\r,\t^2=\r^{n-1},(\s\t)^m=\r^{\frac{n}{2}},\s\r\s^{-1}=\r,\t\r\t^{-1}=\r \rangle.
\end{split}
\end{align*}
\end{itemize} 

Let  $\X_{r, s}$ be a generic algebraic  curve  defined over an algebraically closed field  $k$   and $C_{r, \l, m}$ as above. Then we have the following.

\begin{thm}[\cite{b-sh-sh}]
The Jacobian $\Jac  (\X_{r, s})$ is isogenous to the product of the $C_{r, \l, m}$, for $1 \leq \l \leq s$, namely 
\[  \Jac (\X_{r, s}) \iso  \prod_{1 = \l}^s   \Jac (C_{r, \l, m}),   \] 
if and only if 
\begin{equation} r=4\cdot\frac{1+s-2^s}{ms(s+1)-s\cdot2^{s+1}}\label{num-thy-eqn}.
\end{equation}
\end{thm}

\proof
We denote by $\s_i (x, y_i, z) \to (x, \e_r y_i, z)$, for $i=1, \dots , s$.  Then the quotient spaces $\X_{r, s} / \< \s_i \>$ are  the curves $C_{r, i, s}$, for  $i=1, \dots , s$.  Since $\s_i$ is a central element in $G = \Aut (\X_{r, s} )$ then $H_i := \< \s_i \> \normal G$, for all $i=1, \dots , s$.  Obviously, for all $i \neq j$ we have $H_i \cap H_j = \{ e \}$.  Hence, $H_1, \dots , H_s$ forms a partition for $G$.

The genus for every $C_{r, i, s}$, by Lemma~\ref{lem_1} is given by  \cref{eq_g}.   Then we have 
\[
\begin{split}
\sum_{\l=1}^s \, g \left( C_{r, \l, m} \right)  & = \sum_{\l=1}^s \left( 1 + \frac r 2 \left((r-1)\l m -2  \right)  \right)  \\
& = s(r-1) \left( \frac r 4 m (s+1) -1    \right).   \\
\end{split}
 \]
%
Then we have that 
\[  \frac r 4  m s (s+1) - s   =  rs \cdot 2^{s-1} - 2^s +1.  \]
Hence, 
\[ r = 4 \cdot \frac {1+s-2^s} {ms(s+1)- s \cdot 2^{s+1}}.\]
This completes the proof.
\qed

\begin{rem}
For $m=2$ this result is the case of Theorem~4.2 in \cite{Ya}. We get  $r= \frac  2 s$. Hence, $s=1$ or $s=2$. 
Therefore,    Theorem~4.2 in \cite{Ya} is true only for curves $F_{m, 1}$ or $F_{m, 2}$. 
\end{rem}


Suppose $r,m,s\in\mathbb{N}$ satisfy Eq.~\ref{num-thy-eqn}.  Then $mrs=4k$ for some odd integer $k$. Moreover, 

i) If $s\equiv1\text{ (\mod 2)}$, then $s=1$.

ii) If $s\equiv2\text{ (\mod 4)}$, then $s=2t$ for some odd integer $t$ which satisfies $4^t\equiv1\text{ (\mod }t)$.  Furthermore, $t$ is a multiple of 3.

iii) If $s\equiv0\text{ (\mod 4)}$, then $s=4u$ for some odd integer $u$ which satisfies $16^u\equiv1\text{ (\mod }u)$.  Furthermore, $u$ is a multiple of 3 or 5.
%

\section{Jacobians with complex multiplication}\label{sect-14}

We start with some preliminaries.  An \textbf{Abelian variety} defined over $k$ is an absolutely irreducible projective variety defined over $k$ which is  a group scheme.
A morphism of  Abelian varieties  $\AA$ to   $\B$ is a \textbf{homomorphism} if and only if it maps the identity element of $\AA$ to the identity element of $\B$.  An abelian variety   $\AA/k$ is called \textbf{simple} if it has no   proper non-zero Abelian subvariety over $k$,  it is called \textbf{absolutely simple} (or \textbf{geometrically simple}) if it   is simple over the algebraic closure of $k$. 
 
Let $\AA$, $\B$ be abelian varieties over a  field $k$.  We denote the $\Z$-module of homomorphisms  $\AA \mapsto  \B$  by $\Hom( \AA, \B)$  and the ring of endomorphisms $\AA \mapsto \AA$ by $\End \AA$.  It turns out to be more convenient  to   work with the $\Q$-vector spaces $\Hom^0 (\AA, \B):= \Hom(\AA, \B) \otimes_\Z \Q$, and $ \End^0 \AA:= \End \AA\otimes_\Z \Q$.   Determining $\End \AA$ or $\End^0 \AA$ is an interesting problem on its own; see  \cite{Oort}.

The ring of endomorphisms of generic Abelian varieties is "as small as  possible". For instance, if $\chara (k)=0$ then  $\End(\AA)=\Z$ in general.  If $k$ is a finite field, the Frobenius endomorphism  will generate a larger ring, but again, this will be all in the generic case. 


$\End^0(\AA)$ is a  $\Q$-algebra of dimension $\leq 4\dim (\AA)^2$. Indeed, $\End^0 (\AA)$ is a semi-simple algebra, and by duality  one can apply  a complete classification due to Albert of \emph{possible} algebra structures on $\End^0(\AA)$, which can be found on \cite{Mum}*{pg. 202}. 

We say that an abelian variety $\AA$ has \textbf{complex multiplication} over a field $K$ if the algebra $\End_K^0 (\AA)$ contains a commutative, semisimple $\Q$-algebra of dimension $2\dim \AA$. 

The natural question is which algebras occur as endomorphism algebras? The situation is well understood if $k$ has characteristic $0$ (due to Albert)  but wide open in characteristic $p>0$. 


For $g=1$ (elliptic curves) everything is explicitly known due to M. Deuring. The endomorphism ring of an elliptic curve over a finite field $\Fq$ is never equal to $\Z$ since there is the  Frobenius endomorphism $\phi_{\Fq,\E}$ induced by the Frobenius automorphism of $\Fq$ which has degree $q$. 


Let $\CC$ be a genus 2 curve defined over $k$. What can we say about the $\End_{k}^0  \, (\Jac \CC )$?

\begin{prop}  
Given a genus-two curve $\CC$ defined over $\Q$ and its abelian surface $\Jac \CC$, the endomorphism ring $\End_{\overline \Q}^0  \, (\Jac \CC )$ is either 
$\Q$, a real quadratic field, a CM field of degree 4, a non-split quaternion algebra over $\Q$, $F_1 \oplus F_2$,  where each $F_i$ is either $\Q$ or an imaginary quadratic field, the Mumford-Tate group $F$,  where $F$ is either $\Q$ or an imaginary quadratic field.   
\end{prop}

\begin{rem}  Genus 2 curves with extra involutions have endomorphism ring larger than $\Z$. 
Let $\CC$ be a   genus 2 curve defined over $\Q$. If $\Aut (\CC)$ is isomorphic to the Klein 4-group $V_4$, then  $\CC$ is isomorphic to a curve $\CC^\prime$ with equation 
\[ y^2 = f(x)= x^6 - a x^4+b x^2-1.\]
We denote  $u=a^3+b^3$ and $v=ab$. The    discriminant   
\[\Delta_f = - 2^6 \cdot \left(  27-18v+4u-u^2   \right)^2,\]  
is not a complete square in $\Q$ for any values of $a, b \in \Q$.  In this case $\Gal_{\Q} (f) $ has order 24.  There is a twist of this curve, namely $y^2=f(x)=x^6+a^\prime x^4+ b^\prime x^2 +1$, in which case $\Delta_f$ is a complete square in $\Q$ and $\Gal_{\Q} (f) $ has order 48. In both cases, from \ref{zarhin-thm} we have that   $\End_{\overline \Q} (\Jac \CC^\prime ) \neq  \Z$.
\end{rem}


The following are proved in \cite{zarhin-2}.
 
\begin{thm}
Let $K$ be a field, $\chara K \neq 2$ and $f(x) \in K[x]$ an irreducible polynomial with $\deg f \geq 5$.    If one of the following conditions is satisfied:

\begin{itemize}
\item $\chara K \neq 3$ and $\Gal_K (f) \iso A_n$ or $S_n$ 

\item $\Gal_K (f) \iso M_n$ (Mathiew group) for $n= 11, 12, 22, 23, 24$    
\end{itemize}
then the curve $C: y^2=f(x)$ has $\End J = \Z$. In particular, $\Jac C $ is absolutely simple. 
\end{thm}  

\begin{thm} 

If $f(x)$ is as above, $\chara K =0$, and $p$ an odd prime then the superelliptic curve
$ \CC : y^p = f(x)$ 
has $\Jac (\CC)$ absolutely simple and   $\End (\Jac C) \iso \Z[\varepsilon_p]$.
\end{thm}

\subsection{Curves with many automorphisms}
Let $\CC$ be e genus $g\geq 2$ curve defined over $\C$,  $\p \in \M_g$ its corresponding moduli point, and $G:=\Aut_\C (\CC)$. 

We say that $\CC$ has \textbf{many automorphisms} if $\p \in \M_g$ has a neighborhood $U$ (in the complex topology) such that all curves corresponding to points in $U \setminus \{\p \}$ have automorphism group strictly smaller than $\p$.

\begin{lem} 
The following are equivalent:

\begin{itemize}

\item  $\CC$ has many automorphisms 

\item  There exists a subgroup $H < G$ such that $g \left(  \CC/H \right) = 0$ and $\CC \to \CC/H$ has at most 3 branch points.  

\item  The quotient $\CC/G$ has genus 0 and $\CC \to \CC/G$ has at most three points.  

\end{itemize}

\end{lem}

\begin{question}[F. Oort]
If $\CC$ has many automorphisms, does $\End (\Jac \CC)$ have complex multiplication? 
\end{question}

Wolfart answered this question for all curves of genus $g \leq 4$. For the remainder of this paper we will determine which superelliptic curves of genus $g \geq 10$ have CM.  

Wolfart answered this question for all curves of genus $g \leq 4$. We now determine all superelliptic curves with many automorphisms with genus $5 \leq g \leq 10$. 
The automorphism groups of superelliptic curves, the ramification structure of $\CC \to \CC/G$, and the moduli dimension of each family are determined in \cite{Sa}*{Table~1} for every characteristic $p>5$.

\begin{cor}
A curve $\CC$ with automorphism group $G$ and signature $\sigma$ has many automorphisms if and only if $g (\CC/G ) =0$ and the moduli dimension of the Hurwitz space $\mathcal H (g, G, \sigma)$ is 0.
\end{cor}
See \cite{kyoto} on details the moduli dimension. 

\begin{lem}
Superelliptic curves of genus $5 \leq g \leq 10$ which are not hyperelliptic and with many automorphisms are presented in \cref{many-aut}. 
\end{lem}

\proof
From \cite{Sa}*{Table~1} we picked all cases such that $\delta =0$.  These cases are exactly superelliptic curves with many automorphisms.  Since the hyperelliptic curves with many automorphisms and CM were already studied in \cite{muller-pink}, we delete the cases for which $n=2$.  The rest of the cases are presented below.
\endproof

\begin{prob}\label{problem-equation}
Determine which curves from \cref{many-aut} have Jacobians with complex multiplication.
\end{prob}

Our goal is to determine which of the curves in the above table have CM. We have a  first simple criteria.

\begin{lem}
Let  $\CC$ be an algebraic curve and $\psi : \CC \to \E$ a degree $n$ covering to an elliptic curve.  If the $j$-invariant  $j(\E)$ is not an algebraic integer then $\Jac (\CC)$ does not have CM.
\end{lem}

Moreover, a formula for $\mbox{Sym}^2 \chi$  similar to the one in \cite{muller-pink} can be possibly obtained for superelliptic curves by using the a basis for the space of holomorphic
differentials on $\X$ is given  in \cref{super-basis}.     
A complete discussion of this problem is intended in \cite{obus-shaska}.


\begin{footnotesize}

\begin{table}[hb]
\begin{tabular}{|l|l|l|l|l|l|l|l|}
\hline
Nr. & $\bar G$             & G&$n$  &$m$ & sig. & $\delta$ & Equation $y^n=f(x)$ \\
\hline \hline
\multicolumn{8}{c}{Genus 5} \\
\hline \hline
 2&   \multirow{2}{*}{$C_m$}    & $C_{22}$  &  2  & 11 &   11, 22   & 0   & $x^{11}+1$ \\ 
 2&                        		& $C_{22}$  &  11 & 2  &   2, 22   &  0    & $x^2+1$ \\ 
\hline
5&    \multirow{1}{*}{$D_{2m}$}  &  &  2 & 12 & 2, 4, 12     &  0 & $x^{12}-1$  \\ 
 8&                           &  &  2 & 10 & 2, 4, 20     &  0 & $x (x^{10} -1)$ \\   
\hline   
20 &   $S_4$                     &  & 2  & 0 & 3, $4^2$ & 0 & $x^{12} - 33x^8-33x^4+1$ \\  
\hline   
25 &   $A_5$                     &  & 2 &  & 2,3,10 & 0 & $x (x^{10} + 11x^5-1)$ \\  
\hline \hline
\multicolumn{8}{c}{Genus 6} \\
\hline \hline
 2&     \multirow{2}{*}{$C_m$}                     &   $C_{26}$        &  2  & 13  & $13, 26$ & 0 & $x^{13} + 1$ \\
 2&                       &   $C_{21}$        &  3  & 7  & $7, 21$ & 0  & $x^7 +1$ \\
 2&                        & $C_{20}$  &  4  & 5 &   $5, 20$   & 0   & $x^{5}+1$ \\
 2&                        & $C_{20}$     &  5  & 4  &  $4, 20$   & 0 & $x^4+1$ \\
 2&                        & $C_{21}$     &  7  & 3  & $3, 21$    & 0 & $x^3 +1$ \\
 2&                        & $C_{26}$     &  13 & 2  & $2, 26$    & 0 & $x^2 +1$ \\
\hline
 5&   \multirow{4}{*}{$D_{2m}$}                         & $G_5$ &  2 & 14  & $2, 4, 14$     &  0 & $x^{14}- 1$ \\
 5&                           & $ D_{10}\times C_2$ &  5 & 5  & $2, 5, 10$ &  0 & $x^5-1$  \\
 8&                           & $G_8$ &  2 & 12 & 2, 4, 24     &  0 & $x (x^{12} -1)$ \\
 8&                           & $ D_{12}\times C_3$ &  3 & 6 & 2, 6, 18     &  0 & $x (x^{6} -1)$ \\
 8&                           &$G_8$  &  4 & 4 & 2, 8, 16     &  0 & $x (x^{4} -1)$ \\
 8&                           & $ D_{6}\times C_5$ &  5 &3 & 2, 10, 15     &  0 & $x (x^{3} -1)$ \\
 8&                           &$ D_{4}\times C_7$  &  7 & 2 & 2, $14^2$     &  0 & $x (x^{2} -1)$ \\
\hline
  18&   $S_4$                     & $ G_{18}$ & 4  & 0 & 2, 3, 16 & 0 & $x(x^4-1)$ \\
 19&                             & $ G_{19}$ & 2  & 0 & 2, 6, 8 & 0 & $x(x^4-1)(x^8+14x^4+1)$ \\
\hline
\hline
\end{tabular}
\end{table}


\addtocounter{table}{-1}
\begin{table}[ht]
\caption{(Cont.)}
\scalebox{.9}{
\begin{tabular}{|l|l|l|l|l|l|l|l|}
\hline
Nr. & $\bar G$             & G&$n$  &$m$ & sig. & $\delta$ & Equation $y^n=f(x)$ \\
\hline \hline

\multicolumn{8}{c}{Genus 7} \\
\hline \hline
 2&      \multirow{2}{*}{$C_m$}                   & $C_{30}$     &  2  & 15  &  15, 30   & 0 & $x^{15} +1$ \\
 2&                        & $C_{24}$     &  3  & 8  & 8, 24         &   0 & $x^8+1$ \\
 2&                        & $C_{30}$     &  15  & 2  & 2, 30         &   0 & $x^2+1$ \\
\hline
 5&    \multirow{3}{*}{$D_{2m}$}                       &$G_5$  &  2 & 16  & 2, 4, 16 &  0 & $x^{16}-1$  \\
 5&                           & $D_{18}\times C_3$ &  3 & 9 & 2, 6, 9  &  0 & $x^9-1$  \\
 8&                           & $G_8$ &  2 & 14 & 2, 4, 28     &  0 & $x (x^{14} -1)$ \\
 8&                           & $D_{14}\times C_3$ &  3 & 7 & 2, 6, 21     &  0 & $x (x^{7} -1)$ \\
 8&                           & $G_8$ &  8 & 2 & 2,${16}^2$     &  0 & $x (x^{2} -1)$ \\
\hline \hline
\multicolumn{8}{c}{Genus 8} \\
\hline \hline
 2&    \multirow{1}{*}{$C_m$}                    & $C_{34}$  &  2  & 17 &   17, 34   & 0   & $x^{17}+1$ \\
 2&                        & $C_{34}$  &  17 & 2  &   2, 34   &  0    & $x^2+1$ \\
\hline
 5&  \multirow{2}{*}{$D_{2m}$}                         &$G_5$  &  2 & 18 & 2, 4, 18     &  0 & $x^{18}-1$  \\
 8&                           &$G_8$  &  2 & 16  & 2, 4, 32 &  0 & $x(x^{16}-1)$ \\
\hline
22 &   $S_4$                     &$G_{22}$  & 2  & 0 & 3, 4, 8 & 0 & $x(x^4-1)(x^{12}-33x^8-33x^4+1)$ \\
\hline \hline
\multicolumn{8}{c}{Genus 9} \\
\hline \hline
 2&   \multirow{4}{*}{$C_m$}                     & $C_{38}$  &  2  & 19 &   19, 38   & 0   & $x^{19}+1$ \\
 2&                        & $C_{30}$  &  3 & 10  &  10, 30  &  0    & $x^{10} +1$ \\
 2&                        & $C_{28}$  &  4 & 7  &  7, 28   &  0    & $x^7 +1$ \\
  2&                        & $C_{28}$  &  7 & 4  &  4, 28   & 0    & $x^4 +1$ \\
  2&                        & $C_{30}$  &  10 & 3  &  3, 30    &  0      & $x^3 +1$ \\
  2&                        & $C_{38}$  &  19 & 2  &   2, 38   &  0      & $x^2  +1$ \\
\hline 
 5&   \multirow{4}{*}{$D_{2m}$}                        &$G_5$  &  2 & 20 & 2, 4, 20     &  0 & $x^{20}-1$  \\
 5&                           &$G_5$  &  4 & 8 & 2,$8^2$     &  0 & $x^{8}-1$  \\
 8&                           & $G_8$ &  2 & 18 & 2, 4, 36     &  0 & $x (x^{18} -1)$ \\
 8&                           & $D_{18}\times C_3$ &  3 & 9 & 2, 6, 27     & 0 & $x (x^{9}-1)$ \\
   8&                           & $G_8$ &  4 & 6 & 2, 8, 24     &  0 & $x (x^{6} -1)$ \\
  8&                           & $D_6\times C_7$ &  7 & 3 & 2, 14, 21     &  0 & $x (x^{3} -1)$ \\
   8&                           & $G_8$ &  10 & 2 & 2, $20^2$     &  0 & $x (x^{2} -1)$ \\
\hline
 17&   $S_4$                     &$G_{17}$  & 4  & 0 & 2, 4, 12 & 0 & $x^{8}+14x^4+1$ \\
 21&                             &$G_{21}$  & 2  & 0 & $4^2$ 6 & 0 & $(x^{8}+14x^4+1)(x^{12}-33x^8-33x^4+1)$ \\
\hline
27 &   $A_5$                     &  & 2 &  & 2, 5, 6 & 0 & $x^{20}-228x^{15}+494x^{10}+228x^5+1$ \\
\hline \hline
\multicolumn{8}{c}{Genus 10} \\
\hline \hline
 2&    \multirow{4}{*}{$C_m$}                    & $C_{42}$  & 2 & 21  &  21, 42  &  0    & $x^{21} +1$ \\
 2&                        & $C_{33}$  &  3 & 11  &  11, 33   &  0    & $x^{11} +1$ \\
 2&                        & $C_{30}$  &  5 & 6  &  6, 30    &  0      & $x^6 +1$ \\
 2&                        & $C_{30}$  &  6 & 5  &   5, 30   &  0      & $x^5  +1$ \\
 2&                        & $C_{33}$     &  11  & 3  &  3, 33  & 0 & $x^3 +1$ \\
 2&                        & $C_{42}$     &  21  & 2  & 2, 42       &   0 & $x^2+1$ \\
\hline
 5&                           &$G_5$  &  2 & 22 & 2, 4, 22     &  0 & $x^{22}-1$  \\
 5&                           &$D_{24}\times C_3$  &  3 & 12  & 2, 6, 12   &  0& $x^{12}-1$ \\
 5&                           &$G_5$  &  6 & 6 &    2, 6, 12     &  0 & $x^6-1$  \\
 8&                           & $G_8$ &  2 & 20 & 2, 4, 40    &  0 & $x (x^{20} -1)$ \\
 8&                           & $D_{20}\times C_3$ &  3 & 10 & 2, 6, 30   & 0 & $x (x^{10} -1)$ \\
 8&                           & $D_{10}\times C_5$ &  5 & 5 & 2, 10, 25   & 0 & $x (x^{5}-1)$ \\
 8&                           & $G_8$ &  6 & 4 & 2, 12, 24     &  0 & $x (x^{4} -1)$ \\
 8&                           & $D_4\times C_{11}$ &  11 & 2 & 2, $22^2$     &  0 & $x (x^{2} -1)$ \\
\hline
 18&   $S_4$                     &$G_{18}$  & 6  & 0 & 2, 3, 24 & 0 & $x(x^4-1)$ \\
 20&                             &$S_4\times C_3$  & 3  & 0 & 3, 4, 6 & 0 & $x^{12}-33x^8-33x^4+1$ \\
\hline
25 &   $A_5$                     &$A_5\times C_3$  & 3 & 0 & 2, 3, 15 & 0 & $x(x^{10}+11x^5-1)$ \\
\hline
\end{tabular}}
\vspace{5mm}
\caption{Superelliptic curves for genus $5 \leq g \leq 10$}
\label{many-aut}
\end{table}

\end{footnotesize}

\clearpage

\section{A word on Abelian covers and further directions}

The story obviously doesn't end with superelliptic Jacobians. What is the natural way of extending the study of algebraic curves and their Jacobians? There have been many attempts to study coverings where the monodromy group is more general that a cyclic group.  The next natural groups would be dihedral groups; see \cite{fried-95}.  Another class of coverings (curves) would be the coverings when the monodromy groups is an Abelian group.  Below we briefly suggest two classes of curves which seem the natural extension of problems presented in this paper. 


\subsection{Curves with separated variables}

There is a special class of algebraic curves satisfying an equation of the form $f(x)-g(z)=0 $, where $f,g$ are polynomials with coefficients in $k$. They were first introduced by  by Fried and Macrae  in the wonderful paper \cite{fried-macrae}.  They showed   that 

(a) $f_1(x ) - g_1(z)$ divides $f(x)- g(z)$ if and only if there exists a polynomial $F$ such that $f(t)=F(f_1(t))$, $g(t)=F(g_1(t))$. 

(b) $f(x)-g(z)$ is said to be a minimal separation for $a(x,z)$ if $a(x,z)$ divides $f(x)-g(z)$ and if whenever $a(x,z)$ divides $F(x)-G(z)$  then $f(x)-g(z) $  divides $F(x) - G(z)$.

The polynomial $a(x,z)$ possesses a minimal separation if and only if there is a polynomial $F(x)-G(z)$ in $k[x,z] $  such that $a(x,z)$ divides $F(x)-G(x)$. Most of these results depend on a lemma giving a necessary and sufficient condition for an element $z\in k(x)$ to lie in $k[x]$.  The automorphism group of such curves is a degree $m$ central extension of $\Gal_k (f(x))$, where $n:= | \Gal_k (g(y))|$.  As far as we are aware, nobody has studied in detail automorphism groups of such curves.  Clearly all superelliptic curves are special classes of such curves. 

\begin{prob}
For a given genus $g\geq 2$ list all groups which occur as automorphism groups of curves with separable variables. For each group determine parametric equations of the corresponding family of curves. 
\end{prob}

For more interesting ramifications to this class of curves check \cite{fried-99}, \cite{fried-2012} and \cite{fried-gusic}. 
\subsection{Abelian covers}

Consider a curve $\X$ such that it has a covering $\pi : \X \to \P^1$ which has monodromy group an Abelian group.     Then this monodromy group  is a direct product of cyclic groups.  In this case the theory of  cyclic covers can be used to study such Abelian covers.  We simplify the setup by considering only Galois coverings.  Hence, the following setup.

Let $\X$ be an algebraic curve defined over $k$ such that   $G \embd \Aut (\X)$ is an Abelian group.  Let $G\iso G_1 \times \cdots \times G_r$ be the decomposition of $G$ into cyclic groups.  If  one of     $\X/H_i$ is a genus zero quotient space,  then  the equation of $\X$ is a superelliptic curve.  Suppose non of the $H_i$ fix a zero genus quotient.  Then we check all quotient groups ${\overline G}_i:=G/H_i$. Since $G$ is Abelian, these quotient groups act on the curves as well.  If one of these groups ${\overline G}_i$ fixes a genus zero quotient then again we are in the superelliptic case.  

Is $G$ has no subgroup which fixes a genus 0 field, then we consider all quotients $\X_i:=\X/H_i$.  They have smaller genii, therefore more manageable automorphism groups (which are also Abelian.  going down the lattice of the corresponding function fields we should be able to determine the equation of each quotient curves and therefore the equation of $\X$. 
Thus, for a given $g\geq 2$,   we have a way of   determining equation of all curves $\X$ such that $\Aut (\X)$ is an Abelian group.  To the best of our knowledge, this has not been pursued systematically  for  $g\geq 4$.
\bibliographystyle{amsplain} 
\bibliography{references}{}

\end{document}